\documentclass[12pt,reqno]{report}

\usepackage{pgfplots}
\pgfplotsset{compat=1.13}
\usepackage{float}
\usepackage{tikz}
\usepackage{amsmath}
\usepackage{a4wide}
\usepackage{xfrac}
\usepackage{stmaryrd,mathrsfs,bm,amsthm,mathtools,yfonts,amssymb,color}
\usepackage{xcolor}

\newcommand\C{{\mathbb C}}
\newcommand\R{{\mathbb R}}
\newcommand\PP{{\mathbb P}}

\begingroup
\newtheorem{theorem}{Theorem}[section]
\newtheorem{lemma}[theorem]{Lemma}
\newtheorem{proposition}[theorem]{Proposition}
\newtheorem{corollary}[theorem]{Corollary}

\endgroup

\theoremstyle{definition}

\newtheorem{definition}[theorem]{Definition}
\newtheorem{remark}[theorem]{Remark}
\newtheorem{ipotesi}[theorem]{Assumption}

\title{The masterpieces of John Forbes Nash Jr.}
\author{Camillo De Lellis\\
delellis@math.uzh.ch\\ \\
Institut f\"ur Mathematik\\ 
Universit\"at Z\"urich\\}

\begin{document}

\maketitle

\tableofcontents

\chapter*{Introduction}

John Nash has written very few papers: if for each mathematician in the 20th century we were to divide the depth, originality, and impact of the corresponding production by the number of works, he would most likely be on top of the list, and even more so if we were to divide by the number of pages. In fact all his fundamental contributions can be stated in very few lines: certainly another measure of his genius, but making any survey of his theorems utterly useless.
Discussing the impact of Nash's work is certainly redundant, since all his fundamental contributions have already generated a large literature and an impressive number of surveys and lecture notes.  ``Reworking'' his proofs in my own way, or giving my personal perspective, would be of little interest: much better mathematicians have already developed deep and well-known theories from his seminal papers.

\medskip

When I wask asked to write this contribution to the Abel Volumes I felt enormously honored, but precisely for the reasons listed above it took me very little to realize how difficult it would have been to write something even modestly useful.
This note is therefore slightly unusual: I have just tried to rewrite the original papers in a more modern language while adhering as much as possible to the original arguments. In fact Nash used often a rather personal notation and wrote in a very informal way, here and there a few repetitions can be avoided and the discussions of some, nowadays standard, facts can be removed.  In a sense my role has been simply that of a translator: I just hope to have been a decent one, namely that I have not introduced (too many) errors and wrong interpretations. In particular I hope that these notes might save some time to those scholars who want to work out the details of Nash's original papers, although I strongly encourage anybody to read the source: any translation of any masterpiece always loses something compared to the original and the works of Nash are true masterpieces of the mathematics of the 20th century!

\medskip

These notes leave aside Nash's celebrated PhD thesis on game theory and focus on the remaining four fundamental papers that have started an equal number of revolutions in their respective topics, namely the 1952 note on real algebraic varieties, the 1954 paper on $C^1$ isometric embeddings, the 1956 subsequent work on smooth isometric embeddings and finally the 1958 H\"older continuity theorem for solutions to linear (uniformly) parabolic partial differential equations with bounded nonconstant coefficients. Even the casual reader will realize that everything can be understood up to the smallest detail with a very limited amount of knowledge: I dare say that any good graduate student in mathematics will be able to go through the most relevant arguments with little effort. 

\medskip

I have decided to leave aside the remaining works of Nash in ``pure mathematics'' either because their impact has not been as striking as that of the four mentioned above (as it is the
case for the works \cite{Nash1955,Nash1962,Nash1966}) or because, as it is the case for \cite{Nash1996}, although its impact has been major, this is mainly due to the questions raised by Nash rather than to the actual theorems proved by him. However, for completeness I have included a last chapter with a brief discussion of these remaining four (short!) notes in pure mathematics and of the ``Nash blowup''. 

\section*{Acknowledgements}

I am very grateful to Helge and Ragni for entrusting to me this portion of the Nash volume, a wonderful occasion to deepen my understanding of the mathematics of a true genius, who has had a tremendous influence in my own work.

Most of the manuscript has been written while I was visiting the CMSA at Harvard and I wish to thank Shing-Tung Yau and the staff at CMSA for giving me the opportunity to carry on my work in such a stimulating environment. 

Several friends and colleagues have offered me kind and invaluable help with various portions of this note. In particular I wish to thank
Davide Vittone for giving me several precious suggestions with the Chapters 2 and 3 and reading very carefully all the manuscript;
Gabriele Di Cerbo, Riccardo Ghiloni and J\'anos Kollar for clarifying several important points concerning Chapter 1 and pointing out
a few embarassing mistakes; Tommaso de Fernex and J\'anos Kollar for kindly reviewing a first rather approximate version of Section 5.4;
Eduard Feireisl for his suggestions on Section 5.3; Cedric Villani for allowing me to steal a couple of paragraphs from his beautiful review of \cite{Nash1958} in the Nash memorial article \cite{NashMemorial}; Francois Costantino for helping me with a delicate topological issue;
Jonas Hirsch and Govind Menon for proofreading several portions of the manuscript; Helge Holden for going through all the manuscript with
extreme care. 

\medskip

This work has been supported by the grant agreement 154903 of the Swiss National Foundation.  

\chapter{Real algebraic manifolds}

\section{Introduction}

After his famous PhD thesis in game theory (and a few companion notes on the topic) Nash directed his attention to geometry and specifically to the classical
problem of embedding smooth manifolds in the Euclidean space.\footnote{In a short autobiographical note, cf.~\cite[Ch.~2]{EssentialNash}, Nash states that he made his
important discovery while completing his PhD at Princeton. In his own words ``\ldots I was fortunate enough, besides developing the idea which led to ``NonCooperative Games'', also to make a nice discovery relating manifolds and real algebraic varieties. So, I was prepared actually for the possibility that the game theory work would not be regarded as acceptable as
a thesis in the mathematics department and then that I could realize the objective of a Ph.D. thesis with the other results.''} Consider a smooth closed manifold $\Sigma$ of dimension $n$ (where with {\em closed} we mean, as usual,
that $\Sigma$ is compact and has no boundary). A famous theorem of Whitney (cf.~\cite{Whitney1936,Whitney1944}) shows that $\Sigma$ can be embedded smoothly in $\mathbb R^{2n}$, namely that there
exists a smooth map $w: \Sigma \to \mathbb R^{2n}$ whose differential has full rank at every point (i.e., $w$ is an {\em immersion}) and which is injective (implying therefore
that $w$ is an homeomorphism of $\Sigma$ with $w (\Sigma)$). 

Clearly $w (\Sigma)$ is a smooth submanifold of $\mathbb R^{2n}$ diffeomorphic to $\Sigma$. Whitney showed also that $w$ can be perturbed smoothly to a second embedding $v$ so that
$v (\Sigma)$ is a {\em real analytic} submanifold, namely for every $p\in v (\Sigma)$ there is a neighborhood $U$ of $p$ and a real analytic map $u: U \to \mathbb R^n$ such that
$\{u =0\} = U \cap v (\Sigma)$ and $Du$ has full rank. Whitney's theorem implies, in particular, that any closed smooth manifold $\Sigma$ can be given a real analytic structure, namely an atlas $\mathcal{A}$ of charts where the changes of coordinates between pairs of charts are real analytic mappings. 

In his only note on the subject, the famous groundbreaking paper \cite{Nash1952} published in 1952, Nash gave a fundamental contribution to real algebraic geometry, showing that indeed it is possible to realize any smooth closed manifold of dimension $n$ as an {\em algebraic} submanifold of
$\mathbb R^{2n+1}$. We recall that, classically, any subset of $\mathbb R^N$ consisting of the common zeros of a collection of polynomial equations is called an {\em algebraic subvariety}. We can assign a dimension to any algebraic subvariety using a purely algebraic concept (see below) and the resulting number coincides with the usual metric definitions of
dimension for a subset of the Euclidean space (for instance with the Hausdorff dimension, see \cite[Ch.~2]{EG} for the relevant definition).
The main theorem of Nash's note is then the following.

\begin{theorem}[Existence of real algebraic structures]\label{t:alg_main}
For any closed connected smooth $n$-dimensional manifold $\Sigma$ there is a smooth embedding $v: \Sigma \to \mathbb R^{2n+1}$ such that $v (\Sigma)$ is a connected component of an $n$-dimensional
algebraic subvariety of $\mathbb R^{2n+1}$. 
\end{theorem}

It turns out that for any point $p\in v (\Sigma)$ there is a neighborhood $U$ such that $U\cap v (\Sigma)$ is the zero set of $n+1$ polynomials with linearly independent gradients.  
In his note Nash proved also the following approximation statement, see Theorem~\ref{t:alg_approx}: any smooth embedding $w: \Sigma \to \mathbb R^m$ can be smoothly approximated by an embedding $\bar v$ so that $\bar v (\Sigma)$ is a portion of an $n$-dimensional algebraic subvariety of $\mathbb R^{m}$.
However, in order to achieve the stronger property in Theorem~\ref{t:alg_main}, namely that $\bar v (\Sigma)$ is a {\em connected component} of the subvariety, Nash's argument needs to increase  the target. 
He conjectured that this is not necessary, cf.~\cite[p.~420]{Nash1952}, a fact which was proved much later by Akbulut and King, see \cite{AK}. He also conjectured the existence of a smooth embedding
$z$ (in {\em some} Euclidean space $\mathbb R^N$) such that $z (\Sigma)$ is the whole algebraic subvariety, not merely a connected component, and this was proved by Tognoli in 
\cite{Tognoli}. Both \cite{Tognoli} and \cite{AK} build upon a previous work of Wallace, \cite{Wallace}.

As it happens for the real analytic theorem of Whitney, it follows from Theorem~\ref{t:alg_main} that any smooth closed manifold can be given a real algebraic structure, see below for the precise definition. In his note Nash proved also that such structure is indeed unique, cf.~Theorem~\ref{t:alg_uniq}. 

\medskip

As already mentioned in the previous paragraph, Nash left a few conjectures and open questions in his paper, which were subsequently resolved through the works of Wallace, Tognoli, and Akbulut and King: we refer the reader to King's paragraph in Nash's memorial article \cite{NashMemorial} for further details. The ideas of his paper have generated a large body of literature in real algebraic geometry and terms like Nash manifolds, Nash functions, and Nash rings are commonly used to describe some of the objects arising from his argument for Theorem~\ref{t:alg_main}, see for instance \cite{BCR,Shiota}.

\section{Real algebraic structures and main statements}

Following Nash we introduce a suitable algebraic structure on closed real analytic manifolds $\Sigma$. In \cite{Nash1952} such structures are called {\em real algebraic manifolds}. Since however nowadays  the latter expression is used for a different object, in order to avoid confusion and to be consistent with the current terminology, we will actually use the
term ``Nash manifolds'' for the objects introduced by Nash.

Note that, by the classical Whitney's theorem recalled in the previous section, there is no loss of generality in assuming
the existence of a real analytic atlas for any differentiable manifold $\Sigma$.
The notion of Nash manifold allows Nash to recast Theorem~\ref{t:alg_main} in an equivalent form. The latter
will be given in this section, together with several other interesting conclusions, whose proofs will all be postponed to the next sections. 

\begin{definition}[Basic sets]\label{d:basic}
Any finite collection $\{f_1, \ldots , f_N\}$ of smooth real valued functions over $\Sigma$ is called a {\em basic set}
if the map $f = (f_1, \ldots, f_N)$ is an embedding of $\Sigma$ into $\mathbb R^N$. 
\end{definition}

\begin{definition}[Nash manifolds]\label{d:alg_manifold}
A {\em Nash manifold} is given by a pair $(\Sigma, \mathscr{R})$ where $\Sigma$ is a real analytic manifold
of dimension $n$ and $\mathscr{R}$ a ring of real valued functions over $\Sigma$ satisfying the following requirements:
\begin{itemize}
\item[(a)] Any $f\in \mathscr{R}$ is real analytic;
\item[(b)] $\mathscr{R}$ contains a basic set;
\item[(c)] The transcendence degree of $\mathscr{R}$ must be $n$, more precisely for any collection of $n+1$ distinct elements $f_1, \ldots, f_{n+1}\in \mathscr{R}$ there is a nontrivial polynomial $P$ in $n+1$ variables such that $P (f_1, \ldots , f_{n+1})=0$;
\item[(d)] $\mathscr{R}$ is maximal in the class of rings satisfying (a), (b), (c).
\end{itemize}
\end{definition}

An important (and not difficult) fact following from the definitions is that the algebraic structure of the ring determines
in a suitable sense the manifold $\Sigma$ and hence that the structure as Nash manifold is essentially
unique for every $\Sigma$. 

\begin{proposition}[Algebraic description of Nash manifolds]\label{p:alg_uniq}
On any Nash manifold $(\Sigma, \mathscr{R})$ there is a one-to-one correspondence between maximal ideals of
$\mathscr{R}$ and points of $\Sigma$, more precisely:
\begin{itemize}
\item[(I)] $\mathscr{I} \subset \mathscr{R}$ is a maximal ideal if and only if 
$\mathscr{I} = \{f\in \mathscr{R}: f (p) = 0\}$ for some $p\in \Sigma$.
\end{itemize}
Moreover, if $(\Sigma_1, \mathscr{R}_1)$
and $(\Sigma_2, \mathscr{R}_2)$ are two Nash manifolds, then a map $\phi: \mathscr{R}_1 \to \mathscr{R}_2$ is
a ring isomorphism if and only if there is a real analytic diffeomorphism $\varphi: \Sigma_1 \to \Sigma_2$ such that
$\phi (f) = f \circ \varphi^{-1}$ for any $f \in \mathscr{R}_1$. 
\end{proposition}

Consider now a Nash manifold $(\Sigma, \mathscr{R})$ and recall that by Definition~\ref{d:alg_manifold}(b) we prescribe the
existence of a basic set $\mathscr{B} = \{f_1, \ldots , f_N\}\subset \mathscr{R}$: it follows that $f = (f_1, \ldots, f_N)$ is an
analytic embedding of $\Sigma$ into $\mathbb R^N$. On the other hand by Definition~\ref{d:alg_manifold}(c) there is
a set of nontrivial polynomial relations between the $f_i$'s (because $N>n$) and so it appears naturally that $f (\Sigma)$ is in
fact a subset of a real algebraic variety. Following Nash we will call $f (\Sigma)$ a {\em representation} of the corresponding
Nash manifold. 

\begin{definition}[Representations]\label{d:represent}
If $(\Sigma, \mathscr{R})$ is a Nash manifold, $\mathscr{B} = \{f_1, \ldots , f_N\} \subset \mathscr{R}$ a basic set and $f= (f_1, \ldots , f_N): \Sigma \to \mathbb R^N$, then $f (\Sigma)$ is 
called an {\em algebraic representation} of $(\Sigma, \mathscr{R})$. 
\end{definition}

In order to relate representations with algebraic subvarieties of the Euclidean space we need to introduce the concept
of {\em sheets} of an algebraic subvariety. 

\begin{definition}[Sheets]\label{d:sheets}
A sheet of a real algebraic subvariety $V\subset \mathbb R^N$ is a subset $S\subset V$ satisfying the following
requirements:
\begin{itemize}
\item[(a)] For any $p,q\in S$ there is a real analytic arc $\gamma:[0,1]\to S$ with $\gamma (0)=q$ and $\gamma (1) =p$. 
\item[(b)] $S$ is a maximal subset of $V$ with property (a).
\item[(c)] There is at least one point $p\in S$ with a neighborhood $U$ such that $U\cap V \subset S$. 
\end{itemize}
\end{definition}

Clearly, if $V\subset \mathbb R^N$ is an algebraic subvariety and $S\subset V$ a connected component which happens to be a
submanifold of $\mathbb R^N$, then $S$ is in fact a sheet of $A$. However:
\begin{itemize}
\item[(i)] there might be ``smooth'' sheets which go across singularities, for instance,
if we take $V = \{(x,y) : xy=0\}\subset \mathbb R^2$ and $S = \{(x,y): x=0\}$, cf.~Figure \ref{fig:1};
\item[(ii)] there might be sheets which are connected components of $V$ but are singular, for instance
Bernoulli's lemniscate $\{(x,y): (x^2+y^2)^2 = 2 y^2 - 2y^2\}$ is a connected subvariety of the plane
consisting of one single sheet, singular at the origin.
\end{itemize}

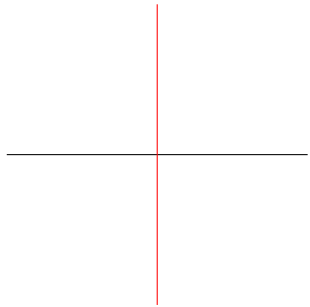
\begin{figure}[H]
\centering
\begin{tikzpicture}
\draw (-2,0) -- (2,0);
\draw [color = red] (0,2) -- (0,-2);
\end{tikzpicture}
\caption{The set $S = \{(x,y)\in \mathbb R^2:x=0\}$ is a sheet of the algebraic subvariety $V = \{(x,y): xy=0\}$. Note that, although the origin is a singular point of $V$, it is not a singular point of $S$. Moreover $S$ is {\em not} a connected component of $V$.}\label{fig:1}
\end{figure}

\begin{figure}[H]
\centering
\begin{tikzpicture}
\begin{axis}[hide axis]
\addplot[domain=-1.41:-0.1, samples=100] {sqrt(sqrt(4*x*x+1)-x*x-1)};
\addplot[domain=-0.1:0.1, samples=100] {sqrt(sqrt(4*x*x+1)-x*x-1)};
\addplot[domain=0.1:1.41, samples=100] {sqrt(sqrt(4*x*x+1)-x*x-1)};

\addplot[domain=-1.414:-1.41, samples=100] {sqrt(sqrt(4*x*x+1)-x*x-1)};
\addplot[domain=-1.4141:-1.414, samples=50] {sqrt(sqrt(4*x*x+1)-x*x-1)};
\draw[line width = 0.2mm] ({-sqrt(2)},0)--({-sqrt(2)},0.01);

\addplot[domain=1.41:1.414, samples=100] {sqrt(sqrt(4*x*x+1)-x*x-1)};
\addplot[domain=1.414:1.4141, samples=50] {sqrt(sqrt(4*x*x+1)-x*x-1)};
\draw[line width = 0.2mm] ({sqrt(2)},0)--({sqrt(2)},0.01);


\addplot[domain=-1.41:-0.1, samples=100] {-sqrt(sqrt(4*x*x+1)-x*x-1)};
\addplot[domain=-0.1:0.1, samples=100] {-sqrt(sqrt(4*x*x+1)-x*x-1)};
\addplot[domain=0.1:1.41, samples=100] {-sqrt(sqrt(4*x*x+1)-x*x-1)};

\addplot[domain=-1.41:-1.414, samples=100] {-sqrt(sqrt(4*x*x+1)-x*x-1)};
\addplot[domain=-1.4141:-1.414, samples=50] {-sqrt(sqrt(4*x*x+1)-x*x-1)};
\draw[line width = 0.2mm] ({-sqrt(2)},0) -- ({-sqrt(2)},-0.01);

\addplot[domain=1.41:1.414, samples=100] {-sqrt(sqrt(4*x*x+1)-x*x-1)};
\addplot[domain=1.414:1.4141, samples=50] {-sqrt(sqrt(4*x*x+1)-x*x-1)};
\draw[line width = 0.2mm] ({sqrt(2)},0)--({sqrt(2)},-0.01);


\end{axis}
\end{tikzpicture}
\caption{Bernoulli's lemniscate is an algebraic subvariety of $\mathbb R^2$ which consists of a single sheet. Note that it is singular at the origin.}\label{fig:2}
\end{figure}
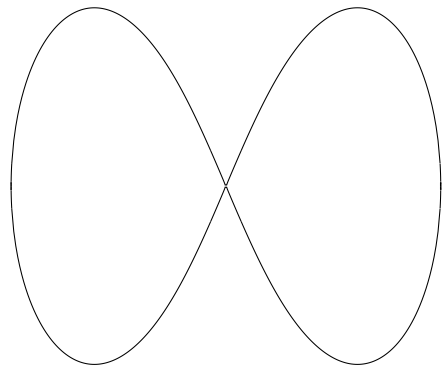

An important observation by Nash is that, by simple considerations, any representation of a Nash manifold is in fact a sheet of an irreducible algebraic subvariety with
dimension equal to that of the manifold. Recall that an algebraic subvariety $V$ is called {\em irreducible} if it cannot be written as the
union of two proper subsets which are also subvarieties. More precisely we have

\begin{proposition}[Characterization of representations]\label{p:rep=var}
A representation of a connected Nash manifold $(\Sigma, \mathscr{R})$ is always a sheet of an irreducible subvariety 
$V$ whose dimension is the same as that of $\Sigma$. Conversely, if $v: \Sigma \to \mathbb R^N$ is a real analytic embedding
of a closed real analytic manifold $\Sigma$ whose image $v (\Sigma)$ is a sheet of an algebraic subvariety, then there is a structure 
of Nash manifold $(\Sigma, \mathscr{R})$ 
for which the components $\{v_1, \ldots , v_N\}$ of $v$ form a basic subset of $\mathscr{R}$. 
\end{proposition}

The outcome of the discussion above is that Theorem~\ref{t:alg_main} can now be equivalently stated in terms of
Nash manifolds. However note that Theorem~\ref{t:alg_main} requires the representation to be more than just a sheet
of an algebraic subvariety: it really has to be a connected component. For this reason Nash introduces a special term:
a representation $v (\Sigma)$ will be called {\em proper} if it is a {\em connected component} of the corresponding algebraic subvariety in
Proposition~\ref{p:rep=var}.  Hence we can now rephrase Theorem~\ref{t:alg_main} in the following way.

\begin{theorem}[Existence of proper representations]\label{t:alg_main_2}
For any connected smooth closed differentiable manifold $\Sigma$ of dimension $n$ there is a structure of Nash manifold
$(\Sigma, \mathscr{R})$ with a basic set $\{v_1, \ldots , v_{2n+1}\} = \mathscr{B}\subset \mathscr{R}$ such that
$v (\Sigma)$ is a proper representation in $\mathbb R^{2n+1}$.
\end{theorem}

Giving up the stronger requirement of ``properness'' of the representation, Nash is able to provide an approximation
with algebraic representations of any smooth embedding, without increasing the dimension of the ambient space. As
a matter of fact Theorem~\ref{t:alg_main_2} will be proved as a corollary of such an approximation theorem, whose statement goes as follows.

\begin{theorem}[Approximation theorem]\label{t:alg_approx}
Let $\Sigma$ be a connected closed differentiable manifold and $w: \Sigma \to \mathbb R^m$ a smooth embedding. Then for any
$\varepsilon >0$ and any $k\in \mathbb N$ there is a structure of Nash manifold $(\Sigma, \mathscr{R})$ with 
a basic set $\{v_1, \ldots , v_m\}$ such that $\|w-v\|_{C^k}< \varepsilon$.
\end{theorem}

In the theorem above $\|\cdot\|_{C^k}$ denotes a suitably defined norm measuring the uniform distance between derivatives of
$w$ and $v$ up to order $k$. The norm will be defined after fixing a finite smooth atlas on $\Sigma$, we refer to the
corresponding section for the details. 

As a final corollary of his considerations, Nash also reaches the conclusion that the structure of Nash manifold is in
fact determined uniquely by the differentiable one. More precisely we have the following result.

\begin{theorem}[Uniqueness of the Nash ring]\label{t:alg_uniq}
If two connected Nash manifolds $(\Sigma_1, \mathscr{R}_1)$ and $(\Sigma_2, \mathscr{R}_2)$ are diffeomorphic as 
differentiable manifolds, then they are also isomorphic as Nash manifolds, namely there is a real analytic $\varphi: \Sigma_1
\to \Sigma_2$ for which the map $\phi (f):= f\circ \varphi^{-1}$ is a ring isomorphism of $\mathscr{R}_1$ with $\mathscr{R}_2$. 
\end{theorem}

\section{Technical preliminaries}

In this section we collect some algebraic and analytical technical preliminaries, standard facts which will be used in the
proofs of the statements contained in the previous sections. We begin with a series of basic algebraic properties.

\begin{definition}\label{d:field_of_def}
Given an algebraic subvariety $V\subset \mathbb R^N$ and a subfield ${\mathbb F}\subset {\mathbb R}$ we say that
${\mathbb F}$ is a {\em field of definition} of $V$ if there is a set $S$ of polynomials with coefficients in ${\mathbb F}$ such
that $V = \{x\in \mathbb R^N: P(x) =0\, , \;\forall P\in S\}$.
\end{definition}

\begin{proposition}[{Cf. \cite[Cor.~3, p.~73, and Prop.~5, p.~76]{Weil}}]\label{p:F1}
For any algebraic subvariety $V\subset \mathbb R^N$ there is a unique minimal field ${\mathbb F}\subset \mathbb R$ of definition,
namely a field of definition of $V$ which does not contain any smaller field of definition. ${\mathbb F}$ is, moreover, finitely generated over
$\mathbb Q$. 
\end{proposition}

\begin{definition}\label{d:dim}
We will say that a certain collection of coordinates $\{x_{i_1}, \ldots , x_{i_m}\}$ is algebraically independent over a field $\mathbb F$
at a point $p= (p_1, \ldots , p_N)$ if there is no nontrivial polynomial $P$ with coefficients in $\mathbb F$ such that $P (p_{i_1}, \ldots , p_{i_m}) = 0$. 

Given a point $p$ in an algebraic subvariety $V\subset \mathbb R^n$ with minimal field of definition $\mathbb F$ we define the algebraic dimension ${\rm dim}_V (p)$  of $p$ with respect to $V$ as the maximal number of coordinates which are algebraically independent over $\mathbb F$ at $p$. The algebraic dimension of $V$ is ${\rm dim}\, (V) = \max \{{\rm dim}_V (p): p \in V\}$ and a point $p\in V$ is called a {\em general point} of $V$ if ${\rm dim}_V (p) = {\rm dim}\, (V)$. 
\end{definition}

\begin{proposition}\label{p:F3-6}
Let $V\subset \mathbb R^N$ be an algebraic subvariety of algebraic dimension $n$ with minimal field of definition $\mathbb F$.
Then the following holds.
\begin{itemize}
\item[(a)] Any collection of $n+1$ coordinates satisfy a nontrivial polynomial relation (as real functions with domain $V$);
\item[(b)] For any general point $p$ of $V$ there is a neighborhood $U$ where $V$ is an $n$-dimensional (real analytic) submanifold  and where any collection of coordinates which are algebraically independent at $p$ over $\mathbb F$ gives a (real analytic) parametrization of $V$.
\item[(c)] If ${\rm dim}_V (p) =m$, then there is an algebraic subvariety $W\subset V$ of algebraic dimension $m$ which contains $p$ and  whose minimal field of definition is contained in $\mathbb F$. 
\end{itemize}
\end{proposition}

The proofs of the statements (a), (b), and (c) can be found in \cite[Ch.~II and Ch.~IV]{Weil}, more precisely see the discussion at \cite[p.~72]{Weil}
and \cite[Th.~3]{Weil}.  

We state here a simple corollary of the above proposition, for which we give the elementary proof.

\begin{corollary}\label{r:Haus=alg}
The algebraic dimension of a subvariety $V$ coincides with its Hausdorff dimension as a subset of $\mathbb R^N$. In fact,
for any $j\leq {\rm dim}\, (V)$, the subset $V_j := \{v\in V: {\rm dim}_V (p) = j\}$ is a set of Hausdorff dimension $j$.
\end{corollary}
\begin{proof} The second part of the statement obviously implies the first. We focus therefore on the second, which we 
prove by induction over ${\rm dim}\, (V)$. The $0$-dimensional case is obvious: if $V$ is a $0$-dimensional subvariety of $\mathbb R^N$, then $V_0=V$ must be contained in $\mathbb F^N$, which is necessarily a countable set ($\mathbb F$ denotes the minimal field of definition of $V$ and recall that it is finitely generated over $\mathbb Q$). 

Assume therefore that the statement holds 
when the dimension of the variety is no larger than $n-1$: we now want to show that  the claim holds when ${\rm dim}\, (V) = n$.
By Proposition~\ref{p:F3-6}, the subset $V_n$ of points $p\in V$ with maximal algebraic dimension is covered by countably many real analytic $n$-dimensional manifolds. Hence $V_n$ has Hausdorff dimension at most $n$ (cf.~\cite[Sec.~3.3]{EG}). On the other hand by Proposition~\ref{p:F3-6}(b) the Hausdorff dimension must be at least $n$. 
Next, let $j<n$. By Proposition~\ref{p:F3-6}(c) any point
$p\in V_j$ is contained in an algebraic subvariety $W$ of algebraic dimension $j$ with minimal field of definition contained in $\mathbb F$. Each such $W$ has Hausdorff dimension $j$, by inductive assumption. On the other hand, since any such $W$ is defined through a finite set of polynomials with coefficient in $\mathbb F$, the set of such $W$ is countable. We have therefore shown that $V_j$ has Hausdorff dimension at most $j$. 

Now consider a point $q\in V$ with ${\rm dim}_V (q) =j$ and an algebraic subvariety $W\subset V$ as above. Let $\mathbb F'$ be its minimal field of definition and consider any $p = (p_1, \ldots , p_N) \in W$. The algebraic dimension ${\rm dim}_W (p)$ is at most $j$, which means that for any collection of $j+1$ distinct coordinates $p_{i_1}, \ldots , p_{i_{j+1}}$ there is a nontrivial polynomial $P$ with coefficients in $\mathbb F'$ such that $P (p_{i_1}, \ldots , p_{i_{j+1}})=0$. Since $\mathbb F' \subset \mathbb F$, we must necessarily have ${\rm dim}_V (p) \leq j$. Thus,
$W\subset V_0\cup V_1\cup \ldots \cup V_j$. On the other hand, we know by inductive assumption that $W$ has Hausdorff dimension $j$ and we have shown that the dimension of each $V_i$ is at most $i$. We then conclude that the Hausdorff dimension of $j$ must be $j$. 
\end{proof}

We are now ready to state the two technical facts in analysis needed in the rest of the chapter. The first is a standard consequence
of the implicit function theorem for real analytic mappings, see for instance \cite[Th.~1.8.3]{KP}). As usual, the tubular neighborhood of size $\delta$ of a subset $S\subset \mathbb R^N$ is the open set consisting of those points whose distance from $S$ is smaller than $\delta$. In this chapter we will denote it by $U_\delta (S)$. 

\begin {proposition}\label{p:A1}
If $\Sigma\subset \mathbb R^N$ is a closed real analytic submanifold, then there is a $\delta > 0$ with the following two properties:
\begin{itemize}
\item[(a)] For any $x\in U_\delta (\Sigma)$ there is a unique point $u(x)\in \Sigma$ of least distance to $x$.
\item[(b)] The map $x\mapsto u (x)$ is real analytic. 
\end{itemize}
\end{proposition}

The first statement needs in fact only the $C^2$ regularity of $\Sigma$, cf.~\cite{Hirsch}. Moreover the proof therein uses the implicit function theorem to give that
$u$ is smooth when $\Sigma$ is smooth: the real analyticity of $u$ follows then directly from \cite[Th.~1.8.3]{KP}.

The following is a classical Weierstrass type result. As usual, given a smooth function
$g$ defined in a neighborhood of a compact set $K\subset \mathbb R^m$ we denote by $\|g\|_{C^0 (K)}$ the number $\max \{ |u(x)|:x\in K\}$ and we let
\[
\|g\|_{C^j (K)} := \sum_{|I| \leq j} \|\partial^I g\|_{C^0 (K)}\, ,
\]
where, given a multiindex $I= (i_1, \ldots , i_m)\in \mathbb N^m$, we let $|I|= i_1 + \cdots + i_m$ and
\[
\partial^I f = \frac{\partial^{|I|} f}{\partial x_1^{i_1} \partial x_2^{i_2}\cdots \partial x^{i_m}_m}\, .
\]

\begin{proposition}\label{p:A2}
Let $U\subset \mathbb R^N$ be an open set, $K \subset U$ a compact set and $f: U \to \mathbb R$ a smooth function.
Given any $j\in \mathbb N$ and any $\varepsilon >0$, there is a polynomial $P$ such that $\|f-P\|_{C^j (K)} \leq \varepsilon$.
\end{proposition}

\begin{proof} Using a partition of unity subordinate to a finite cover of $K$ we can assume, without loss of generality, that $f\in C^\infty_c (U)$. 
The classical Weierstrass theorem corresponds to the case $j=0$, see for instance \cite{Rudin}: however the proof given in the latter reference, which regularizes $f$ by convolution with suitable polynomials, gives easily the statement above for general $j$. Nash in \cite{Nash1952} provides instead the following elegant argument. Consider first a box $[-M/2, M/2]^N\subset \mathbb R^N$ containing the support of $f$ and let $\tilde{f}$ be the $M$-periodic function which coincides with $f$ on the box. If we expand $\tilde{f}$ in Fourier series as
\[
\tilde{f} (x) = \sum_{\lambda \in \mathbb Z^N} a_\lambda e^{2\pi M \lambda \cdot x}\, 
\]
and consider the partial sums
\[
S_m (x):= \sum_{|\lambda|\leq m} a_\lambda e^{2\pi M \lambda \cdot x}\, ,
\]
then clearly $\|S_m - f\|_{C^j (K)} = \|S_m -\tilde{f}\|_{C^j (K)} \leq \varepsilon/2$ for $m$ large enough.
On the other hand $S_m$ is an entire analytic function and thus for a sufficiently large degree $d$ the Taylor polynomial $T^d_m$ of $S_m$ at $0$ satisfies $\|S_m - T^d_m\|_{C^j (K)} \leq \varepsilon/2$. 
\end{proof}

\section{The algebraic description of Nash manifolds and the characterization of representations as sheets}

\begin{proof}[Proof of Proposition~\ref{p:alg_uniq}] 
First of all, for any (proper) ideal $\mathscr{I}$ the set $Z = Z (\mathscr{I})$ of points of $\Sigma$ at which all elements of $\mathscr{I}$ vanish must be nonempty. Otherwise, for any point $p\in \Sigma$ there would be an element $f_p \in \mathscr{I}$ such that $f_p (p) \neq 0$. Choose then an open neighborhood $U_p$ such that $f_p \neq 0$ on $U_p$ and cover $\Sigma$ with finitely many $U_{p_i}$. The function $f:= \sum_i f_{p_i}^2$ would belong to the ideal $\mathscr{I}$ and would be everywhere nonzero. But then $\frac{1}{f}$ would belong to $\mathscr{R}$, $f\cdot \frac{1}{f} =1$ would belong to the ideal $\mathscr{I}$ and the latter would coincide with $\mathscr{R}$, contradicting the assumption that $\mathscr{I}$ s a proper ideal. 

Given a point $p$ and a basic set $\mathscr{B} = \{v_1, \ldots , v_N\}\subset \mathscr{R}$, the function 
\[
g (y) := \sum_i (v_i (y) - v_i (p))^2
\]
vanishes only at $p$ and belongs to $\mathscr{R}$. Thus we have:
\begin{itemize}
\item[(i)] The set $\mathscr{I} (p)$ of all elements which vanish at $p$ must be nonempty. Moreover, it cannot be the whole ring $\mathscr{R}$ because it does not contain the constant function $1$. It is thus a proper ideal and it must be maximal: 
any larger ideal $\mathscr{J}$ would necessarily have $Z ( \mathscr{J}) = \emptyset$.
\item[(ii)] If $\mathscr{I}$ is a maximal ideal, then there must be an element $p\in Z (\mathscr{I})$ and, since $\mathscr{I} \subset \mathscr{I} (p)$, we must necessarily have $\mathscr{I} = \mathscr{I} (p)$. 
\end{itemize}

This shows the first part of the proposition.
Next, let $(\Sigma_1, \mathscr{R}_1)$ and $(\Sigma_2, \mathscr{R}_2)$ be two Nash manifolds. Clearly,
if $\varphi: \Sigma_1\to \Sigma_2$ is a real analytic diffeomorphism such that $\phi (f) := f\circ \varphi^{-1}$ maps $\mathscr{R}_1$ onto $\mathscr{R}_2$, then $\phi$ is a ring isomorphism. Vice versa, let $\phi: \mathscr{R}_1 \to \mathscr{R}_2$ be a ring isomorphism. Using the correspondence above, given a point $p\in \Sigma_1$ we have a corresponding maximal ideal $\mathscr{I} (p) \subset \mathscr{R}$, which is mapped by $\phi$ into a maximal  ideal of $\mathscr{R}_2$: there is then a point
$\varphi (p)\in \Sigma_2$ such that $\phi (\mathscr{I} (p)) = \mathscr{I} (\varphi (p))$. 
We now wish to show that 
\begin{equation}\label{e:iso_campo}
\phi (f) (\varphi (p)) =f (p)\, .
\end{equation}
First observe that:
\begin{equation}\label{e:Van}
\mbox{if $f$ vanishes at $p$, then $\phi (f)$ must vanish at $\varphi (p)$.}
\end{equation}
This follows from the property $\phi (f)\in \phi (\mathscr{I} (p)) = \mathscr{I} (\varphi (p))$. 

Next we follow the convention that, given a number $q\in \mathbb R$, we let $q$ denote both the function constantly equal to $q$ on $\Sigma_1$ and that equal to $q$ on $\Sigma_2$.
Since $1$ is the multiplicative unit of $\mathscr{R}_1$ and $\mathscr{R}_2$, then $\phi (1) =1$. Hence, using the ring axioms, it follows easily that $\phi (q)=q$ for any $q\in \mathbb Q$. Observe next that if $f\in \mathscr{R}_1$ is a positive function on $\Sigma_1$, then $g := \sqrt{f}$ is a real analytic function and it must belong to $\mathscr{R}_1$, otherwise the latter ring would not satisfy the maximality condition of Definition~\ref{d:alg_manifold}(d). Hence, if $f>0$, then $\phi (f) = (\phi (\sqrt{f}))^2 \geq 0$. Thus $f>g$ implies $\phi (f) \geq \phi (g)$. Fix therefore a constant real $\alpha$ and two rational numbers $q > \alpha >q'$. We conclude $q=\phi (q) \geq \phi (\alpha)\geq \phi (q')=q'$. Since $q$ and $q'$ might be chosen arbitrarily close to $\alpha$, this implies that $\phi (\alpha) = \alpha$.

Having established the latter identity, we can combine it with \eqref{e:Van} to conclude \eqref{e:iso_campo}. Indeed, assume $f (p) = \alpha$.
Then $g = f-\alpha$ vanishes at $p$ and thus, by \eqref{e:Van}, $\phi (g) = \phi (f)-\alpha$ vanishes at $\varphi (p)$: thus $\phi (f) (\varphi (p)) =
f (p)$. 

Next, $\varphi^{-1}$ is the map induced by the inverse of the isomorphism $\phi$, from which we clearly conclude $\phi (f) = f\circ \varphi^{-1}$.  
It remains to show that $\varphi$ is real analytic: the same argument will give the real analyticity of $\varphi^{-1}$ as well, thus completing the proof.
Let $\mathscr{B}_1 = \{f_1, \ldots , f_N\}$ be a basic set for $(\Sigma_1 , \mathscr{R}_1)$ and $\mathscr{B}_2 = \{g_{N+1}, \ldots , g_{N+M}\}$ be a basic set for $(\Sigma_2, \mathscr{R}_2)$. Set $g_i := f_i \circ \varphi^{-1} = \phi (f_i)$ for $i\leq N$ and $f_j := g_j \circ \varphi
= \phi^{-1} (g_j)$ for $j\geq N+1$. Then $\{f_1, \ldots , f_{N+M}\}$ and $\{g_1, \ldots , g_{N+M}\}$ are basic sets for $\Sigma_1$ and $\Sigma_2$ respectively. The map $f= (f_1, \ldots , f_{N+M}): \Sigma_1 \to \mathbb R^{N+M}$ is a real analytic embedding of $\Sigma_1$
and $g = (g_1, \ldots , g_{N+M})$ a real analytic embedding of $\Sigma_2$ with the same image $S$. We therefore conclude that
$\varphi = (g|_S)^{-1} \circ f$ is real analytic, which completes the proof. 
\end{proof}

\begin{proof}[Proof of Proposition~\ref{p:rep=var}]
{\bf Representation $\Longrightarrow$ Sheet.} We consider first a Nash manifold $(\Sigma, \mathscr{R})$ of dimension $n$ and a representation $\mathscr{B} = \{f_1, \ldots , f_N\}\subset \mathscr{R}$. Our goal is thus to show that, if we set $f = (f_1, \ldots , f_N)$, then $S := f (\Sigma)$ is a sheet of an $n$-dimensional algebraic subvariety $V\subset \mathbb R^N$. First, recalling that $\mathscr{B}$ is a basic set, we know that for each choice of $1\leq i_1< i_2 < \ldots < i_{n+1}\leq N$ there is a (nontrivial) polynomial $P= P_{i_1\ldots i_{n+1}}$ such that
$P (f_{i_1}, \ldots , f_{i_{n+1}}) =0$. Let then $V_0$ be the corresponding algebraic subvariety, namely the set of common zeros of the polynomials $P_{i_1\ldots i_{n+1}}$. Clearly, by Proposition~\ref{p:F3-6}(b) the dimension of $V_0$ can be at most $n$. Otherwise there would be a point $q\in V_0$ of maximal dimension $d\geq n+1$ and there would be a neighborhood $U$ of $q$ such that $V_0\cap U$ is a real analytic $d$-dimensional submanifold of $\mathbb R^N$. This would mean that, up to a relabeling of the coordinates and to a possible restriction of the neighborhood, $U\cap V_0$ is the graph of a real analytic function of the first $d$ variables $x_1, \ldots, x_{d}$. But then
this would contradict the existence of a nontrivial polynomial of the first $n+1\leq d$ variables which vanishes on $V_0$.

Note moreover that, since $f$ is a smooth embedding of $\Sigma$, by Corollary~\ref{r:Haus=alg} the dimension must also be at least 
$n$. Hence we have concluded that the dimension of $V$ is precisely $n$.

Next, if $V_0$ is reducible, then there are two nontrivial subvarieties $V$ and $W$ of $V_0$ such that $V_0 = V\cup W$.
One of them, say $V$, must intersect $S$ on a set $A$ of positive $n$-dimensional volume. If $P$ is any polynomial among the ones defining $V$, we then must have $P (f_1, \ldots , f_N) =0$ on $A$: however, since $P (f_1, \ldots , f_N)$ is real analytic, $A$ has positive measure and $S$ is a connected submanifold of $\mathbb R^N$, we necessarily have $P (f_1, \ldots , f_N)=0$ on the whole $S$. We thus conclude that $S \subset V =: V_1$. If $V_1$ were reducible, we can go on with the above procedure and create a sequence $V_0\supset V_1 \supset \ldots$ of algebraic varieties containing $S$: however, by the well-known descending chain condition in the Zariski topology (cf.~\cite{Hartshorne}), this procedure must stop after a finite number of steps. Thus, we have achieved the existence of an $n$-dimensional irreducible subvariety $V$ such that
$S\subset V$.

\medskip

We claim that $S$ is a sheet of $V$. First of all, by Corollary~\ref{r:Haus=alg}, $S$ must contain a general point $p$ of $V$
because its dimension is $n$. Moreover,
by Proposition~\ref{p:F3-6} we know that there is a neighborhood $U$ of $p$ such that $U\cap V$ is an $n$-dimensional submanifold. By further restricting $U$ we can assume that both $U\cap V$ and $U\cap S$ are connected $n$-dimensional submanifolds. Since $S\subset V$, we must obviously have $S\cap U = V\cap U$. Hence $p$ is a point which satisfies condition (c) in Definition~\ref{d:sheets}. Next, fix a second point $q\in S$ and let $\bar p= f^{-1} (p)$ and $\bar q = f^{-1} (q)$. Since $\Sigma$ is a connected real analytic manifold, we clearly know that there is $\bar \gamma: [0,1]\to \Sigma$ real analytic\footnote{Here we are using the nontrivial fact that in a connected real analytic manifold any pair of points can be joined by a real analytic arc. One simple argument goes as follows: use first Whitney's theorem to assume, without loss of generality, that $\Sigma$ is a real analytic submanifold of $\mathbb R^N$. Fix two points $p$ and $q$ and use the existence of a real analytic projection in a neighborhood of $\Sigma$ to reduce our claim to the existence of a real analytic arc connecting any two points inside a connected open subset of the Euclidean space. Finally use the Weierstrass polynomial approximation theorem to show the latter claim.}  such that $\bar \gamma (0) = \bar q$ and $\bar \gamma (1) = \bar q$. Thus $\gamma := f \circ \bar\gamma$ is a map as in Definition~\ref{d:sheets}(a). It remains to show that $S$ is maximal among the subsets of $V$ satisfying Definition~\ref{d:sheets}(a). 

So, let $\tilde{S}$ be the maximal one containing $S$ and fix $p\in \tilde{S}$: we claim that indeed $p\in S$. By assumption we know that there is a real analytic curve $\gamma: [0,1]\to \tilde{S}$ such that $\gamma (0) \in S$ is a general point and $\gamma (1) = p$. First of all, since $\gamma (0)$ is a general point of $V$, there is a neighborhood $U$ of $p$ where $S$ and $V$ coincide. Hence there is $\delta>0$ such that $\gamma ([0, \delta[)\subset S$. Set next 
\[
s:= \sup \{s\in [0,1] : \gamma ([0,s[)\subset S\}\, .
\]
Clearly, $s$ is a maximum. Moreover, by compactness of $S$, $q:= \gamma (s) \in S$: we need then to show that $s=1$. Assume, instead, that $s<1$. Let $U$ be some coordinate chart in the real analytic manifold $\Sigma$ containing
$f^{-1} (q)$ and $y : U \to \mathbb R^n$ corresponding real analytic coordinates. There is $\delta >0$ such that
$f^{-1}  (\gamma ([s-\delta, s]))\subset U$. The map $\tilde{\gamma} := y \circ f^{-1} \circ \gamma :[s-\delta, s] \to \mathbb R^n$ is then
real analytic. Hence
there is $\eta>0$ so that $\tilde{\gamma} (t)$ can be expanded in power series of $(t-s)$ on the interval $]s-\eta, s]$. Such power series
converges then on $]s-\eta, s+\eta[$ and extends $\tilde{\gamma}$ to a real analytic map on $]s-\eta, s+\eta[$. Now,
$\gamma|_{]s-\eta, s+\eta[}$ and $\bar \gamma := f \circ y^{-1} \circ \tilde\gamma$ are two maps which coincide on the 
interval $]s-\eta, s]$: since they are both real analytic, they must then coincide on the whole $]s-\eta, s+\eta[$. Hence 
$\gamma ([0, s+\eta[)\subset S$, contradicting the maximality of $s$. 

\medskip

{\bf Sheet $\Longrightarrow$ Representation.} Let $v: \Sigma \to \mathbb R^N$ be a real analytic embedding of an $n$-dimensional real analytic manifold such that $S = v (\Sigma)$ is a sheet of an algebraic subvariety $V$ with minimal field of definition $\mathbb F$. Pick now a point $q\in S$ for which there is neighborhood $U$ with $U\cap S = V \cap U$. By Corollary~\ref{r:Haus=alg} there must necessarily be a point $p\in V\cap U = S\cap U$ with $m := {\rm dim} (V) = {\rm dim}_V (p) \geq n$. 
By Proposition~\ref{p:F3-6}(c) there is an algebraic subvariety $W\subset V$ with algebraic dimension $m$ containing $p$ and with field of definition $\mathbb F'\subset \mathbb F$. Note that by the latter property we must necessarily have ${\rm dim}_W (p) \geq m$ and thus
$p$ is a general point of $W$. Therefore, by Proposition~\ref{p:F3-6}(b) applied to $W$, there is a neighborhood $U'\subset U$ of $p$ such that $U'\cap W$ is an $m$-dimensional
connected submanifold: since $U'\cap S =U' \cap V \supset U' \cap W$ and $S$ is an $n$-dimensional submanifold, $m=n$ and there is a neighborhood
of $p$ where $W$ and $S$ coincide. 

We claim now that $v (\Sigma) = S\subset W$. Fix $p' \in S$: we know that there is an analytic function $\gamma : [0,1]\to S$ such that $\gamma (0) =p$ and $\gamma (1) = p'$. If $P$ is a polynomial of $N$ variables which vanishes on $W$, then $P\circ \gamma$ vanishes on a neighborhood of $0$. Since $P\circ \gamma$ is real analytic, it must thus vanish on the whole interval $[0,1]$ and thus $P (p')=0$. This shows that $p'$ is a zero of any polynomial which vanishes on $W$, which implies that $p'\in W$. From the very definition of sheet, it follows that $S$ is not only a sheet of the subvariety $V$, but also a sheet of the subvariety $W$. 

Having established that $v (\Sigma)$ is a sheet of an $n$-dimensional subvariety of $\mathbb R^N$, it follows that any collection of $n+1$ functions chosen among the coordinates $v_1, \ldots , v_N$ must satisfy a nontrivial polynomial relation. Thus $\mathscr{B} := \{v_1, \ldots , v_N\}$ is a basic set. Now consider the ring $\mathscr{R}'$ of real analytic functions generated by $\mathscr{B}$: such ring obviously satisfies the requirements (a) and (c) of Definition~\ref{d:alg_manifold}. Choosing a maximal one (among those satisfying these two requirements and containing $\mathscr{B}$) we achieve the desired structure $(\Sigma, \mathscr{R})$ of which $v$ is a representation.  
\end{proof}

\section{Proof of the existence of representations and of the approximation theorem}

The proofs of the two theorems follow indeed the same path and will be given at the same time. Before coming to them we need however the following
very important lemma.

\begin{lemma}\label{l:poly_dec}
Let $Q$ and $R$ be two monic polynomials in one variable of degrees $d_1$ and $d_2$ with real coefficients and no common factors. Let $P=Q R$ be their product. 
Then any monic polynomial $\tilde{P}$ of degree $d= d_1+d_2$ with real coefficients in a suitable neighborhood $U$ of $P$ can be factorized in two monic polynomials $\tilde{Q}$ and $\tilde{R}$ of degrees $d_1$ and $d_2$, with real coefficients and
which lie near $Q$ and $R$ respectively. Such decomposition is unique and the coefficients of the polynomials of each factor depend analytically upon those of $\tilde P$. 
\end{lemma}
\begin{proof}
First of all we show that the decomposition is unique. Note that two polynomials have no common factors if and only if
they have no (complex) root in common. Let $z_1, \ldots , z_{d_1}$ be the roots of $Q$ and $w_1, \ldots , w_{d_2}$ those of $R$ (with repetitions, accounting for 
multiplicities). If $\tilde{P}$ is close to $P = QR$, then its roots will be close to $z_1, \ldots z_{d_1}, w_1, \ldots , w_{d_2}$ and thus they can be divided in unique way
in two groups: $d_1$ roots close to the roots of $Q$ and $d_2$ roots close to those of $R$. Clearly the zeros of the factor $\tilde{Q}$ must be close to those of $Q$ and
thus $\tilde{Q}$ is uniquely determined, which in turn determines also the other factor $\tilde{R}$. Note moreover that the coefficients of
both $\tilde{Q}$ and $\tilde{R}$ must be real: it suffices to show that if a (nonreal) root $\zeta$ of $\tilde{P}$ is a root of 
$\tilde{Q}$, then its complex conjugate $\bar \zeta$ is also a root of $\tilde{Q}$. Indeed, either $\zeta$ is close to a real root of $Q$, in which case $\bar \zeta$ is close to the same root, or $\zeta$ is close to a nonreal root of $z_i$ of $Q$, in which case $\bar \zeta$ is close to $\bar z_i$, which must be a root of $Q$ because $Q$ has real coefficients. 

In order to show the existence and the real analytic dependence, set 
\begin{align*}
Q (x) & =  x^{d_1} + \sum_{i=1}^{d_1} a_i x^{d_1-i}\\
R (x) & = x^{d_2} + \sum_{i=1}^{d_2} b_i x^{d_2-i}\\
P (x) & = x^d + \sum_{i=1}^d c_i x^{d-i}\, .
\end{align*}
We then desire to find a neighborhood $U$ of $c = (c_1, \ldots , c_d)\in \mathbb R^d$ and a real analytic map
$(\alpha, \beta): U\to \mathbb R^{d_1}\times \mathbb R^{d_2}$ with the properties that
\begin{itemize}
\item[(a)] $x^d + \sum_i \tilde c_i x^{d-i} = (x^{d_1} + \sum_j \alpha_j (\tilde c) x^{d_1-j})(x^{d_2} + \sum_k \beta_k (\tilde c) x^{d_2-k})$ for any $\tilde c \in U$;
\item[(b)] $\alpha (c) = a$ and $\beta (c) = b$.
\end{itemize}
Given $\alpha$, $\beta$ vectors in some neighborhoods $U_1$ and $U_2$ of $a$ and $b$, let $Q_\alpha := x^{d_1} + \sum_j \alpha_j x^{d_1-j}$, $R_\beta := x^{d_2} + \sum_k \beta_k x^{d_2-k}$
and $Q_\alpha R_\beta = x^d + \sum_i \gamma_i x^{d-i}$. This defines a real analytic (in fact polynomial!) map $U_1 \times U_2 \ni (\alpha, \beta) \mapsto \gamma (\alpha , \beta) \in \mathbb R^d$ with the property that $\gamma (a, b) = c$. Our claim will then follow from the inverse function theorem if we can show that the determinant of the matrix of partial derivatives of $\gamma$ at the point $(a,b)$ is nonzero. The latter matrix is however the Sylvester matrix of the two polynomials $Q$ and $R$: the determinant of the Sylvester matrix of two polynomials (called the resultant), vanishes if and only if the two polynomials have a common zero, see \cite{Akritas}.  
\end{proof}

We are now ready to prove the two main theorems, namely Theorem~\ref{t:alg_approx} and Theorem~\ref{t:alg_main_2}.

\begin{proof}[Proof of Theorem~\ref{t:alg_approx}]
We start with Theorem~\ref{t:alg_approx} and consider therefore a smooth embedding $w: \Sigma \to \mathbb R^m$ of a smooth closed connected manifold $\Sigma$ of dimension
$n$. By Whitney's theorem we can assume, without loss of generality, that $w$ is real analytic. Consider now a tubular neighborhood $U := U_{4\delta} (\Sigma)$ so that the
nearest point projection $x\mapsto \pi (x) \in \Sigma$ is real analytic on $U$ and let $v: U \to \mathbb R^m$ be the function $v (x) := \pi (x) -x$. For each $x$ let also
$T_{\pi (x)} \Sigma$ be the $n$-dimensional tangent space to $\Sigma$ at $\pi (x)$ (considered as a linear subspace of $\mathbb R^m$) and let $\xi \mapsto {\mathbf K} (x) \xi$ be
the orthogonal projection from $\mathbb R^m$ onto $T^\perp_{\pi (x)} \Sigma$, namely the orthogonal complement of the tangent $T_{\pi (x)} \Sigma$. 
We therefore consider ${\mathbf K} (x)$ to be a symmetric $m\times m$ matrix with coefficients which
depend analytically upon $x$. Let next $u$ and $\mathbf L$ be two maps with polynomial dependence on $x$ which on $U_{3\delta} (\Sigma)$ approximate well the maps $v$ and $\mathbf K$.
More precisely
\begin{itemize}
\item[(i)] $\mathbf{L} (x)$ is an $m\times m$ symmetric matrix for every $x$, with entries which are polynomial functions of the variable $x$; similarly the components of $u(x)$ are polynomial functions of $x$;
\item[(ii)] $\|u-v\|_{C^N (U_{3\delta} (\Sigma))} + \|\mathbf{L}- \mathbf{K}\|_{C^N (U_{3\delta} (\Sigma))} \leq \eta$, where $N$ is a large natural number and $\eta$ a small real
number, whose choices will be specified later. 
\end{itemize}
The characteristic polynomial of $\mathbf{K}$ is $P (\lambda) = (\lambda -1)^{m-n} \lambda^n$. We can then apply Lemma~\ref{l:poly_dec} and,
assuming $\eta$ is sufficiently small, the characteristic polynomial $P_x (\lambda)$ of $\mathbf{L} (x)$ can be factorized as $Q_x (\lambda) R_x (\lambda)$
where 
\begin{itemize}
\item[(iii)] $R_x (\lambda)$ is close to $\lambda^n$;
\item[(iv)] $Q_x (\lambda)$ is close to $(\lambda -1)^{m-n}$;
\item[(v)] The coefficients of $R_x$ and $Q_x$ depend analytically upon $x$.
\end{itemize}
It turns out that both $Q_x$ and $R_x$ have all real roots (since $\mathbf L (x)$ is symmetric, its eigenvalues are all real). Moreover, the eigenvectors
with eigenvalues which are roots of $R_x$ span an $n$-dimensional vector subspace $\tau (x)$ of $\mathbb R^m$ which is close to $T_{\pi (x)} (\Sigma)$. 
On the other hand the eigenvectors with eigenvalues which are roots of $Q_x$ span the orthogonal of $\tau (x)$, which we will denote by $\tau (x)^\perp$ (recall that $\mathbf{L} (x)$ is a symmetric
matrix!). Consider next the symmetric matrix $\mathbf{P} (x) = R_x (\mathbf{L} (x))$. Then the kernel of the linear map $\xi \mapsto \mathbf{P} (x)\xi$ 
is $\tau (x)$. Moreover $|\mathbf{P} (x) \xi -\xi| \leq C \eta |\xi|$ for every $\xi\in \tau (x)^\perp$, where $C$ is only a dimensional constant: this happens because $\mathbf{P} (x)$ is close to $(\mathbf{K} (x))^n$, whose linear action on $T^\perp_{\pi (x)} \Sigma$ is the identity. 

Consider now the map 
\[
z (x) := x + v(x) - \underbrace{\mathbf{K} (x) \mathbf{P} (x) u (x)}_{=:\psi (x)}\, .
\]
The map $x\mapsto z (x)$ is clearly real analytic on $U_{2\delta}$. Moreover, as $\eta\downarrow 0$, the map $v-\psi$ converges to $x \mapsto v (x)- \mathbf{K} (x) \mathbf{K} (x) v (x) = 0$,
because $\mathbf{K} (x) v (x) = v (x)$. The latter convergence is in $C^N$. Since $N$ is larger than $1$, for $\eta$ sufficiently close to $0$ this will imply the local invertibility of the
function $z$. In fact, by the inverse function theorem and compactness of $\overline{U_{3\delta} (\Sigma)}$ we conclude the existence of a $\sigma>0$ and an $\eta_0$ such that, if $\eta < \eta_0$,
then $z$ is injective in $B_\sigma (y)$ for every $y\in U_{2\delta} (\Sigma)$. Then, choosing $\eta < \min \{\eta_0, \sigma/(3 C)\}$ for a suitable dimensional constant $C$ we conclude the global injectivity of $z$ on $U_{2\delta} (\Sigma)$: if we have
$z (x) = z (x')$ and $x\neq x'$, then necessarily $|x-x'|\geq \sigma$. On the other hand the $C^0$ norm of the difference between $z$ and the identity map is given by $C\eta$ and thus
we can estimate 
\[
|z (x) - z(x')|\geq |x-x'| - |z (x)-x| - |z(x')-x'|\geq \sigma - \frac{2\sigma}{3}\, .
\] 
Finally, by possibly choosing $\eta$ even smaller, we can assume that $U_\delta (\Sigma)$ is
contained in $z (U_{2\delta} (\Sigma))$. 

Let now $z^{-1}$ be the inverse of $z$ on $U_\delta (\Sigma)$, which is analytic by the inverse function theorem. We claim that the real analytic subvariety $\Gamma = z^{-1} (\Sigma)$ is a sheet of an algebraic subvariety:
this would complete the proof of Theorem~\ref{t:alg_approx}, provided $N$ is large enough and $\eta$ small enough.

\medskip

Note now that, for any choice of $x$, $x+ v (x) = \pi (x)$ belongs to $\Sigma$ and $\psi (x)$ is orthogonal to $T_{\pi (x)} \Sigma$, by definition of $\mathbf{K} (x)$. Hence
$z (x)$ belongs to $\Sigma$ if and only if $\psi (x) = 0$. We conclude therefore that $\Gamma$ is indeed the set where $\psi$ vanishes. Recall moreover that, choosing $\eta$ sufficiently small,
$\mathbf{P} (x) u (x)$ belongs to the plane $\tau (x)^\perp$ which is close to $T_{\pi (x)}^\perp \Sigma$: hence $\mathbf{K} (x) \mathbf{P} (x) u(x) =0$ is equivalent to the condition
$\mathbf{P} (x) u (x) = 0$. $\Gamma$ is therefore the zero set of 
\[
R_x (\mathbf{L} (x)) u (x) = 0\, .
\]
Note however that the coefficients of the polynomial $R_x (\lambda)$ are just analytic functions of $x$ and {\em not} polynomial functions of $x$: it is therefore
not obvious that $\Gamma$ is a sheet of an algebraic subvariety. From now on we let $\phi (x)$ be the map $R_x (\mathbf{L} (x)) u (x)$. 

Consider now $\mathbb R^{m+n}$ as a product of $\mathbb R^m$ with the linear space of polynomials of degree $n$ and real coefficients in the unknown $\lambda$. 
In other words, to every point $(x, a)\in \mathbb R^{m+n}$ we associate the pair $x\in \mathbb R^n$ and $p_a (\lambda) = \lambda^n + a_1 \lambda^{n-1} + \ldots + a_n$.
For any $(x,a)$ 
consider the polynomial $q_{x,a} (\lambda)$ which is the remainder of the division of $P_x (\lambda)$, the characteristic polynomial of $\mathbf{L} (x)$, by the polynomial $p_a (\lambda)$. In particular, let $\eta_j (x,a)$ be the coefficients of $q_{x, a}$, namely
\[
q_{x,a} (\lambda) = \eta_1 (x,a) \lambda^{n-1} + \eta_2 (x,a) \lambda^{n-2} + \ldots + \eta_n (x,a)\, . 
\]
The corresponding map $(x,a) \mapsto \eta (x,a) = (\eta_1 (x,a), \ldots , \eta_n (x,a))$ is a polynomial map, because the coefficients of $P_x (\lambda)$ depend polynomially on $x$! For any element $(x, a)$, define $\varphi (x, a) := p_a (\mathbf{L} (x)) u (x)$ and consider thus the system of polynomial equations
\begin{align}\label{e:sistema_croce}
\left\{
\begin{array}{ll}
\eta (x,a) = 0\\ \\
\varphi (x,a) = 0
\end{array}
\right.
\end{align}
Such system defines a real algebraic subvariety $V$ of $\mathbb R^{m+n}$. 
Now, consider the analytic map $x \mapsto \Psi (x) = (x, R_x) \in \mathbb R^{m+n}$. Since the remainder of the division of $P_x$ by $R_x$ is $0$, we clearly have 
$\eta (\Psi (x)) =0$. Moreover, since $\varphi (x, R_x) = \phi (x)$, we conclude that $\Psi (\Gamma)$ is a subset of the set of solutions of \eqref{e:sistema_croce},
namely a subset of $V$. Moreover $\Psi (\Gamma)$ is a real analytic embedding of $\Gamma$ and hence also a real analytic embedding of $\Sigma$. We next claim that
$\Psi (\Gamma)$ is in fact an isolated sheet of $V$. The only thing we need to show is that in a neighborhood of $\Psi (\Gamma)$ the only solutions of \eqref{e:sistema_croce}
must be images of $\Gamma$ through $\Psi$. If $(x',a)$ is a zero of \eqref{e:sistema_croce} near an element of $(x, R_x)\in \Psi (\Gamma)$, it then follows that the polynomial $p_a$ 
must be close to the polynomial $R_x$ and must be a factor of $P_{x'}$. Recall however that $R_x (\lambda)$ is close to the polynomial $\lambda^n$ and,
by Lemma~\ref{l:poly_dec}, nearby $\lambda^n$ there is a unique factor of $P_{x'}$ which is a monic polynomial of degree $n$ close to $\lambda^n$: such
factor is $R_{x'}$! This implies that $p_a = R_{x'}$ and
hence that $\varphi (x',a) = \phi (x')$. But then $\phi (x') =0$ implies that $x'\in \Gamma$, which completes the proof that $\Psi (\Gamma)$ is an isolated sheet of the real algebraic subvariety $V$ of $\mathbb R^{m+n}$.

\medskip
In particular, $\Psi (\Gamma)$ is a proper representation, by Proposition~\ref{p:rep=var}. 
But $\Gamma$ is a projection of such representation, which
is still an analytic submanifold and thus it is easy to see that $\Gamma$ is also a representation of $\Sigma$: namely the components of $z^{-1}: \Sigma \to \mathbb R^n$ give a basic set
$\mathscr{B}$ 
of $\Sigma$ and, using the same procedure of the proof of Proposition~\ref{p:rep=var} we can find a Nash ring $\mathscr{R}$ containing $\mathscr{B}$, concluding the proof of Theorem~\ref{t:alg_approx}.
\end{proof}

\begin{proof}[Proof of Theorem~\ref{t:alg_main_2}] Fix a connected smooth closed differentiable manifold of dimension $n$ and, following the previous proof, consider the isolated sheet $\Psi (\Gamma)$ of the algebraic subvariety $V$ of $\mathbb R^{m+n}$ constructed above.
We next use the classical projection argument of Whitney, cf.~\cite{Whitney1936}, to show that, if $\pi$ is the orthogonal projection of $\mathbb R^{m+n}$ onto a generic (in the sense of Baire category) $2n+1$-dimensional subspace of $\mathbb R^{m+n}$, $\pi (\Psi (\Gamma))$ is still a submanifold, it is a connected component of $\pi (V)$ and that $\pi (V)$ is a an algebraic subvariety\footnote{The projection of an algebraic subvariety is not always an algebraic subvariety: here as well we are taking advantage of the genericity of the projection.} of $\mathbb R^{2n+1}$. The latter claim would then give a proper representation in $\mathbb R^{2n+1}$ and would thus show
Theorem~\ref{t:alg_main_2}. 

In order to accomplish this last task, we first observe that it suffices to show the existence of a projection onto an hyperplane, provided $m+n > 2n+1$: we can then keep reducing the dimension of the ambient Euclidean space until we reach $2n+1$. Next, for each hyperplane $\tau \subset \mathbb R^{m+n}$ we denote by $P_\tau$ the orthogonal projection onto it. The classical argument of Whitney implies that:
\begin{itemize}
\item[(a)] For a dense open subset of $\tau$'s in the Grassmanian $G$ of hyperplanes of $\mathbb R^{m+n}$ the map $P_\tau$ restricted on $\Psi (\Gamma)$ is an immersion (i.e. its differential has full rank at every $p\in \Psi (\Gamma)$).
\item[(b)] For a generic subset of $\tau$'s, $P_\tau$ is injective on $\Psi (\Gamma)$.
\end{itemize}
Thus for a dense open subset of $\tau$'s, $P_\tau \circ \Psi$ is an embedding of $\Gamma$. However, note that point (b) cannot be obviously extended to give injectivity of $P_\tau$ on
the whole subvariety $W$, because $W\setminus \Psi (\Gamma)$ is not necessarily a submanifold. We claim that, nonetheless, 
\begin{equation}\label{e:non_si_toccano}
P_\tau (\Psi (\Gamma))\cap P_\tau (W\setminus \Psi (\Gamma)) = \emptyset \qquad \mbox{for $\tau$ in a dense open subset of $G$.}
\end{equation}
Indeed, by Proposition~\ref{p:F3-6}, we know that $W\setminus \Psi (\Gamma)$ can be covered by countably many submanifolds $W_i$, of dimension $d_i \leq n$. Without loss of generality we can
assume that each $W_i$ is compact, has smooth boundary and does not intersect $\Psi (\Gamma)$. Consider the map $\Psi (\Gamma) \times W_i \ni (x,y) \mapsto z (x,y) := \frac{x-y}{|x-y|}$. Since $z$ is smooth, $z (\Psi (\Gamma) \times W_i)$ is a (closed) set of Hausdorff dimension at most $n + d_i \leq 2n < m+n-1$ and thus it is meager. In particular we conclude that the set $ K := z (\Psi (\Gamma) \times (W\setminus \Psi (\Gamma)))$ is a countable union of meager sets and thus a set of first category.
Hence the set $U\subset \mathbb S^{m+n-1}$ of points $p$ for which neither $p$ nor $-p$ belongs to $K$ is a generic subset of $\mathbb S^{m+n-1}$. Clearly, the set of hyperplanes $\tau$ orthogonal
to $\{p, -p\}\subset U$ is a generic subset of hyperplanes for which $P_\tau (\Psi (\Gamma)) \cap P_\tau (W\setminus \Psi (\Gamma)) = \emptyset$. 

Finally, it is a classical fact in real algebraic geometry that, for a generic subset of $\tau$, $P_\tau (W)$ is a real algebraic subvariety. Nash refers to the ``classical algebraic geometrical method of generic linear projection'', cf.~\cite[p. 415]{Nash1954}. However, it is possible to conclude the existence of a good projection directly with an algebraic variant of Whitney's argument\footnote{Many thanks to Riccardo Ghiloni for suggesting this argument, which follows closely the proof of \cite[Lem.~3.2]{Jelonek}.}. For completeness we report this alternative possibility in the next two paragraphs.

\medskip

Consider the complexification $W_\C \subset \C^{m+n}$ of $W$ (i.e., $W_\C$ is the smallest complex algebraic subvariety of $\C^{m+n}$ containing $W$). We have that $W_\C$ has (real) dimension $2n$, $W=W_\C \cap \R^{m+n}$ and $\Psi(\Gamma)$ is contained in the set $W_\C^*$ of nonsingular points of $W_\C$: for any point $p \in \Psi(\Gamma)$ there is a neighborhood $U$ of $p$ in $\C^{m+n}$ such that $U \cap W_\C$ is the zero set of $m$ polynomials with linearly independent gradients. We identify $\PP^{m+n-1}(\C)$ with the hyperplane at infinity of $\C^{m+n}$. Thus, we can consider $\PP^{m+n}(\C)$ as the union $\C^{m+n} \cup \PP^{m+n-1}(\C)$. For each nonzero vector $\tau$ of $\C^{m+n}$ we indicate by $[\tau]$ the corresponding point of $\PP^{m+n-1}(\C)$. Let $S$ denote the set of all $[\tau]$ of the form $\tau=x-y$ with $x,y \in W_\C$ and $x \neq y$. Note that $S$ has Hausdorff dimension at most $4n$ and the same is true for its closure\footnote{Observe that in this context the closure in the Euclidean topology coincides with the Zariski closure.} $T$ in $\PP^{m+n-1}(\C)$. The set $T$ contains all points at infinity of $W_\C$ (i.e. $T$ contains the intersection between $\PP^{m+n-1}(\C)$ and the closure of $W_\C$ in $\PP^{m+n}(\C)$). It is immediate to verify that $T$ contains also all the points $[\tau]$ such that $\tau$ is a nonzero vector of $\C^{m+n}$ tangent to the complex manifold $W_\C^*$ at some of its points. Since $2(m+n-1)>4n$, $T$ turns out to be a proper (i.e. $T \subsetneqq \PP^{m+n-1}(\C)$) complex algebraic subvariety of $\PP^{m+n-1}(\C)$. Thus, the subset $\PP^{m+n-1}(\R)$ of $\PP^{m+n-1}(\C)$ cannot be completely contained in $T$. Choose $[\nu] \in \PP^{m+n-1}(\R) \setminus T$. Denote by $H$ the hyperplane of $\R^{m+n}$ orthogonal to $\tau$ and by $H_\C \subset \C^{m+n}$ its complexification. 

Observe that the orthogonal projection $\rho:\R^{m+n} \to H$ extends to the projection $\rho_\C:\C^{m+n} \to H_\C$ which maps each point $x$ into the unique point of the intersection between $H_\C$ and the projective line joining $[\nu]$ and~$x$. Since $[\nu] \not\in T$, the restriction $\rho_\C'$ of $\rho_\C$ to $W_\C$ is proper and injective, and it is an immersion on $W_\C^*$. In particular, $\rho_\C'(W_\C)$ is a complex algebraic subvariety of $H_\C$ and $\rho_\C'(x)$ is a nonsingular point of $\rho_\C'(W_\C)$ for each $x \in \Psi(\Gamma)$. It follows immediately that the restriction $\rho'$ of $\rho$ to $W$ is an homeomorphism onto its image and it is a real analytic embedding on $\Psi(\Gamma)$. It remains to prove that $\rho'(W)$ is a real algebraic subvariety of $H$. It suffices to show that $\rho'(W)=\rho_\C'(W_\C) \cap H$ or, equivalently, that $\rho_\C'(W_\C) \cap H \subset \rho'(W)$. Let $x \in \rho_\C'(W_\C) \cap H$ and let $y \in W_\C$ with $\rho_\C'(y)=x$. We must prove that $y \in \R^{m+n}$. Note that the conjugate point $\overline{y}$ of $y$ belongs to $W_\C$, because $W_\C$ can be described by real polynomial equations. In this way, since $[\nu]$ is real (i.e. $[\nu] \in \PP^{m+n-1}(\R)$), $\rho_\C'(\overline{y})=\overline{x}=x=\rho_\C'(y)$. On the other hand, $\rho_\C'$ is injective and hence $y \in \R^{m+n}$ as desired.
\end{proof}

\section{Proof of the uniqueness of the Nash ring}

We finally turn to Theorem~\ref{t:alg_uniq}.
Let $(\Sigma, \mathscr{R}_1)$ and $(\Gamma, \mathscr{R}_2)$ be two structures of Nash manifolds on two diffeomorphic manifolds and consider 
two corresponding proper representations $v_1: \Sigma \to \mathbb R^{n_1}$ and $v_2 : \Gamma
\to \mathbb R^{n_2}$. Let $\alpha: \Gamma \to \Sigma$ be a diffeomorphism and, using Whitney's theorem, assume without loss of generality that
$\alpha$ is real analytic and define $a := v_1 \circ \alpha \circ v_2^{-1}$ on $v_2 (\Gamma)$. Consider a neighborhood $U_\delta (v_2 (\Gamma))$ where the nearest point projection $\pi_2$ on $v_2 (\Gamma)$ is real analytic and
let $w:= a \circ \pi_2$: $w$ is
a real analytic mapping from $U_\delta (v_2 (\Gamma))$ onto $v_1 (\Sigma)$. We can then approximate $w$ in $C^1$ with a map $z$ whose coordinate functions are polynomials.
If the approximation is good enough, we can assume that $w$ takes values in a neighborhood $U_\eta$ of $v_1 (\Sigma)$ where the nearest point projection $\pi_1$ is real analytic and well defined.
Now the nearest point projection $\pi_1 (y)$ of a point $y$ onto $v_1 (\Sigma)$ is in fact characterized by the orthogonality of $y-\pi_1 (y)$ to the tangent space to $v_1 (\Sigma)$ at $\pi_1 (y)$. It is easy to see that this is a set of polynomial conditions when $v_1 (\Sigma)$ is, as in this case, a smooth real algebraic submanifold. Thus $\pi_1$ is an algebraic function. 
Hence $\zeta := \pi_1 \circ z$ is also an algebraic function. If $z$ is close enough to $w$ in the $C^k$ norm, then the restriction of $z$ to $v_2 (\Gamma)$ will be close enough to
$a$ in the $C^1$ norm: in particular when this norm is sufficiently small the restriction of $z$ to $v_2 (\Sigma)$ must be a diffeomorphism of $v_2 (\Gamma)$ with $v_1 (\Sigma)$. By the implicit function theorem, the inverse will also be real analytic. Since, however, $z$ is algebraic, its inverse will also be algebraic. Thus $z$ gives the desired isomorphism between the two algebraic structures.

\chapter{$C^1$ isometric embeddings}

\section{Introduction}\label{s:C1_intro}

Consider a smooth $n$-dimensional manifold $\Sigma$ with a smooth Riemannian tensor $g$ on it. If $U\subset \Sigma$ is a coordinate patch, we write $g$ as customary in local coordinates:
\[
g = g_{ij} dx_i \otimes dx_j\, ,
\]
where we follow the Einstein's summation convention. The smoothness of $g$ means that, for any chart of the smooth atlas, the coefficients $g_{ij}$ are $C^\infty$ functions.

An isometric immersion (resp. embedding) $u : \Sigma \to \mathbb R^n$ is an immersion (resp. embedding) which preserves the length of curves, namely such that
\[
\ell_g (\gamma) = \ell_e (u\circ \gamma) \qquad \mbox{for any $C^1$ curve $\gamma: I \to \Sigma$.}
\]
Here $\ell_e (\eta)$ denotes the usual Euclidean length of a curve $\eta$, namely
\[
\ell_e (\eta) = \int |\dot \eta (t)|\, dt\, , 
\]
whereas $\ell_g (\gamma)$ denotes the length of $\gamma$ in the Riemannian manifold $(\Sigma, g)$: if $\gamma$ takes values in a coordinate patch $U\subset \Sigma$ the explicit formula is
\begin{equation}\label{e:Riem_length}
\ell_g (\gamma) = \int \sqrt{ g_{ij} (\gamma (t)) \dot\gamma_i (t) \dot\gamma_j (t)}\, dt\, .
\end{equation}
The existence of isometric immersions (resp. embeddings) is a classical problem, whose formulation is attributed to the Swiss mathematician Schl\"afli, see \cite{Schlaefli}. At the time of Nash's works \cite{Nash1954, Nash1956} comparatively little was known about the existence of such maps. Janet \cite{Janet}, Cartan \cite{Cartan} and Burstin \cite{Burstin} had proved the existence of local isometric embeddings in the case of analytic metrics. For the very particular case of $2$-dimensional spheres endowed with metrics of positive Gauss curvature, Weyl in \cite{Weyl} had raised the question of the existence of isometric embeddings in $\mathbb R^3$. The Weyl's problem was solved by Lewy in \cite{Lewy} for analytic metrics and, only shortly before Nash's work, another brilliant young mathematician, Louis Nirenberg, had settled the case of smooth metrics (in fact $C^4$, see Nirenberg's PhD thesis \cite{NirenbergPhD} and the note \cite{Nirenberg}); the same problem was solved independently by Pogorolev \cite{Pogorelov1951}, building upon the work of Alexandrov \cite{Alexandrov1948} (see also \cite{Pogorelov}). 

\medskip

In his two papers on the topic written in the fifties (he wrote a third contribution in the sixties, cf.~\cite{Nash1966}), Nash completely revolutionized the subject. He first proved a very counterintuitive fact which shocked the geometers of his time, namely the existence
of $C^1$ isometric embeddings in codimension $2$ in the absence of topological obstructions. He then showed the existence of smooth embeddings in sufficiently high codimension, introducing his celebrated approach to ``hard implicit function theorems''.
In this chapter we report the main statements and the arguments of the first paper \cite{Nash1954}.

\medskip

We start by establishing the following useful notation. First of all we will use the Einstein summation convention on repeated indices. We then will denote by $e$ the standard Euclidean metric on $\mathbb R^N$, which in the usual coordinates is expressed by the tensor 
\[
\delta_{ij} dx_i \otimes dx_j\, .
\] 
If $v: \Sigma \to \mathbb R^N$ is an immersion, we denote by $v^\sharp e$ the pull-back metric on $\Sigma$. When $U\subset \Sigma$ is a coordinate patch, the pull-back metric in the local coordinates is then given by
\[
v^\sharp e = (\partial_i v \cdot \partial_j v) dx_i \otimes dx_j\, ,
\]
where $\partial_i v$ is the $i$-th partial derivative of the map $v$ and $\cdot$ denotes the usual Euclidean scalar product.
The obvious necessary and sufficient condition in order for a $C^1$ map $u$ to be an isometry is then given by $u^\sharp e =g$, which amounts to the identities
\begin{equation}\label{e:PDE_isom}
g_{ij} = \partial_i u \cdot \partial_j u\, .
\end{equation}
Note that this is a system of $\frac{n (n+1)}{2}$ partial differential equations in $N$ unknowns (if the target of $u$ is $\mathbb R^N$).

In order to state the main theorems of Nash's 1954 note, we need to introduce the concept of ``short immersion''.
\begin{definition}[Short maps]\label{d:short}
Let $(\Sigma, g)$ be a Riemannian manifold. 
An immersion $v: \Sigma \to \mathbb R^N$ is short if we have the inequality $v^\sharp e \leq g$ in the sense of quadratic forms:
more precisely $h\leq g$ means that 
\begin{equation}\label{e:corta_def}
h_{ij} w^i w^j \leq g_{ij} w^i w^j \qquad \mbox{for any tangent vector $w$}.
\end{equation} 
Analogously  we write $h < g$ when \eqref{e:corta_def} holds with a {\em strict} inequality for any nonzero tangent vector. Hence, if the immersion $v: \Sigma \to \mathbb R^N$ satisfies the inequality $v^\sharp e < g$, we say that it is {\em strictly} short.
\end{definition}

Using \eqref{e:Riem_length} we see immediately that a short map shrinks the length of curves, namely $\ell_e (v(\gamma)) \leq \ell_g (\gamma)$ for every smooth curve $\gamma$.
The first main theorem of Nash's paper is then the following result

\begin{theorem}[Nash's $C^1$ isometric embedding theorem]\label{t:main_C1_1}
Let $(\Sigma, g)$ be a smooth closed $n$-dimensional Riemannian manifold and $v: \Sigma \to \mathbb R^N$ a $C^\infty$ short immersion with $N\geq n+2$. Then, for any $\varepsilon >0$ there is a $C^1$ isometric immersion $u: \Sigma\to \mathbb R^N$ such that $\|u-v\|_{C^0} < \varepsilon$. If $v$ is, in addition, an embedding, then $u$ can be assumed to be an embedding as well. 
\end{theorem}

The closedness assumption can be removed, but the corresponding statement is slightly more involved and in particular we need the
notion of ``limit set''. 

\begin{definition}[Limit set]\label{d:limit_set}
Let $\Sigma$ be a smooth manifold and $v: \Sigma \to \mathbb R^N$. Fix an exhaustion of compact sets $\Gamma_k \subset \Sigma$, namely $\Gamma_k \subset \Gamma_{k+1}$ and $\cup_k \Gamma_k = \Sigma$. The limit set of $v$ is the collection of points $q$ which are limits of any sequence $\{v (p_k)\}$ such that $p_k \in \Sigma \setminus \Gamma_k$. 
\end{definition}

\begin{theorem}[$C^1$ isometric embedding, nonclosed case]\label{t:main_C1_2}
Let $(\Sigma, g)$ be a smooth $n$-dimensional Riemannian manifold. The same conclusions of Theorem~\ref{t:main_C1_1} can be drawn if the map $v$ is short and its limit set does not intersect its image. Moreover, we can impose that the nearby isometry $u$ has the same limit set as $v$ if $v$ is strictly short. 
\end{theorem}

Combined with the classical theorem of Whitney on the existence of smooth immersions and embeddings, the above theorems have the following corollary.

\begin{corollary}\label{c:Nash+Whitney}
Any smooth $n$-dimensional Riemannian manifold has a $C^1$ isometric immersion in $\mathbb R^{2n}$ and a $C^1$ isometric embedding in $\mathbb R^{2n+1}$. If in addition the manifold is closed, then there is a $C^1$ isometric embedding\footnote{Closed manifolds can be $C^1$ isometrically {\em immersed} in lower dimension: already at the time of Nash's paper this could be shown in $\mathbb R^{2n-1}$ (for $n>1$!) using Whitney's immersion theorem. Nowadays we can use Cohen's solution of the immersion conjecture to lower the dimension to $n-a(n)$, where $a(n)$ is the number of $1$'s in the binary expansion of $n$, cf.~\cite{Cohen}.} in 
 $\mathbb R^{2n}$. 
\end{corollary}

\begin{remark}
In Nash's original paper the $C^0$ estimate of Theorem~\ref{t:main_C1_1} is not mentioned, but it is an obvious outcome of the proof. Moreover, Nash states explicitly that it is possible to relax the condition $N\geq n+2$ to the (optimal) $N\geq n+1$ using more involved computations, but he does not give any detail. Indeed, such a statement was proved shortly after by Kuiper in \cite{Kuiper}, with a suitable adaptation of Nash's argument. The final result is then often called the Nash--Kuiper Theorem.
\end{remark}

\medskip

The Nash--Kuiper $C^1$ isometric embedding theorem is often cited as one of the very first instances of Gromov's $h$-principle, cf.~\cite{Eliashberg,GromovBook}. Note that it implies that any closed $2$-dimensional oriented Riemannian manifold can be 
embedded in an arbitrarily small ball of the Euclidean $3$-dimensional space with a $C^1$ isometry. This statement is rather striking and
counterintuitive, especially if we compare it to the classical rigidity for the Weyl's problem (see the classical works of Cohn-Vossen and Herglotz \cite{CohnVossen,Herglotz}): if $\Sigma$ is a $2$-dimensional sphere and
$g$ a $C^2$ metric with positive Gauss curvature, the image of every $C^2$ isometric embedding $u: \Sigma \to \mathbb R^3$ is the boundary
of a convex body, uniquely determined up to rigid motions of $\mathbb R^3$. Nash's proof of Theorem~\ref{t:main_C1_1} (and Kuiper's subsequent modification) generates indeed a $C^1$ isometry which has no further regularity. It is interesting to notice
that a sufficiently strong H\"older continuity assumption on the first derivative is still enough for the validity of the rigidity statement in the Weyl's problem (see \cite{Borisov1958,CDS}), whereas for a sufficiently low H\"older exponent $\alpha$ the Nash--Kuiper Theorem still holds in $C^{1,\alpha}$
(see \cite{Borisov1965,CDS, DIS}). The existence of a threshold exponent distinguishing between the two different behaviors in low codimension is a widely open problem, cf.~\cite[p. 219]{GromovBook} and \cite[Problem 27]{Yau}, which bears several relations with a well-known conjecture in the theory of turbulence, solved very recently with methods inspired by Nash's approach to Theorem~\ref{t:main_C1_1}, cf.~\cite{DS-Inv,BDIS,S-ICM,BDSV,Isett}.

\section{Main iteration}\label{s:main_C1_iter}

We start by noticing that Theorem~\ref{t:main_C1_1} is a ``strict subset'' of Theorem~\ref{t:main_C1_2}: if $\Sigma$ is closed, then the limit set of any map is empty.  
Moreover, the following simple topological fact will be used several times:

\begin{lemma}\label{l:covering}
Let $\Sigma$ be a differentiable $n$-dimensional manifold and $\{V_\lambda\}$ an open cover of $\Sigma$. Then there is an open cover $\{U_\ell\}$ with the properties that:
\begin{itemize}
\item[(a)] each $U_\ell$ is contained in some $V_\lambda$;
\item[(b)] the closure of each $U_\ell$ is diffeomorphic to an $n$-dimensional ball;
\item[(c)] each $U_\ell$ intersects at most finitely many other elements of the cover;
\item[(d)] each point $p\in \Sigma$ has a neighborhood contained in at most $n+1$ elements of the cover;
\item[(e)] $\{U_\ell\}$ can be subdivided into $n+1$ classes $\mathcal{C}_i$ consisting of pairwise disjoint $U_\ell$'s.
\end{itemize}
\end{lemma}

\begin{proof}
By a classical theorem $\Sigma$ can be triangulated (see \cite{Whitehead}) and by locally refining the triangulation we can assume that each simplex is contained in some $V_\lambda$. Denote by $S$ such triangulation and enumerate its vertices as $\{S^0_i\}$, its $1$-dimensional edges 
as $\{S^1_i\}$ and so on. Then take the barycentric subdivision of $S$ and call it $T$ (cf.~Figure \ref{fig:3}).  We notice the following facts:
\begin{itemize}
\item[(i)] For each vertex $S^0_i$ consider the interior $U^0_i$ of the star of $S^0_i$ in the triangulation $T$, see Figure \ref{fig:4} (recall that the star of $S^0_i$ is usually defined as the union of all simplices of the triangulation which contain $S^0_i$, cf.~for instance \cite[p. 178]{Hatcher}). Observe that the $U^0_i$ are pairwise disjoint. 
\item[(ii)] For each edge $S^1_i$ consider the interior $U^1_i$ of the star of $S^1_i$ in the triangulation $T$, see Figure \ref{fig:4}. The $U^1_i$ are pairwise disjoint. Moreover, observe that
if $U^1_i \cap U^0_j \neq \emptyset$, then $S^0_j \subset S^1_i$.
\item[(iv)] Proceed likewise up to $n-1$.
Complete the collection $\{U^t_i: 0\leq t \leq n-1\}$ with the interiors $U^n_i$ of the $n$-dimensional simplices $S^n_i$ of $S$ and denote such final collection by $\mathscr{C}$. 
\end{itemize}

\begin{figure}
\centering
\begin{tikzpicture}
\draw (-6,0) -- (-3,2) -- (-2,4) -- (-1,1) -- (1,0.5) -- (-1.5,-1) -- (-3.5,-3) -- (-3.5,-1) -- (-6,0);
\draw (-3,2) -- (-1,1) -- (-1.5,-1) -- (-3.5,-1) -- (-3,2);
\draw (-2.5,0) -- (-3,2);
\draw (-2.5,0) -- (-1,1);
\draw (-2.5,0) -- (-1.5,-1);
\draw (-2.5,0) -- (-3.5,-1);

\draw (3,0) -- (6,2) -- (7,4) -- (8,1) -- (10,0.5) -- (7.5,-1) -- (6.5,-3) -- (5.5,-1) -- (3,0);
\draw (6,2) -- (8,1) -- (7.5,-1) -- (5.5,-1) -- (6,2);
\draw (6.5,0) -- (6,2);
\draw (6.5,0) -- (8,1);
\draw (6.5,0) -- (7.5,-1);
\draw (6.5,0) -- (5.5,-1);
\draw [color = red] (3,0) -- (5.75, 0.5);
\draw [color = red] (6,2) -- (4.25,-0.5);
\draw [color = red] (5.5,-1) -- (4.5, 1);
\draw [color = red] (5.75,0.5) -- (6.5,0);
\draw [color = red] (5.5,-1) -- (6.25,1);
\draw [color = red] (6,2) -- (6,-0.5);
\draw [color = red] (6,2) -- (7.5,2.5);
\draw [color = red] (7,4) -- (7, 1.5);
\draw [color = red] (8,1) -- (6.5,3);
\draw [color = red] (6.5,0) -- (7,1.5);
\draw [color = red] (6,2) -- (7.25,0.5);
\draw [color = red] (8,1) -- (6.25,1);
\draw [color = red] (8,1) -- (8.75, -0.25);
\draw [color = red] (10,0.5) -- (7.75, 0);
\draw [color = red] (7.5,-1) -- (9,0.75);
\draw [color = red] (6.5,0) -- (7.75,0);
\draw [color = red] (7.5,-1) -- (7.25,0.5);
\draw [color = red] (8,1) -- (7, -0.5);
\draw [color = red] (5.5,-1) -- (7,-2);
\draw [color = red] (6.5,-3) -- (6.5,-1);
\draw [color = red] (7.5,-1) -- (6,-2);
\draw [color = red] (6.5,0) -- (6.5,-1);
\draw [color = red] (5.5,-1) -- (7,-0.5);
\draw [color = red] (7.5,-1) -- (6,-0.5);
\end{tikzpicture}
\caption{A planar triangulation $S$ and its barycentric subdivision $T$.}\label{fig:3}
\end{figure}
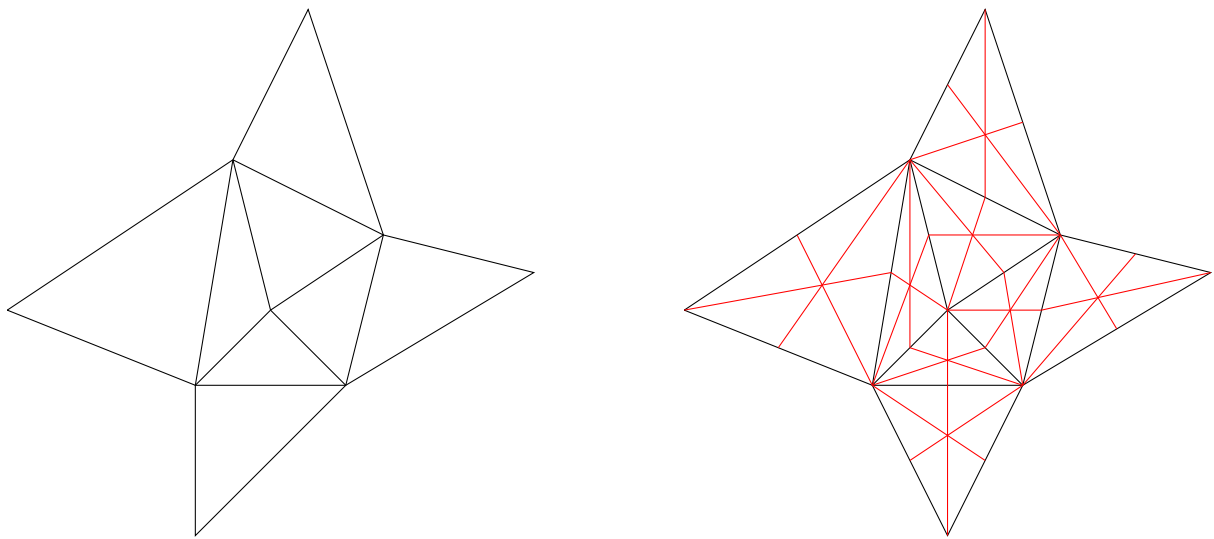

\begin{figure}
\centering
\begin{tikzpicture}
\fill[gray] (-2.75,1) -- ({-6.5/3},1) -- (-1.75,0.5) -- ({-5/3},0) -- (-2,-0.5) -- ({-7.5/3},{-2/3}) -- (-3,-0.5) -- (-3,{1/3}) -- (-2.75,1);

\draw (-6,0) -- (-3,2) -- (-2,4) -- (-1,1) -- (1,0.5) -- (-1.5,-1) -- (-2.5,-3) -- (-3.5,-1) -- (-6,0);
\draw (-3,2) -- (-1,1) -- (-1.5,-1) -- (-3.5,-1) -- (-3,2);
\draw (-2.5,0) -- (-3,2);
\draw (-2.5,0) -- (-1,1);
\draw (-2.5,0) -- (-1.5,-1);
\draw (-2.5,0) -- (-3.5,-1);
\draw (-6,0) -- (-3.25, 0.5);
\draw (-3,2) -- (-4.75,-0.5);
\draw (-3.5,-1) -- (-4.5, 1);
\draw (-3.25,0.5) -- (-2.5,0);
\draw (-3.5,-1) -- (-2.75,1);
\draw (-3,2) -- (-3,-0.5);
\draw (-3,2) -- (-1.5,2.5);
\draw (-2,4) -- (-2, 1.5);
\draw (-1,1) -- (-2.5,3);
\draw (-2.5,0) -- (-2,1.5);
\draw (-3,2) -- (-1.75,0.5);
\draw (-1,1) -- (-2.75,1);
\draw (-1,1) -- (-0.25, -0.25);
\draw (1,0.5) -- (-1.25, 0);
\draw (-1.5,-1) -- (0,0.75);
\draw (-2.5,0) -- (-1.25,0);
\draw (-1.5,-1) -- (-1.75,0.5);
\draw (-1,1) -- (-2, -0.5);
\draw (-3.5,-1) -- (-2,-2);
\draw (-2.5,-3) -- (-2.5,-1);
\draw (-1.5,-1) -- (-3,-2);
\draw (-2.5,0) -- (-2.5,-1);
\draw (-3.5,-1) -- (-2,-0.5);
\draw (-1.5,-1) -- (-3,-0.5);

\fill[gray] (6,2) -- (7,{7/3}) -- (8,1) -- ({20.5/3},1) -- (6,2);

\draw (3,0) -- (6,2) -- (7,4) -- (8,1) -- (10,0.5) -- (7.5,-1) -- (6.5,-3) -- (5.5,-1) -- (3,0);
\draw (6,2) -- (8,1) -- (7.5,-1) -- (5.5,-1) -- (6,2);
\draw (6.5,0) -- (6,2);
\draw (6.5,0) -- (8,1);
\draw (6.5,0) -- (7.5,-1);
\draw (6.5,0) -- (5.5,-1);
\draw (3,0) -- (5.75, 0.5);
\draw (6,2) -- (4.25,-0.5);
\draw (5.5,-1) -- (4.5, 1);
\draw (5.75,0.5) -- (6.5,0);
\draw (5.5,-1) -- (6.25,1);
\draw (6,2) -- (6,-0.5);
\draw (6,2) -- (7.5,2.5);
\draw (7,4) -- (7, 1.5);
\draw (8,1) -- (6.5,3);
\draw (6.5,0) -- (7,1.5);
\draw (6,2) -- (7.25,0.5);
\draw (8,1) -- (6.25,1);
\draw (8,1) -- (8.75, -0.25);
\draw (10,0.5) -- (7.75, 0);
\draw (7.5,-1) -- (9,0.75);
\draw (6.5,0) -- (7.75,0);
\draw (7.5,-1) -- (7.25,0.5);
\draw (8,1) -- (7, -0.5);
\draw (5.5,-1) -- (7,-2);
\draw (6.5,-3) -- (6.5,-1);
\draw (7.5,-1) -- (6,-2);
\draw (6.5,0) -- (6.5,-1);
\draw (5.5,-1) -- (7,-0.5);
\draw (7.5,-1) -- (6,-0.5);

\end{tikzpicture}
\caption{The shaded area on the left depicts one of the sets $U^0_i$, whereas the shaded area on the right depicts one of the sets $U^1_j$.}\label{fig:4}
\end{figure}
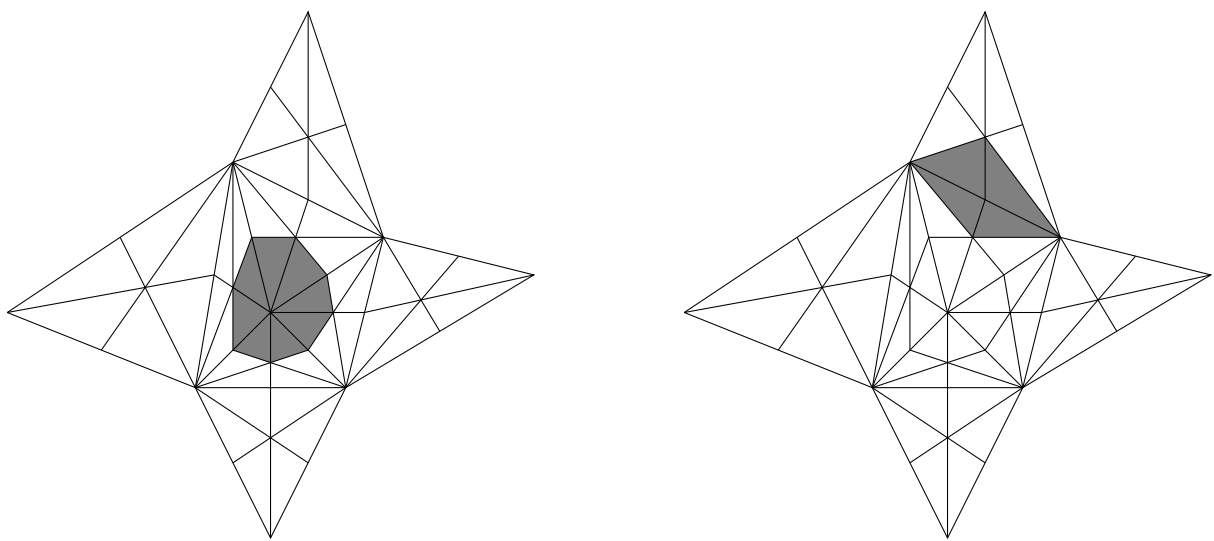

The family $\mathscr{C}$ is obviously an open cover of the manifolds which satisfies (a) and (e) by construction. If two distinct elements $U^s_i$ and $U^t_j$ have nonempty intersection and $s\geq t$, then $s>t$ and $S^t_j$ is a face of $S^s_i$: this implies that $\mathscr{C}$ satisfies (c). Statement (d) is an obvious consequence of (e). Each $U_s^j$ is diffeomorphic to the open Euclidean $n$-dimensional ball, but its closure is only homeomorphic to the closed ball: however, it suffices to choose an appropriate smaller open set for each $U^j_s$ to achieve finally an open cover which satisfies (b) and retains the other four properties. 
\end{proof}

From now on we fix therefore a smooth manifold $\Sigma$ as in Theorem~\ref{t:main_C1_2} and a corresponding smooth atlas $\mathcal{A} = \{U_\ell\}$ (which is either finite or countably infinite) where the $U_\ell$'s have compact closure and satisfy the properties (b), (c), and (d) of Lemma~\ref{l:covering}. 

Given any symmetric $(0,2)$ tensor $h$ on $\Sigma$ we write $h = h_{ij} dx_i \otimes dx_j$ and denote by $\|h\|_{0,U_\ell}$ the supremum of the Hilbert--Schmidt norm of the matrices $h_{ij} (p)$ for $p\in U_\ell$. Similarly, if $v: \Sigma \to \mathbb R^N$ is a $C^1$ map, we write $\|Dv\|_{0, U_\ell}$ for the supremum of the Hilbert--Schmidt norms of the matrices $Dv (p) = (\partial_1 v (p), \ldots , \partial_n v (p))$, where $p\in U_\ell$. Finally, we set
\begin{align*}
\|h\|_0 & := \sup_{\ell} \|h\|_{0, U_\ell}\, ,\\
\|Dv\|_0 &:= \sup_{\ell} \|Dv\|_{0, U_\ell}\, .
\end{align*}
We are now ready to state the main inductive statement\footnote{This is what Nash calls ``a stage'', cf.~\cite[p.~391]{Nash1954}} whose iteration will prove Theorem~\ref{t:main_C1_2}.

\begin{proposition}[Iteration stage]\label{p:iter_C1}
Let $(\Sigma ,g)$ be as in Theorem~\ref{t:main_C1_2} and $w: \Sigma \to \mathbb R^N$ a smooth strictly short immersion. For any choice of positive numbers $\eta_\ell >0$ and any $\delta>0$ there is a smooth short immersion $z: \Sigma \to \mathbb R^N$ such that 
\begin{align}
\|z-w\|_{0, U_\ell} &< \eta_\ell\, \qquad\forall \ell\, ,\label{e:iter_1}\\ 
\|g - z^\sharp e\|_0 & < \delta\, ,\label{e:iter_2}\\
\|Dw - Dz\|_0 & < C \sqrt{\|g - w^\sharp e\|_0}\, ,\label{e:iter_3}
\end{align} 
for some dimensional constant $C$. If $w$ is injective, then we can choose $z$ injective.
\end{proposition}

Note that the right-hand side of \eqref{e:iter_3} might be $\infty$ (because $\Sigma$ is not necessarily compact), in which case the condition \eqref{e:iter_3} is an empty requirement.
We show first how to conclude Theorem~\ref{t:main_C1_2} from the proposition above. Subsequently we close this section by proving Corollary~\ref{c:Nash+Whitney}. The rest of the chapter will then be dedicated to prove Proposition~\ref{p:iter_C1}.

\begin{proof}[Proof of Theorem~\ref{t:main_C1_2}]
Let $v_0:= v$ and $\varepsilon$ be as in the statement and assume for the moment that $v$ is an immersion. Moreover, without loss of generality we can assume that $v$ is strictly short:  it suffices to multiply $v$ by a constant smaller than (but sufficiently close to) $1$. Note that such operation will change the limit set of $v$, which explains why in the last claim of the theorem we assume directly that $v$ is strictly short.

We will produce a sequence of maps $v_q$ by applying iteratively Proposition~\ref{p:iter_C1}. Since the limit set of $v$ is closed and $v (\overline{U_\ell})$ compact, there is a positive number $\beta_\ell$ such that any point of $v (\overline U_\ell)$ is at distance at least $\beta_\ell$ from the limit set of $v$. We then define the numbers 
\begin{align*}
\bar \eta_{q,\ell} & :=  2^{-q-1} \min\{\varepsilon, \beta_\ell, 2^{-\ell}\}\, ,\\
\delta_q &:= 4^{-q}\, .  
\end{align*}
At each $q\geq 1$ we apply Proposition~\ref{p:iter_C1} with $w = v_{q-1}$, $\eta_\ell = \bar \eta_{q,\ell}$, and $\delta = \delta_q$ to
produce $z=: v_q$. We then conclude immediately that:
\begin{itemize}
\item[(a)] $\|v_{q} - v_{q-1}\|_0\leq 2^{-q-1} \varepsilon$ and thus $v_q$ converges uniformly to some $u$ with $\|u-v\|_0 \leq \varepsilon \sum_{q\geq 1} 2^{-q-1} = \frac{\varepsilon}{2}$;
\item[(b)] similarly $\|v-u\|_{0, U_\ell} \leq \beta_\ell \sum_{q\geq 1} 2^{-q-1} = \frac{\beta_\ell}{2}$;
\item[(c)] again by a similar computation $\|u-v\|_{0, U_\ell} \leq 2^{-\ell}$ and thus the limit set of $u$ coincides with the limit set of $v$; combined with the estimate above, this implies that the limit set of $u$ does not intersect the image of $u$;
\item[(d)] $\|Dv_{q}- Dv_{q-1}\|_0 \leq C 2^{-q+1}$ for every $q\geq 2$ and thus $u$ is a $C^1$ map (observe that we claim no bound on $\|Dv_1 - Dv_0\|_0$; on the other hand we do not need it!);
\item[(e)] since $v_q$ converges to $u$ in $C^1$, we have $g - u^\sharp e = \lim_q (g-v_q^\sharp e) = 0$ and thus $u$ is an isometry, from which we also conclude that the differential of $u$ has everywhere full rank and hence $u$ is an immersion.
\end{itemize}
It remains to show that, if $v$ is injective, then the iteration above can be arranged so to guarantee that $u$ is also injective.
To this aim, notice first that all the conclusions above certainly hold in case we implement the same iteration applying Proposition~\ref{p:iter_C1} with parameters
$\eta_{q, \ell}$ smaller than $\bar \eta_{q,\ell}$. Moreover the proposition guarantees the injectivity at each step: we just need to show that the limit map is also injective. For each $q$ consider the compact set $V_q := \cup_{\ell \leq q} \overline{U}_\ell$ and the positive numbers
\[
2 \gamma_i := \min \{ |v_i (x)- v_i (y)|: d (x,y) \geq 2^{-i}, x,y \in V_i \}\,  \qquad \mbox{for $i<q$,}
\]
where $d$ is the geodesic distance induced by the Riemannian metric $g$.
We then set $\eta_{q, \ell} := \min \{ \bar \eta_{q, \ell}, 2^{-q-1} \gamma_1, 2^{-q-1} \gamma_2, \ldots , 2^{-q-1} \gamma_{q-1}\}$ and apply the iteration as above with $\eta_{q,\ell}$ in place of $\bar\eta_{q,\ell}$. We want to check that the resulting $u$ is injective. Fix $x\neq y$ in $\Sigma$ and choose $q$ such that $2^{-q} \leq d (x,y)$ and $x,y\in V_q$. We can then estimate
\[
|u(x)- u(y)|\geq |v_q (x)- v_q (y)| - \sum_{k\geq q} \|v_{k+1}-v_k\|_{0, V_q} \geq 2 \gamma_q - \sum_{k\geq q} 2^{-k-1} \gamma_q \geq \gamma_q>0\, .
\]
Hence $u(x)\neq u(y)$. The arbitrariness of $x$ and $y$ shows that $u$ is injective and completes the proof.\end{proof}

\begin{proof}[Proof of Corollary~\ref{c:Nash+Whitney}]
Recall that, according to Whitney's embedding theorem in its strong form (see \cite{Whitney1944}), any smooth differentiable manifold $\Sigma$ of dimension $n$ can be embedded in $\mathbb R^{2n}$. If the manifold in addition is closed, then it suffices to multiply the corresponding map by a sufficiently small positive constant to make it short and the existence of a nearby $C^1$ isometry with the desired property follows from Theorem~\ref{t:main_C1_1}. 

The general case requires somewhat more care.
Fix a smooth Riemannian manifold $(\Sigma, g)$ of dimension $n$, {\em not closed}.
Below we will produce a suitable smooth embedding $z: \Sigma \to \mathbb R^N$ for $N= (n+1) (n+2)$, with the additional properties that
\begin{itemize}
\item[(i)] $z$ is a short map;
\item[(ii)] the limit set of $z$ is $\{0\}$ and does not intersect the image of $z$.
\end{itemize}
We then can follow the standard procedure of the proof of the Whitney's embedding theorem in its weak form (cf.~\cite{Whitney1936}): if we consider the Grassmannian of $2n+1$ dimensional planes $\pi$ of $\mathbb R^N$, we know that, for a subset of full measure, the projection $P_\pi$ onto $\pi$ is injective and has injective differential on $z (\Sigma)$. A similar argument shows that, for a set of planes $\pi$ of full measure, $P_\pi (z (\Sigma))$ does not contain the origin. Since clearly $P_\pi\circ z$ is also short, the map $v:= P_\pi\circ z$ satisfies  the assumptions of Theorem~\ref{t:main_C1_2}. If we drop the injectivity assumption on $\pi$ (namely we restrict to immersions), we can project on a suitable $2n$-dimensional plane.

Coming to the existence of $z$, we use the atlas $\{U_\ell\}$ of $\Sigma$ given by Lemma~\ref{l:covering} and we let $\Phi_\ell : U_\ell \to \mathbb R^n$ be the corresponding charts. Observe that, since $\Sigma$ is not closed, the atlas is necessarily (countably) infinite. After further multiplying each $\Phi_\ell$ by a positive scalar we can assume, without loss of generality, that $|\Phi_\ell|\leq 1$. Recall the $n+1$ classes $\mathcal{C}_i$ of Lemma~\ref{l:covering}(e). Consider then a family of smooth functions $\varphi_\ell$, each supported in $U_\ell$, with $0\leq \varphi_\ell \leq 1$ and such that for any point $p\in \Sigma$ there is at least one $\varphi_\ell$ which is equal to $1$ in some neighborhood of $p$. Finally, after numbering the elements of the atlas, we fix a vanishing sequence $\varepsilon_\ell$ of strictly monotone positive numbers, whose choice will be specified in a moment. 

We are now ready to define our map $z$, which will be done specifying each component $z_j$. Fix $p\in \Sigma$ and $i\in \{1, \ldots , n+1\}$. If $p$ does not belong to any element of $\mathcal{C}_i$, then we set $z_{(i-1) (n+2) +1} (p) = \ldots = z_{i(n+2)} (p) =0$. Otherwise, there is a unique $U_\ell\in \mathcal{C}_i$ with $p\in U_\ell$ and we set:
\begin{align}
z_{(i-1) (n+2) +j} (p) & = \varepsilon_\ell^2 \varphi_\ell (p) (\Phi_\ell (p))_j \qquad \mbox{for $j\in \{1, \ldots , n\}$,}\\
z_{(i-1) (n+2) +n+1} (p) &= \varepsilon_\ell^2 \varphi_\ell (p)\, ,\\
z_{(i-1) (n+2) + n+2} (p) &= \varepsilon_\ell \varphi_\ell (p)\, .
\end{align}
Now, for any point $p$ there is at least one $\ell$ for which $\varphi_\ell$ is identically equal to $1$ in a neighborhood of $p$: this will have two effects, namely that the differential of $z$ at $p$ is injective and that $z (p) \neq 0$. Since the limit set of $z$ is obviously $0$, condition (ii) above is satisfied. To prove that $z$ is an embedding we need to show that $z$ is injective. Fix two points $p$ and $q$
and fix a $U_\ell \in \mathcal{C}_i$ for which $\varphi_\ell (p)=1$. If $q\in U_\ell$, then either $\varphi_\ell (q) \neq 1$, in which case $z_{(i-1) (n+2) +n+1} (p) \neq z_{(i-1) (n+2) + n+1} (q)$, or $\varphi_\ell (q) =1$. In the latter case we then conclude $z (q) \neq z(p)$ because $\Phi_\ell (p)\neq \Phi_\ell (q)$. 
If $q\not \in U_\ell$ and $\varphi_{\ell'} (q) =0$ for any other $U_{\ell'}\in \mathcal{C}_i$, then $z_{(i-1) (n+2) +n+1} (q) = 0 \neq z_{(i-1) (n+2) + n +1} (p)$. Otherwise there is a $U_{\ell'}\in \mathcal{C}_i$ distinct from $U_\ell$ such that $\varphi_{\ell'} (q) \neq 0$. In this case we have
\[
\frac{z_{(i-1)(n+2) + n+1} (p)}{z_{i(n+2)} (p)} = \varepsilon_{\ell} \neq \varepsilon_{\ell'} = \frac{z_{(i-1)(n+1) + n+1} (q)}{z_{i(n+2)} (q)}\, .
\]
Thus $z$ is injective.

Finally, by choosing the $\varepsilon_\ell$ inductively appropriately small, it is easy to show that we can ensure the shortness of $z$. 
\end{proof}

\section{Decomposition in primitive metrics}

We will call ``primitive metric''\footnote{Although the term is nowadays rather common, it was not introduced by Nash, neither in \cite{Nash1954} nor in the subsequent paper \cite{Nash1956}.} any $(0,2)$ tensor having the structure $a^2 d\psi \otimes d\psi$ for some pair of smooth functions $a$ and $\psi$. Note that such two tensor is only positive semidefinite and thus it is certainly not a Riemannian metric. The next fundamental lemma shows that any Riemannian metric can be written as a (locally finite) sum of primitive metrics satisfying some additional technical requirements.

\begin{proposition}\label{p:decomp}
Let $\Sigma$ be a smooth $n$-dimensional manifold, $h$ a smooth positive definite $(0,2)$ tensor on it and $\{U_\ell\}$ a cover of $\Sigma$. Then there is a countable collection $h_j$ of primitive metrics such that $h = \sum_j h_j$ and
\begin{itemize}
\item[(a)] Each $h_j$ is supported in some $U_\ell$.
\item[(b)] For any $p\in \Sigma$ there are at most\footnote{In his paper Nash claims indeed a much larger $K(n)$, cf.~\cite[bottom of p.~386]{Nash1954}.} $K (n) = \frac{n (n+1)^2}{2}$ $h_j$'s whose support contains $p$. 
\item[(c)] The support of each $h_j$ intersects the supports of at most finitely many other $h_k$'s.
\end{itemize}
\end{proposition}
\begin{proof} First of all, for each point $p\in \Sigma$ we find a neighborhood $V_p\subset U_\ell$ (for some $\ell$) and $J (n) = \frac{n(n+1)}{2}$ primitive metrics $h_{p1}, \ldots, h_{pJ}$ on $V_p$ such that $h = h_{p1}+ \ldots + h_{p}$. In order to do this fix coordinates on $U_\ell \ni p$ and write $h$ as $h = h_{ij} dx_i \otimes dx_j$. Consider the space ${\rm Sym}_{n\times n}$ of symmetric $n\times n$ matrices and let $M$ be the matrix with entries $h_{ij} (p)$. Now, since the set of all matrices of the form $v\otimes v$ is a linear generator of ${\rm Sym}_{n\times n}$, there are $J$ such matrices $A'_i = w_i\otimes w_i$ which are linearly independent. Consider $M' := \sum_i A'_i$. By standard linear algebra we can find a linear isomorphism $L$ of $\mathbb R^n$ such that $L^T M' L=M$: indeed, since both $M$ and $M'$ are symmetric we can find $O$ and $O_1$ orthogonal such that
$D = O^T M O$ and $D_1 = O_1^T M' O_1$ are diagonal matrices. Since $M$ and $M'$ are both positive definite, the entries of $D$ and $D_1$ are all positive. Let therefore $D^{- \sfrac{1}{2}}$ and $D_1^{- \sfrac{1}{2}}$ be the diagonal matrices whose entries are the reciprocal of the square roots of the entries of $D$ and $D_1$, respectively. If we set $U:= O D^{- \sfrac{1}{2}}$ and $U_1 := O_1 D_1^{- \sfrac{1}{2}}$, then clearly
$U^T M U = U_1^T M' U_1$ is the identity matrix. Thus $L:= U_1 U^{-1}$ is the linear isomorphism we were looking for.
Having found $L$, if we set $A_i = L^T A'_i L = (Lw_i)\otimes (Lw_i) = v_i\otimes v_i$, we conclude that $M = \sum_i A_i$. 

Next, there are unique linear maps $\mathcal{L}_i : {\rm Sym}_{n\times n} \to \mathbb R$ such that $A = \sum_i \mathcal{L}_i (A) v_i\otimes v_i$ for every $A$. Thus, if we consider the maps $\psi_i (x) = v_i \cdot x$ in local coordinates, we find smooth functions $\alpha_i : U_\ell \to \mathbb R$ such that 
\[
h = \sum_{i=1}^J \alpha_i d\psi_i \otimes d\psi_i\, .
\] 
Note that $\alpha_i (p) = \mathcal{L}_i (M) =1$ for every $i\in \{1, \ldots , J\}$ and thus in a neighborhood $V_p$ of $p$ each $\alpha_i$ is the square of an appropriate smooth function $a_i$. The tensors $h_{pi} := a_i^2 d\psi_i\otimes d\psi_i$ are the required primitive metrics.

Finally we apply Lemma~\ref{l:covering} and refine the covering $V_p$ to a new covering $W_\ell$ with the properties listed in the lemma. For each $W_\ell$ we consider a $V_p\supset W_\ell$ and define the corresponding primitive metrics $h_{(\ell1)} = h_{p1}, \ldots , h_{(\ell J)}= h_{pJ}$ (we use the subscript $(\ell j)$ in order to avoid confusions with the explicit expression of the initial tensor $h$ in a given coordinate system!). We then consider compactly supported functions $\beta_\ell \in C^\infty_c (W_\ell)$ with the property that for any point $p$ there is at least a $\beta_\ell$ which does not vanish at $p$ and we set
\[
\varphi_\ell := \frac{\beta_\ell}{\sqrt{\sum_j \beta_j^2}}\, .
\]
The tensors $\varphi_\ell^2 h_{(\ell j)}$ satisfy all the requirements of the proposition.\footnote{The argument of Nash is slightly different, since it covers the space of positive definite matrices with appropriate simplices.}
\end{proof}

\section{Proof of the main iterative statement}

To complete the proof of the Proposition~\ref{p:iter_C1} we still need one technical ingredient.

\begin{lemma}\label{l:NB}
Let $B$ be a closed subset of $\mathbb R^n$ diffeomorphic to the $n$-dimensional closed ball and $\omega:B \to \mathbb R^N$ a smooth immersion with $N\geq n+2$. Then there are two smooth maps $\nu,b: B \to \mathbb R^N$ such that
\begin{itemize}
\item[(a)] $|\nu(q)|=|b(q)|=1$ and $\nu (q)\perp b (q)$ for every $q\in B$;
\item[(b)] $\nu(q)$ and $b(q)$ are both orthogonal to $T_{\omega (q)} (\omega (B))$ for every $q\in B$.
\end{itemize}
\end{lemma}

\begin{proof}
For any point $p$ there exists a neighborhood of it and a pair of maps as above defined on the neighborhood: first
select two orthonormal vectors $\nu (p)$ and $b(p)$ which are normal to $T_{\omega (p)} (\omega(B))$ and, by smoothness of $\omega$, observe that
they are almost orthogonal to $T_{\omega (q)} (\omega (B))$ for every $q$ in a neighborhood of $p$. By first projecting on the normal bundle and then using the standard Gram--Schmidt orthogonalization procedure we then produce the desired pair. The problem of passing from the local statement to the global one can be translated into the existence of a suitable section of a fiber bundle: since $B$ is topologically trivial, this is a classical conclusion.\footnote{Nash cites Steenrod's classical book, \cite{Steenrod}.} 

However, one can also use the following elementary argument.\footnote{Nash writes {\em Also they could be obtained by orthogonal propagation}, cf.~\cite[top of p.~387]{Nash1954}.} We first observe that it suffices to produce $\nu$ and $b$ continuous: we can then smooth them by convolution, project on the normal bundle, and use again a Gram--Schmidt procedure to produce a pair with the desired properties. We just have to ensure that the projection on the normal bundle still keeps the two vectors linearly independent at each point. Since $\nu$ and $b$ are orthonormal and orthogonal to $\omega (B)$, this is certainly the case if the smoothings are $\varepsilon$-close to them in the uniform topology, where $\varepsilon >0$ is a fixed geometric constant. Next, in order to show the existence of a continuous pair with properties (a) and (b), assume without loss of generality that $B=\overline{B}_1 (0)\subset \mathbb R^n$ and consider the set $R$ of all radii $r$ for which there is at least one such pair on $\overline{B}_r (0)$. As observed above $R$ is not empty. Let $\rho$ be the supremum of $R$: we claim that $\rho\in R$. Indeed choose $\rho_k\in R$ with $\rho_k \uparrow \rho$ and let $\nu_k, b_k$ be two corresponding continuous maps on $\overline{B}_{\rho_k} (0)$ satisfying (a) and (b).
We define $\tilde{\nu}_k$ and $\tilde{b}_k$ on $B_1$ by setting them equal to $\nu_k$ and $b_k$ on $B_{\rho_k} (0)$ and extending them further by
\[
\tilde{\nu}_k (x)= \nu_k \left(\rho_k \frac{x}{|x|}\right) \qquad \mbox{and}\qquad \tilde{b}_k (x) = b_k \left(\rho_k \frac{x}{|x|}\right)\qquad\mbox{for $|x|\geq \rho_k$.}
\] 
Note that the two maps satisfy (a). As for (b), by the smoothness of $\omega$, for any $\eta>0$ there is $\delta>0$ such that, if $|x|\leq \rho_k +\delta$, then the angle between $\tilde{\nu}_k (x)$ (resp. $\tilde{b}_k (x)$) and the tangent space $T_{\omega (x)} \omega (B)$) is at least $\frac{\pi}{2}-\eta$. On the other hand, once $\eta$ is smaller than a geometric constant, we can project $\tilde{\nu}_k$ and $\tilde{b}_k$ on the normal bundle and apply Gram--Schmidt to produce a continuous pair which satisfies the desired requirements on $\overline{B}_{\sigma} (0)$ for $\sigma_k = \min\{1, \rho_k +\delta\}$. 
Thus $\sigma_k$ belongs to $R$. By definition $\rho \geq \sigma_k$ for every $k$: letting $k\uparrow \infty$ and using that $\rho_k\uparrow \rho$, we conclude $\rho\geq \min \{1, \rho +\delta\}$, namely $\rho =1$. Thus $\sigma_k =1$ for $k$ large enough, which implies $1\in R$ and concludes the proof. 
\end{proof}

\begin{proof}[Proof of Proposition~\ref{p:iter_C1}]
Fix a partition of unity $\varphi_\ell$ subordinate to $U_\ell$. 
Now, each fixed $U_\ell$ intersects a finite number of other $U_j$'s: denote the set of relevant indices by $I (\ell)$. We can therefore choose $\delta_\ell>0$ in such a way that $(1-\delta_\ell) g - w^\sharp e$ is positive definite and 
\begin{equation}\label{e:C1_uno}
\|\delta_\ell g \|_{0, U_j} < \frac{\delta}{2}  \qquad \mbox{for every $j\in I (\ell)$.}
\end{equation}
Construct now the function $\varphi := \sum_\ell \delta_\ell \varphi_\ell$ and set $h:= (1-\varphi) g - w^\sharp e$. Clearly
\begin{equation}\label{e:C1_due}
\|g - (h+w^\sharp e)\|_0 < \frac{\delta}{2}\, 
\end{equation}
and
\begin{equation}\label{e:C1_tre}
g - (h+w^\sharp e) > 0\, .
\end{equation}
In particular, if we choose $\delta'_\ell$ appropriately and we impose that the final map $z$ satisfies
\begin{equation}\label{e:C1_quattro}
\|z^\sharp e - (w^\sharp e +h)\|_{0, U_\ell} < \delta'_\ell\qquad \mbox{for every $\ell$},
\end{equation}
we certainly conclude that $z$ is short and satisfies \eqref{e:iter_2}. Moreover, we will impose the stronger condition
\begin{equation}\label{e:C1_cinque}
\|Dw - Dz\|^2_{0, U_\ell} <2 K(n)^2 \|g - w^\sharp e\|_{0, U_\ell}
\end{equation}
in place of \eqref{e:iter_3}, where $K(n)$ is the constant in Proposition~\ref{p:decomp}.
Hence from now on we focus on producing a map $z$ satisfying the local conditions \eqref{e:iter_1}, \eqref{e:C1_quattro}, and 
\eqref{e:C1_cinque}.

Next, we apply Proposition~\ref{p:decomp} to write $h = \sum_j h_j$, where each $h_j$ is a primitive metric and is supported in some $U_\ell$. We assume the index $j$ starts with
$1$ and follows the progression of natural numbers (note that the $h_j$'s are either finite or countably infinite).
Recall, moreover, that at any point of $\Sigma$ at most $K(n)$ of the $h_j$'s are nonzero and that, for any fixed $j$, only finitely many $U_\ell$ intersect the
support of $h_j$, since the latter is a compact set: the corresponding set of indices will be denoted by $L(j)$. 
We next order the $h_j$'s and we inductively add to the map $w$ a smooth ``perturbation'' map $w^p_j$, whose support coincides with that of $h_j$. 
If we let $w_j := w+ w^p_1 + \ldots + w^p_j$ be the ``resulting map'' after $j$ steps, we then claim the following estimates:
\begin{align}
\|w^p_j\|_{0, U_\ell}< &\; \frac{\eta_\ell}{K(n)} \;\;\quad \qquad \mbox{for all $\ell\in L (j)$,}\label{e:induct_1}\\
\|Dw^p_j\|_{0, U_\ell}^2 < &\; 2 \|h\|_{0, U_\ell} \quad\qquad \mbox{for all $\ell\in L (j)$,}\label{e:induct_2}\\
\|w_j^\sharp e - (w_{j-1}^\sharp e +h_j)\|_{0, U_\ell} < &\; \frac{\delta'_\ell}{K(n)} \quad\quad \qquad \mbox{for all $\ell\in L (j)$\label{e:induct_3}.}
\end{align}
We will prove below the existence of $w^p_j$, whereas we first show how to conclude. We set $z = w + \sum_j w^p_j$. Fix any $U_\ell$ and any point $q\in U_\ell$. Observe that, since 
$\overline{U}_\ell$ is compact, only finitely many perturbations $w^p_j$ are nonzero in $U_\ell$ and thus $z$ is smooth in $U_\ell$. Next, note that at most $K(n)$ $h_j$'s (and hence at most $K(n)$ $w^p_j$'s) are nonzero at $q$. Thus we can sum up all the estimates in \eqref{e:induct_1} and \eqref{e:induct_2} to conclude
\begin{align}
|w (q) - z(q)| \leq&\; \sum_j \|w^p_j\|_{0, U_\ell} <  \eta_\ell\, ,\\
|Dw (q) - Dz (q)| \leq &\; \sum_j \|Dw^p_j\|_{0, U_\ell} < \sqrt{2} K(n) \|h\|_{0, U_\ell}
< \sqrt{2} K(n)\|g-w^\sharp e\|_{0, U_\ell}\, ,
\end{align}  
where in the last inequality we can use \eqref{e:C1_tre}.
Finally, we write
\begin{align}
z^\sharp  e - (w^\sharp e+h) = z^\sharp e - w^\sharp e - \sum_j h_j = \sum_{j\geq 1} (w_j^\sharp e - (w_{j-1}^\sharp e + h_j))
\end{align}
(where $w_0:= w$)
and thus we can use \eqref{e:induct_3} to conclude, at the point $q$ and using the coordinate pach $U_\ell$,
\[
|(z^\sharp e - (g+h)) (q)| < \delta'_\ell\, .
\]
This completes the proof of \eqref{e:iter_1}, \eqref{e:C1_quattro} and \eqref{e:C1_cinque}. 

In order to define $w^p_j$, select a $U_\ell$ and apply Lemma~\ref{l:NB} on $U_\ell$ with $\omega = w_{j-1}$ to find two orthonormal smooth vector
fields $\nu, b: U_\ell \to \mathbb R^N$ with the property that $\nu$ and $b$ are normal to $w_{j-1} (U_\ell)$. Recall that $h_j = a_j^2 d\psi_j\otimes d\psi_j$ and set
\[
w^p_j (x) = a_j (x) \frac{\nu (x)}{\lambda} \cos \lambda \psi_j (x) + a_j (x) \frac{b (x)}{\lambda} \sin \lambda\psi_j (x)\, ,
\]
where $\lambda$ is a positive parameter, which will be chosen very large.

Note first that \eqref{e:induct_1} is obvious provided $\lambda$ is large enough. Next compute, in the coordinate patch $U_\ell$, 
\[
Dw^p_j (x) = \underbrace{- a_j(x) \sin \lambda \psi_j (x)\,\nu (x)\otimes d\psi_j (x)}_{A (x)} + \underbrace{a_j (x) \cos \lambda \psi_j (x)\, b(x)\otimes d\psi_j (x)}_{B (x)} + E (x)\, ,
\]
where $|E (x)|\leq C_{j-1} \lambda^{-1}$, for a constant $C_{j-1}$ which depends on the smooth functions $a_j$, $\psi_j$, $b$ and $\nu$, but not on $\lambda$ (note that in the line
above we understand all summands as $N\times n$ matrices).
We then obviously have
\[
|Dw^p_j (x)|^2 \leq a_j (x)^2 |d\psi_j (x)|^2 + C_{j-1} \lambda^{-1} \leq \|h_j\|_{0, U_\ell} + C_{j-1} \lambda^{-1} \leq \|h\|_{0, U_\ell} + C_{j-1} \lambda^{-1}\, 
\]
(here and in what follows, $C_{j-1}$ denotes constants which might change from line to line but are independent of the parameter $\lambda$).
Since $\|h\|_{0, U_\ell}$ is positive, it suffices to choose $\lambda$ large enough to achieve \eqref{e:induct_2}. 

Next write the tensor $\bar h:= w_j^\sharp e - w_{j-1}^\sharp e$ in coordinates as $\bar h = \bar h_{ik} dx_i\otimes dx_k$ and observe that the $\bar h_{ik}$ are simply the entries
of the symmetric matrix
\[
Dw_j^T Dw_j - Dw_{j-1}^T Dw_{j-1}\, .
\]
Recall that $Dw_j = Dw_{j-1} + A + B + E$. By the conditions on $\nu$ and $b$ we have
\[
0 = A^T B = B^T A = A^T Dw_{j-1} = Dw_{j-1}^T A = B^T Dw_{j-1} = Dw_{j-1}^T B\, .
\] 
We thus conclude that
\[
|Dw_j^T Dw_{j-1} - Dw_{j-1}^T Dw_j - (A^T A + B^T B)| \leq C_{j-1} \lambda^{-1}\, .
\]
On the other hand,
\[
A^T A + B^T B = a_j^2 (\cos^2 \lambda \psi_j + \sin^2 \lambda \psi_j) d\psi_j \otimes d\psi_j = a_j^2 d\psi_j \otimes d\psi_j = h_j\, .
\] 
Hence \eqref{e:induct_3} follows
at once for $\lambda$ large.

It must be noticed that so far we have shown \eqref{e:induct_1}, \eqref{e:induct_2}, and \eqref{e:induct_3} only for the chosen coordinate patch which contains the support of 
$h_j$, whereas the estimates are claimed in all coordinate patches which intersect the support of $h_j$. On the other hand, on these other coordinate patches the same computations yield the same estimates, and since there are only finitely many such patches to take into account, our claims readily follow for an appropriate choice of $\lambda$.

It remains to show that, if $w$ is injective, then $z$ too can be chosen to be injective. Fix $p, q\in \Sigma$. For $j$ sufficiently large we have $z (p) = w_j (p)$ and $z(q) = w_j (q)$. Thus it suffices to show the injectivity of $w_j$. We will show, inductively on $j$, that this can be achieved by choosing $\lambda$ sufficiently large. Thus assume that $w_{j-1}$ is injective. If $p,q$ are not contained in the support of $h_j$, then $w_{j-1} (q) = w_j (q)$ and $w_{j-1} (p)= w_j (p)$ and thus we are done. Since the support of $h_j$ is a compact subset of $U_\ell$, there is a constant $\beta$ such that $|w_{j-1} (p) - w_{j-1} (q)| \geq 2\beta$ for every $q$ in the support of $h_j$ and $p\not\in U_\ell$. For such pairs of points $w_j (p) \neq w_j (q)$ as soon as $\|w_j - w_{j-1}\|_0 \leq \beta$, which can be achieved by choosing $\lambda$ sufficiently large. It remains to check $w_j (p)\neq w_j (q)$ when one point belongs to the support of $h_j$ and the other to $U_\ell$ (and they are distinct!).
Consider that $\overline U_{\ell}$ is a compact set and, since $w_{j-1}$ is injective, its restriction to $\overline U_\ell$ is a smooth embedding. It then follows that, for a sufficiently small $\eta>0$, there is a well-defined orthogonal projection $\pi$ from the normal tubular neighborhood $T$ of thickness $\eta$ of 
$w_{j-1} (U_\ell)$ onto $w_{j-1} (U_\ell)$. Of course if $\lambda$ is sufficiently large $w_j (U_\ell)$ takes values in $T$ and thus, by definition of $w_j - w_{j-1}$, $\pi (w_j (q)) = w_{j-1} (q) \neq w_{j-1} (p) = \pi (w_j (p))$. Obviously this implies $w_j (p)\neq w_j (q)$ and completes the proof.
\end{proof}

\chapter{Smooth isometric embeddings}

\section{Introduction}

Two years after his counterintuitive $C^1$ theorem (see Theorem~\ref{t:main_C1_1}), Nash addressed and solved the general problem of the existence of smooth
isometric embeddings in his other celebrated work \cite{Nash1956}. As in the previous chapter we consider Riemannian manifolds $(\Sigma, g)$, but this time of class $C^k$ with $k\in \mathbb  N \cup \{\infty\}\setminus \{0\}$: this means that there is a $C^\infty$ atlas for $\Sigma$ and that, in any chart the coefficients $g_{ij}$ of the metric tensor in the local coordinates are $C^k$ functions. Nash's celebrated theorem in \cite{Nash1956} is then the following result.

\begin{theorem}[Nash's smooth isometric embedding theorem]\label{t:Ck_1}
Let $k\geq 3$, $n \geq 1$ and $N= \frac{n(3n+11)}{2}$. If $(\Sigma, g)$ is a closed $C^k$ Riemannian manifold of dimension $n$, then there is a $C^k$ isometric embedding $u: \Sigma \to \mathbb R^N$. 
\end{theorem}

In \cite{Nash1956}, Nash covered also the case of nonclosed manifolds as a simple corollary of Theorem~\ref{t:Ck_1}, but with a much weaker bound on the codimension. More precisely he claimed the existence of isometric embeddings for $N' = (n+1) N$. His proof contains however a minor error (Nash really proves the existence of an isometric {\em immersion}) which, as pointed out by Robert Solovay (cf.~Nash's comment in \cite[p.~209]{EssentialNash}), can be easily fixed using the same ideas, but at the price of increasing slightly the dimension $N'$. 

\begin{corollary}[$C^\infty$ isometric embedding, nonclosed case]\label{c:Ck_2}
Let $k\geq 3$, $n \geq 1$, 
\[
N' = (n+1) N = (n+1) \frac{n (3n+11)}{2} \quad\mbox{and}\quad N'' = N' + 2n+2\, .
\] 
If $(\Sigma, g)$ is a $C^k$ Riemannian manifold of dimension $n$, then there is a $C^k$ isometric embedding $u: \Sigma \to \mathbb R^{N''}$ and a $C^k$ isometric immersion $z:\Sigma \to \mathbb R^{N'}$. 
\end{corollary}

\medskip

The dimension of the ambient space in the theorems above has been lowered by subsequent works of Gromov and G\"unther. Moreover, starting from Gromov's work, Nash's argument has been improved to show statements similar to Theorem~\ref{t:main_C1_1}. More precisely, Gromov and Rokhlin first proved in
\cite{GrRo} that any short map on a smooth compact Riemannian manifold can be approximated by isometric embeddings of class $C^\infty$ if the dimension of the ambient Euclidean space is at least $\frac{n(n+1)}{2} + 4n +5$. The latter threshold was subsequently lowered by Gromov in \cite{GromovBook} to $\frac{n(n+1)}{2}+2n+3$ and by G\"unther in \cite{Gunther} to $\frac{n(n+1)}{2} + \max\{2n,5\}$ (see also \cite{GuntherICM}).
If $g$ is real analytic and $m\geq \frac{n(n+1)}{2} + 2n +3$, then any short embedding in $\mathbb R^m$ can be uniformly approximated by analytic isometric embeddings: in \cite{Nash1966} Nash extended Theorem~\ref{t:Ck_1}, whereas the approximation statement was shown first by Gromov for $m\geq \frac{n(n+1)}{2} +3n + 5$ in \cite{Gromov70} and lowered to the threshold above in \cite{GromovBook}.
Corresponding theorems can also be proved for noncompact manifolds $M$, but they are more subtle; for instance the noncompact case with real analytic metrics was left in \cite{Nash1966} as an open problem; we refer the reader to \cite{Gromov70,GromovBook} for more details. 

\medskip

On the regularity side, Jacobowitz in \cite{Jacobowitz} extended Nash's theorem to $C^{k, \beta}$ metrics (achieving the existence of $C^{k, \beta}$ embeddings) for $k+\beta > 2$. However, the case of $C^2$ metrics is still an open problem (it is also interesting to notice that K\"allen in \cite{Kallen} used a suitable improvement of Nash's methods for Theorem~\ref{t:main_C1_1} in order to show the existence of $C^{1, \alpha}$ isometric embeddings with $\alpha < \frac{k+\beta}{2}$ when $k+\beta \leq 2$: the existence of a $C^2$ isometric embedding for $C^2$ metrics is thus an endpoint result for two different ``scales''). 

\medskip

The starting point of Nash in proving Theorem~\ref{t:Ck_1} is first to solve the linearization of the corresponding system of PDEs \eqref{e:PDE_isom}: in particular he realized that a suitable ``orthogonality Ansatz'' reduces the linearization to a system of linear equations which {\em does not involve derivatives of the linearization of the unknown}, cf.~\eqref{e:normal}-\eqref{e:linearization2}. The latter system can then be solved via linear algebra when the dimension of the target space is sufficiently high. 

Having at hand a (simple) solution formula for the linearized system, one would like to recover some implicit (or inverse) function theorem to be able to assert the existence of a solution to the original nonlinear
system \eqref{e:PDE_isom}. There are of course several iterative methods in analysis to prove implicit function theorems, but in Nash's case there is a central analytic difficulty: his solution of the linearized system experiences a phenomenon which in the literature is usually called {\it loss of derivative}.
This problem, which was very well known and occurs in several other situations, looked insurmountable. Mathematics needed the genius of Nash in order to realize that one can deal with it by introducing a suitable regularization mechanism, see in particular the discussion of Section \ref{s:path}. 

\medskip

This key idea has numerous applications in a wide range of problems in partial differential equations where a purely functional--analytic implicit function theorem fails. The first author to put Nash's ideas in the framework of an abstract implicit function theorem was J. Schwartz, cf.~\cite{Schwartz}. However, the method became known as the Nash--Moser iteration shortly after Moser succeeded in developing a general framework going beyond an implicit function theorem, which he applied to a variety of problems in his fundamental papers \cite{MoserPNAS,Moser1,Moser2},  in particular to the celebrated KAM theory. Subsequently several authors generalized these ideas and a thorough mathematical theory has been developed by Hamilton in \cite{Hamilton}, who defined the categories of ``tame Fr\'echet spaces'' and ``tame nonlinear maps''. Such ideas are usually presented in the framework of a Newton iteration scheme. However, although Nash's original argument is in some sense close in spirit, in practice Nash truly constructs a smooth ``curve'' of approximate solutions solving a
suitable infinite dimensional ordinary differential equation: the curve starts with a map which is close to be a solution and brings it to a final one which is a solution. This ``smooth flow'' idea seems to have been lost in the subsequent literature. 

\medskip

It is rather interesting to notice that, in order to solve the isometric embedding problem, Nash did not really need to resort to the very idea which made his
work so famous in the literature of partial differential equations: G\"unther has shown in \cite{Gunther} that the linearization of the isometric embedding system can be solved via a suitable elliptic operator. Hence, one can ultimately appeal to standard contraction arguments in Banach spaces via Schauder estimates, at least if we replace the $C^k$ assumption of Nash's Theorem~\ref{t:Ck_1} with a $C^{k,\alpha}$ assumption for some $\alpha$ contained in the open interval $(0,1)$.

\section{The perturbation theorem}\label{s:perturbation}

As in the previous chapters, we use Einstein's convention on repeated indices.
From now on, given a closed $n$-dimensional manifold $\Sigma$, we fix an atlas $\{U_\ell\}$ as in Lemma~\ref{l:covering}. Given a function $f$ on $\Sigma$, we define then $\|D^k f\|_0$ and $\|f\|_k$ as in Section \ref{s:main_C1_iter}. Given an $(i,j)$ tensor $T$, consider its expression in coordinates  in the patch $U_\ell$, namely
\[
T^{\alpha_1\ldots \alpha_i}_{a_1\ldots a_j} (u) \frac{\partial}{\partial u_{\alpha_1}} \otimes \cdots \otimes \frac{\partial}{\partial u_{\alpha_i}} \otimes du_{a_1} \otimes \cdots \otimes du_{a_j}\, . 
\]
We then define 
\[
\|D^k T\|_{0, U_\ell} := \sum_{\alpha_r, a_s}  \|D^k T^{\alpha_1\ldots \alpha_i}_{a_1\ldots a_j}\|_{0, U_\ell}\, , \;\;
\|D^k T\|_0 := \sup_\ell \|D^k T\|_{0, U_\ell} \;\;\mbox{and}\;\;
\|T\|_k := \sum_{i\leq k} \|D^i T\|_0\, .
\]
It is easy to see that these norms satisfy the Leibnitz-type inequality
\begin{equation}\label{e:Leibnitz}
\|D^k (T\otimes S)\|_0 \leq \sum_{i\leq k} \|D^i T\|_0 \|D^{k-i}\|_0
\end{equation}
and, when contracting a given tensor, namely for $\bar{T}^{\alpha_2 \ldots \alpha_i}_{a_2\ldots a_j} = \sum_k T^{k \alpha_2 \ldots \alpha_i}_{k a_2\ldots a_j}$, we have the corresponding inequality
\begin{equation}\label{e:contraction}
\|\bar T\|_0 \leq n \|T\|_0\, .
\end{equation}
Nash's strategy to attack Theorem~\ref{t:Ck_1} is to prove first a suitable perturbation result. Let us therefore start with a smooth embedding $w_0 = (w_1, \ldots , w_N): \Sigma \to \mathbb R^N$ and set $h := g-w_0^\sharp e$. Assuming $h$ small we look for a (nearby) map $u: \Sigma\to \mathbb R^N$ such that 
$u^\sharp e = g$, namely $u^\sharp e - w_0^\sharp e = h$. In fact, we would like to build $u$ as right endpoint of a path of maps starting at $w_0$. More precisely, consider a smooth curve $[t_0, \infty) \ni t \mapsto h (t)$ in
the space of smooth $(0,2)$ tensors joining $0 = h(t_0)$ and $h = h (\infty)$; we would like to find a corresponding smooth deformation $w (t)$ 
of $w (t_0) = w_0$ to $w (\infty) = u$ so that 
\begin{equation}\label{e:path}
w (t)^\sharp e = w_0^\sharp e + h (t)\qquad \mbox{ for all $t$.}
\end{equation} 
Following Nash's convention, we denote with an upper dot the differentiation
with respect to the parameter $t$. 

If we fix local coordinates $x_1, \ldots , x_n$ in a patch $U$ and differentiate \eqref{e:path}, we then find the
following linear system of partial differential equations for the velocity $\dot w (t)$:
\begin{equation}\label{e:linearization}
\frac{\partial w_\alpha}{\partial x_i} \frac{\partial \dot w_\alpha}{\partial x_j} +  \frac{\partial \dot w_\alpha}{\partial x_i} \frac{\partial w_\alpha}{\partial x_j}
= \dot h_{ij}\, .
\end{equation}
In fact, since the expression in the right-hand side of \eqref{e:linearization} will appear often, we introduce the shorthand notation
$2 d w \odot d\dot w$ for it, more precisely:
\begin{definition}
If $u, v\in C^1 (\Sigma, \mathbb R^N)$, we let $du \odot dv$ be the $(0,2)$ tensor $\frac{1}{2} ((u+v)^\sharp e - v^\sharp e - u^\sharp e)$, which in local
coordinates is given by 
\[
\frac{1}{2} \left(\frac{\partial v_\alpha}{\partial x_i} \frac{\partial u_\alpha}{\partial x_j} + \frac{\partial u_\alpha}{\partial x_i}\frac{\partial v_\alpha}{\partial x_j}\right)\, .
\]
\end{definition}
A second important idea of Nash is to assume that $\dot w$ is orthogonal to $w (\Sigma)$, namely
\begin{equation}\label{e:normal}
\frac{\partial w_\alpha}{\partial x_j} \dot w_\alpha = 0 \qquad \forall j\in \{1, \ldots, n\}\, .
\end{equation}
Under this condition we have
\[
0 = \frac{\partial}{\partial x_i} \left(\frac{\partial w_\alpha}{\partial x_j} \dot w_\alpha\right) = 
\frac{\partial \dot w_\alpha}{\partial x_i}\frac{\partial w_\alpha}{\partial x_j} + \dot w_\alpha \frac{\partial^2 w_\alpha}{\partial x_i \partial x_j}\, ,
\]
and we can rewrite \eqref{e:linearization} as 
\begin{equation}\label{e:linearization2}
-2 \frac{\partial^2 w_\alpha}{\partial x_j \partial x_i} \dot w_\alpha 
= \dot h_{ij}\, .
\end{equation}
Clearly, in order to solve \eqref{e:normal}--\eqref{e:linearization2}, it would be convenient if the resulting system of linear equations were linearly independent,
which motivates the following definition.

\begin{definition}\label{d:free}
A $C^2$ map $w:\Sigma\to \mathbb R^{\bar N}$ is called {\em free}\footnote{The term free was not coined by Nash, but introduced later
in the literature by Gromov.} if, on every system of local coordinates $x_1, \ldots , x_n$, the following $n + \frac{n(n+1)}{2}$ vectors are linearly independent at every $p\in \Sigma$:
\begin{equation}\label{e:lin_indip}
\frac{\partial w}{\partial x_j} (p)\, , \frac{\partial^2 w}{\partial x_i \partial x_j}(p)\, , \qquad \forall i\leq j\in \{1, \ldots , n\}\, .
\end{equation}
\end{definition} 

Although the condition \eqref{e:lin_indip} is stated in local coordinates, the definition is independent of their choice.
Observe moreover that a free map is necessarily an immersion and that we must have $\bar N\geq \frac{n (n+3)}{2}$. If a free map is injective,
then we will call it a free embedding.
The main ``perturbation theorem'' of Nash's paper (and in fact the most spectacular part of his celebrated work) is then the following
statement. In order to prove it, Nash introduced his famous regularization procedure to overcome the most formidable obstruction posed by \eqref{e:linearization}. 

\begin{theorem}[Perturbation theorem]\label{t:perturbation}
Assume $w_0: \Sigma \to \mathbb R^N$ is a $C^\infty$ free embedding. Then there is a positive constant $\varepsilon_0$, depending upon $w_0$, such that, if $h$ is a $C^k$ $(0,2)$ tensor with
$\|h\|_3\leq \varepsilon_0$ and $k\geq 3$ (with possibly $k = \infty$), then there is a $C^k$ embedding $\bar u : \Sigma \to \mathbb R^N$ such that $\bar u^\sharp e = w_0^\sharp e + h$. 
\end{theorem}

Solving the embedding problem using Theorem~\ref{t:perturbation} certainly requires to produce maps which are ``close'' to be an isometric embedding.
However note that there is a rather subtle issue: since the threshold $\varepsilon_0$ depends upon $w_0$, producing a ``good starting'' $w_0$ is not at all obvious. We will tackle this issue immediately in the next sections and then come to the proof of Theorem~\ref{t:perturbation} afterwards.

\section{Proof of the smooth isometric embedding theorem}

In order to exploit Theorem~\ref{t:perturbation}, Nash constructs an embedding $u_0$ of $\Sigma$ which is the cartesian product of two smooth maps $w$ and $\bar{w}$, which he calls,
respectively, the $Z$-embedding and the $Y$-embedding. One crucial elementary ingredient is the following
remark.

\begin{remark}\label{r:product} 
If $f_1: \Gamma \to \mathbb R^n$ and $f_2: \Gamma \to \mathbb R^m$ are two $C^1$ maps, then $(f_1\times f_2)^\sharp e = f_1^\sharp e + f_2^\sharp e$, where we just understand $f_1^\sharp e$, $f_2^\sharp e$ and $(f_1\times f_2)^\sharp e$ as $(0,2)$ tensors (note that they are positive semidefinite, but not necessarily positive {\em definite}). 
\end{remark}

The strategy of Nash can be summarized as follows:
\begin{itemize}
\item[(i)] fix first a free $C^\infty$ smooth embedding $w_0$ (the $Z$-embedding) which is
(strictly) short with respect to $g$ (cf.~Definition~\ref{d:short}), and consider the threshold $\varepsilon_0$ needed to apply Theorem~\ref{t:perturbation};
\item[(ii)] then use a construction somewhat reminiscent of the proof of Theorem~\ref{t:main_C1_1} to build a smooth $\bar{w}$ such that $h:= g- w_0^\sharp - \bar{w}^\sharp$ satisfies
$\|h\|_3\leq \varepsilon_0$;
\item[(iii)] if $\bar{u}$ is finally the map produced by Theorem~\ref{t:perturbation} applied to $w_0$ and $h$, we then set $u := \bar u \times \bar w$ and conclude Theorem~\ref{t:Ck_1}.
\end{itemize}
It is indeed not difficult to produce the $Z$- and $Y$-embeddings if we allow very large dimensions. In order
to achieve the dimension $N$ claimed in Theorem~\ref{t:Ck_1}, Nash follows a much subtler argument which requires the metric difference $g - w_0^\sharp e$ to satisfy a certain nontrivial property: an important ingredient is
the following proposition, whose proof is postponed to the end of the section.

\begin{proposition}\label{p:primitive_scelte}
There are $N_0:= \frac{n (n+3)}{2}$ smooth functions $\psi^r$ on $\Sigma$ such that, for each $p$ in $\Sigma$, $\{d\psi^r(p) \otimes d\psi^r (p) : r\in \{1, \ldots N_0\}\}$ spans
the space $S_p = {\rm Sym}\, (T^*_p \Sigma \otimes T^*_p \Sigma)$. 
\end{proposition}

In fact, if we had the more modest goal of proving the above statement with a much larger $N_0$, we could use the same arguments of Proposition~\ref{p:decomp}. In the proof of Theorem~\ref{t:Ck_1} we still need two technical lemmas, whose proofs will also be postponed. The first one is a classical fact in linear algebra, which will be used also in the next sections.

\begin{lemma}\label{l:Gram}
Consider a $k\times \kappa$ matrix $A$ of maximal rank $k\leq \kappa$. For every vector $v\in \mathbb R^k$, the vector $\omega := A^T (A A^T)^{-1} v$ is a solution of 
the linear system $A \omega = v$. Indeed, $\omega$ gives the solution with smallest Euclidean norm.
\end{lemma}

\begin{remark}\label{r:smooth_dependence}
Note two big advantages of the solution $\omega$ determined through the formula $\omega = A^T (A \cdot A^T)^{-1}v$: 
\begin{itemize}
\item[(a)] $\omega$ depends smoothly upon $A$;
\item[(b)] $\omega$ goes to $0$ when $A$ is fixed and $v$ goes to $0$; indeed this statement remains true even if, while $v$ goes to $0$, the matrix $A$ varies in a compact set over which $A\cdot A^T$ is invertible.
\end{itemize}
\end{remark}

The second is a more sophisticated tool which is used indeed twice in this section.\footnote{It must be observed that Nash employs this fact without explicitly stating it and he does not prove it neither he gives a reference. He uses it twice, once in the proof of Theorem~\ref{t:Ck_1} and once in the proof of Proposition~\ref{p:primitive_scelte}, and although in the first case one could appeal to a more elementary argument, I could not see an easier way in the second.}

\begin{lemma}\label{l:dimension_bound}
Consider a real analytic manifold $\mathcal{M}$ of dimension $r$ and a real analytic map $F: \mathcal{M} \times \mathbb R^\kappa \to \mathbb R^k$.  If, for each $q\in \mathcal{M}$, the set $\mathcal{Z} (q) := \{v: F (q,v)=0\}$ has Hausdorff dimension at most $d$, then the 
set $\mathcal{Z}:= \{v: \exists q\in \mathcal{M} \mbox{ with } F (q,v) =0\}$ has dimension at most $r+d$.
\end{lemma}

\begin{proof}[Proof of Theorem~\ref{t:Ck_1}]
Let $\psi^r$ be the functions of Proposition~\ref{p:primitive_scelte} and set $\gamma := \sum_r d\psi^r \otimes d\psi^r$. After multiplying all the functions by a small factor, we can
assume that $\gamma < g$. Using Theorem~\ref{t:main_C1_1}, we then find a $C^1$ embedding $w: \Sigma \to \mathbb R^{2n}$ such that $w^\sharp e = g-\gamma$. By density of $C^\infty$ functions in $C^1$, we then get a smooth embedding $v$ such that $\|v^\sharp e - (g-\gamma)\|_0< \delta$, where $\delta>0$ is a parameter which will be chosen later. Indeed, by the Whitney's theorem we can assume that $v (\Sigma)$ is a real analytic subvariety, which will play an important role towards the end of the proof. Consider $v$ as an embedding in the larger space $\mathbb R^{\bar N}$ with $\bar N = \frac{n(n+5)}{2}$. We will perturb $v$ to a smooth free embedding $w_0: \Sigma \to \mathbb R^{\bar N}$ with the property that
$\|w_0^\sharp e - (g- \gamma)\|_0 < 2\delta$. Before coming to the proof of the existence of $w_0$, let us first see how we complete the argument. 

First observe that the $(0,2)$ tensor $w_0^\sharp e - (g-\gamma)$ can be written as 
\[
w_0^\sharp e - (g-\gamma) = \sum_r b_r d\psi^r \otimes d\psi^r\, ,
\]
where, thanks to Lemma~\ref{l:Gram}, the coefficients $b_r$ can be chosen smooth. In fact, notice that the coefficients become arbitrarily small as we decrease $\delta$: for a suitable choice of $\delta$ we can thus assume $\|b_r\|_0 \leq \frac{1}{2}$. This is the only requirement on $\delta$: from now on we can consider that the smooth free embedding $w_0$ has been fixed, which in turn gives a positive threshold $\varepsilon_0$ for the applicability of Theorem~\ref{t:perturbation}. Next write
\[
g - w_0^\sharp e = \gamma - \sum_r b_r d\psi^r \otimes d\psi^r = \sum_r (1-b_r) d\psi^r\otimes d\psi^r = \sum_r a_r^2 d\psi^r \otimes d\psi^r\, ,
\]
for the smooth functions $a_r := \sqrt{1-b_r}$. Define $\bar w : \Sigma \to \mathbb R^{2N_0}$ setting
\[
\bar w_{2(i-1) +1} (p) := \frac{a_r (p)}{\lambda} \sin  \lambda \psi^r (p)\, , \qquad \bar w_{2i} (p) := \frac{a_r (p)}{\lambda} \cos \lambda\psi^r (p)\, .
\]
A straightforward computation yields 
\[
\bar w^\sharp e = \sum_r a_r^2 d\psi^r\otimes d\psi^r + \frac{1}{\lambda^2} \sum_r da_r \otimes da_r\, .
\]
In particular,
\[
h := g - (w_0\times \bar w)^\sharp e  = - \frac{1}{\lambda^2} \sum_r da_r \otimes da_r\, .
\]
For $\lambda$ sufficiently large we certainly have $\|h\|_3 \leq \varepsilon_0$ and from Theorem~\ref{t:perturbation} we achieve a $C^k$ embedding $\bar u: \Sigma \to \mathbb R^{\bar N}$ such that $\bar u^\sharp e = w_0^\sharp e +h$. It turns out that $u := \bar u \times \bar w$ is a $C^k$ embedding of $\Sigma$ into $\mathbb R^N = \mathbb R^{\bar N} \times \mathbb R^{2N_0}$ and that $u^\sharp e = g$.

In order to complete the proof, we still need to perturb $v$ to a free $w_0$. 
For any $\eta>0$ we want to construct a free map $w_0: \Sigma \to \mathbb R^{\bar N}$ such that $\|w_0-v\|_1 \leq \eta$. Clearly, for $\eta$ sufficiently small $w_0$ is an embedding. In order to produce $w_0$ we consider the $2n+ n(2n+1)$ functions given by
\[
v_i\, , v_i v_j\, ,  \qquad j\leq i \in \{1, \ldots , 2n\}\, , 
\]  
and those $C^2$ maps $w_0: \Sigma \to \mathbb R^{\bar N}$ given by the formula
\[
(w_0)_\alpha := \sum_i C_\alpha^i v_i + \sum_{j\leq i} D_\alpha^{ij} v_i v_j\, \, ,
\]
for constant coefficients $C_\alpha^i, D_\alpha^{ij}$. We claim that, for a generic choice of the constants $C_\alpha^i$ and $D_\alpha^{ij}$, the map $w_0$ is free. Indeed, consider the set $\mathcal{G}$ of subspaces $L$ of
$\mathbb R^{n + \frac{n(n+1)}{2}}$ with dimension $n-1 + \frac{n(n+1)}{2}$. For each $(p, L)\in \Sigma \times \mathcal{G}$, consider the set $\mathcal{C} (p, L)$ of coefficients $C_\alpha^i, D^\alpha_{ij}$ for which, in a local system of coordinates, 
\begin{equation}\label{e:Valpha}
V_\alpha (p) := \left(\frac{\partial w_\alpha}{\partial x_1} (p), \ldots , \frac{\partial w_\alpha}{\partial x_n} (p),
\frac{\partial^2 w_\alpha}{\partial x_1^2}  (p), \frac{\partial^2 w_\alpha}{\partial x_1 \partial x_2} (p) \ldots , 
\frac{\partial^2 w_\alpha}{\partial x_n^2} (p)\right)\in L 
\end{equation}
for all $\alpha \in \{1, \ldots , \bar N\}$.
This is a set of (linear) conditions which varies analytically as $(p, L)$ varies in the $(2n-1 + \frac{n (n+1)}{2}) = (\bar N -1)$-dimensional manifold
$\Sigma \times \mathcal{G}$. We next show that, if $\bar d$ is the dimension of the linear space of possible coefficients $C_\alpha^i, D_\alpha^{ij}$, then the dimension of each $\mathcal{C} (p, L) $ is at most $d= \bar d - \bar N$. In view of Lemma~\ref{l:dimension_bound} this implies that the union of all $\mathcal{C} (p, L)$ has dimension at most $\bar d -1$. Since the latter is indeed the closed set
$\mathcal{B}$ of ``bad coefficients'' for which $w$ is not free, we have conclude that $\mathcal{B}$ must have empty interior.
 
To complete the proof\footnote{Indeed Nash does not give any argument and just refers to a similar reasoning that he uses in Proposition~\ref{p:primitive_scelte} below.} it remains to bound the dimension of $\mathcal{C} (p, L)$.
Hence fix $p$ and, without loss of generality, assume that $(x_1, \ldots, x_n) = (v_1, \ldots , v_n)$ is a system of coordinates around $p$.
Consider the $M = n + \frac{n (n+1)}{2}$ functions $f_1 = v_1, \ldots , f_n = v_n, f_{n+1} = v_1^2, f_{n+2} = v_1 v_2 , \ldots , v_M = v_n^2$ and
the corresponding vector valued map $f$. It is easy to check that the vectors $\frac{\partial f}{\partial x_1} (p), \ldots, \frac{\partial f}{\partial x_n} (p), \frac{\partial^2 f}{\partial x_1^2} (p) , \frac{\partial^2 f}{\partial x_1 \partial x_2} (p), \ldots , \frac{\partial^2 f}{\partial x_n^2} (p)$ are linearly independent. But then it follows that the vectors 
\[
\bar V_j (p) := \left(\frac{\partial f_j}{\partial x_1} (p), \ldots , \frac{\partial f_j}{\partial x_n} (p),
\frac{\partial^2 f_j}{\partial x_1^2} (p) , \frac{\partial^2 f_j}{\partial x_1 \partial x_2} (p) \ldots , 
\frac{\partial^2 f_j}{\partial x_n^2} (p)\right)\\
\]
are also linearly independent. Hence there is one of them which does not belong to $L$. For each $\alpha\in \{1, \ldots\, \bar N\}$
there is therefore at least one choice of the coefficients $C_\alpha^i, D_\alpha^{ij}$ for which the corresponding vector $V_\alpha (p)$ in \eqref{e:Valpha} does
not belong to $L$. Since $\alpha$ can be chosen in $\bar N$ different ways, the dimension of $\mathcal{C} (p, L)$ is at most $d = \bar d- \bar N$, which completes the proof.
\end{proof}

\begin{proof}[Proof of Proposition~\ref{p:primitive_scelte}]
The argument is very similar to the last part of the proof of Theorem~\ref{t:Ck_1} above. Consider again an embedding $v: \Sigma \to \mathbb R^{2n}$ which makes $v (\Sigma)$ a real analytic submanifold. Let then $f_{ij}$ be the $n (2n+1)$ functions $v_i + v_j$ such that $i\leq j$ and consider
\[
\psi^r := A^r_{ij} f_{ij}\, , \qquad \mbox{for $r\in \{1, \ldots , N_0\}$},
\]
where the space of all possible constant coefficients $A^r_{ij}$ has dimension $\bar d$. Our aim is to show that a generic choice of the coefficients give a system of functions $\psi^r$ which satisfy the conclusions of the proposition.

Let therefore $\mathcal{B}$ be the closed subset of coefficients for which the conclusion fails, namely for each element in $\mathcal{B}$
there exists a point $p$ at which the tensors $d\psi^r (p)\otimes d\psi^r (p)$ do not span the whole space $S_p := {\rm Sym}\, (T^*_p \Sigma\otimes T^*_p \Sigma)$. If we consider the set $\mathcal{G}_p$ of linear subspaces  of $S_p$ of codimension $1$, the real analytic manifold $\mathcal{M} := \{(p, L): L \in \mathcal{G}_p\}$ has dimension $n -1 + \frac{n(n+1)}{2} = N_0 -1$. For each $(p, L)$ we let $\mathcal{C} (p, L)$ be the set of coefficients for which $d\psi^r (p)\otimes d\psi^r (p)$ belongs to $L$ for every $r=1, \ldots , N_0$: this is the zero set of a system of homogeneous quadratic polynomials in the coefficients $A^r_{ij}$. Moreover, in a real analytic atlas for $\mathcal{M}$ these quadratic polynomials depend analytically upon $(p,L)\in \mathcal{M}$. Set $\mathcal{B} = \cup_{(p, L)\in \mathcal{M}} \mathcal{C} (p, L)$. As above we can invoke Lemma~\ref{l:dimension_bound}: if we can bound the dimension of the each $\mathcal{C} (p, L)$ with $\bar d - N_0$, then the dimension of $\mathcal{B}$ is at most $\bar d -1$.

Fix therefore $(p, L)$ and for each $r$ consider the linear space $\pi_r$ of indices $A_{ij}^r$. Without loss of generality we can assume
that $(v_1, \ldots, v_n) = (x_1, \ldots, x_n)$ is a system of coordinates around $p$. Therefore the set $\{d f_{ij}\otimes df_{ij}$ with $i\leq j\leq n\}$ spans the whole space $S_p$ and there is at least one element among them which does not belong to $L$. In turn this means that the subset $\mathcal{C}^r (p, L)\subset \pi_r$ of coefficients $A^r_{ij}$ such that $d\psi^r\otimes d\psi^r$ belongs to $L$ has codimension at least $1$ in $\pi_r$. Therefore the dimension of $\mathcal{C} (p, L) = \mathcal{C}^1 (p, L)\times \mathcal{C}^2 (p, L) \times \ldots \times \mathcal{C}^{N_0} (p, L)$ is at most $d=\bar d -N_0$. This shows $d + N_0 -1 < \bar d$ and completes the proof.\footnote{Nash suggests an alternative argument which avoids the discussion of the dimensions of $\mathcal{C} (p, L)$ and 
$\mathcal{B}$. One can apply his result on real algebraic varieties to find an embedding $v$ which realizes $v (\Sigma)$ as a real algebraic submanifold, cf.~Theorem~\ref{t:alg_main}. Then any set of coefficients $A^r_{ij}$ which is algebraically independent over the minimal field $\mathbb F$ of definition of $v (\Sigma)$ (see Proposition~\ref{p:F1}) belongs to the complement of 
$\mathcal{B}$. Since $\mathbb F$ is finitely generated over the rationals (see Proposition~\ref{p:F1}), it has countable cardinality and the conclusion follows easily.}
\end{proof}

\begin{proof}[Proof of Lemma~\ref{l:Gram}]
It is obvious that $\omega$ solves the desired linear system. Let now $w$ be any solution of minimal Euclidean norm: $w$ is uniquely determined by the property of being orthogonal to the kernel of $A$. However, the kernel of $A$ consists of those vectors which are orthogonal to the image of $A^T$: since the $\omega$ of the lemma belongs to the image of $A^T$, this completes the proof.
\end{proof}

\begin{proof}[Proof of Lemma~\ref{l:dimension_bound}]
Covering $\mathcal{M}$ with a real analytic atlas consisting of countably many charts, we can assume, without loss of generality, that $\mathcal{M}$ is the Euclidean $r$-dimensional ball $B$. Consider next
$Z := \{(q,v): F (q,v) =0\} \subset B \times \mathbb R^\kappa \subset \mathbb R^r\times \mathbb R^\kappa$.
If $\pi: \mathbb R^{r+\kappa}\to \mathbb R^\kappa$ is the projection on the second factor, then $\mathcal{Z} = \pi (Z)$ has at most the dimension of $Z$: it suffices therefore to show that ${\rm dim}\, (Z) \leq r+d$.

Now, $Z$ is a real analytic subvariety in $\mathbb R^{r+\kappa}$ with the property that its slices
$\{q\} \times \mathcal{Z} (q) := Z \cap (\{q\}\times \mathbb R^\kappa)$ all have dimension at most $d$. The dimension $s$ of $Z$ equals the dimension of its regular part $Z^r$ and without loss of generality we can assume that $Z^r$ is connected. Consider now standard coordinates $(x_1, \ldots , x_r)$ on $\mathbb R^r\times \{0\}\subset \mathbb R^{r+\kappa}$ and regard $x_1$ as a function over $Z^r$. By Sard's theorem almost every $\alpha$ is a regular value for $x_1$ on $Z^r$. If one such value $\alpha$ has nonempty preimage, then $Z^r \cap \{x_1 = \alpha\}$ is a submanifold of dimension $s-1$. Otherwise it means that $x_1 (Z^r)$ has measure $0$: since however $x_1 (Z)$ is connected, we must have
$x_1 (Z) = \{ \alpha_0\}$ for some value $\alpha_0$, that is,
$Z^r \cap \{x_1=\alpha_0\} = Z^r$. In both cases we have conclude that there is at least one value $\alpha_0$ such that $Z^r \cap \{x_1 = \alpha_0\}$ is a smooth submanifold of dimension no smaller than $s-1$. Inductively repeating this argument, we conclude that
there is a $q$ such that $Z^r \cap (\{q\}\times \mathbb R^\kappa)$ is a regular submanifold of dimension at least $s-r$.
Since $Z^r \cap (\{q\}\times \mathbb R^\kappa)\subset \{q\}\times \mathcal{Z} (q)$, we infer $s-r\leq d$, which concludes the proof of our claim. 
\end{proof}

\section{Smoothing operator}\label{s:smoothing}

In order to show Theorem~\ref{t:perturbation} we will need to smooth tensors efficiently and get sharp estimates on the $\|\cdot\|_k$ norms of the smoothing. This will be achieved, essentially, by convolution but, since we will need rather refined estimates, the convolution kernel must be chosen carefully. In the remaining sections the specific form of the regularizing operator will play no role:
the only important ingredients are summarized in the following proposition.

\begin{proposition}[Smoothing operator]\label{p:convolution}
There is a family of smoothing operators $\mathcal{S}_\varepsilon$ with $\varepsilon \in ]0, 1[$ such that\footnote{In Nash's paper the operator is called $S_\theta$, where $\theta$ corresponds to $\varepsilon^{-1}$. Since it is nowadays rather unusual to parametrize a family of convolutions as Nash does, I have switched to a more modern convention.}
\begin{itemize}
\item[(a)] $T \mapsto \mathcal{S}_\varepsilon T$ is a linear map on the space of continuous $(i,j)$ tensors; for each such $T$ $\mathcal{S}_\varepsilon T$ is smooth and depends smoothly upon $\varepsilon$.
\item[(b)] For any integers $r\geq s$ and $i,j$, there is a constant $C = C(r,s,i,j)$ such that
\begin{equation}\label{e:conv_der}
\|D^r (\mathcal{S}_\varepsilon T)\|_0 \leq C \varepsilon^{s-r} \|T\|_s \qquad \mbox{for every $C^s$ $(i,j)$ tensor $T$ and $\varepsilon \leq 1$;}  
\end{equation}
\item[(c)] If we denote by $\mathcal{S}'_\varepsilon$ the linear operator $T \mapsto \frac{\partial}{\partial \varepsilon} \mathcal{S}_\varepsilon T$, then for any integers $r,s,i,j$, there is a constant $C=C(r,s,i,j)$ such that
\begin{equation}\label{e:conv_der_2}
\|D^r (\mathcal{S}'_\varepsilon T)\|_0 \leq C \varepsilon^{s-r-1} \|T\|_s \qquad \mbox{for every $C^s$ $(i,j)$ tensor $T$ and $\varepsilon \leq 1$;}  
\end{equation}
\item[(d)] For any integers $s\geq r$ and $i,j$ 
there is a constant $C = C(r,s,i,j)$ such that
\begin{equation}\label{e:conv_der_3}
\|D^r (T-\mathcal{S}_\varepsilon T)\|_0 \leq C \varepsilon^{s-r} \|T\|_s \qquad \mbox{for every $C^s$ $(i,j)$ tensor $T$ and $\varepsilon\leq 1$.}  
\end{equation}
\end{itemize}
\end{proposition}

\begin{proof} As a first step we reduce the problem of smoothing tensors to that of smoothing functions. To achieve this, we fix a smooth embedding of $\Sigma$ into $\mathbb R^{2n}$ (whose existence is guaranteed by the Whitney's embedding theorem), and we therefore regard $\Sigma$ as a submanifold of $\mathbb R^{2n}$. 
We fix moreover a tubular neighborhood $V_{3\eta}$ of $\Sigma$ and assume that the size $3\eta$ is sufficiently small so that the nearest point projection $\pi: V_{3\eta} \to \Sigma$ is well defined and $C^\infty$. Consider now a coordinate patch $U$ on $\Sigma$ and a corresponding system of local coordinates $(u_1, \ldots , u_n)$. We then define the map $x: U \to \mathbb R^{2n}$
where $(x_1 (u), \ldots , x_{2n} (u))$ gives the standard coordinates in $\mathbb R^{2n}$ of the point with coordinates $u$ in $U$. If $\mathcal{N} (U) := \pi^{-1} (U)$, we then define $u: \mathcal{N} (U) \to U$ by letting $u(x)$ be the coordinates, in $U$, of $\pi (x)$. Clearly $u\circ x$ is the identity and $x\circ u$ becomes the identity when restricted on $U\subset \Sigma$. Then, given an $(i,j)$ tensor $T$, which in the local coordinates on $U$ can be expressed as
\[
\sum_{\alpha_1, \ldots, \alpha_i, a_1, \ldots , a_j} T^{\alpha_1\ldots \alpha_i}_{a_1\ldots a_j} (u) \frac{\partial}{\partial u_{\alpha_1}} \cdots \frac{\partial}{\partial u_{\alpha_i}} du_{a_1} \cdots du_{a_j}\, ,
\]
we 
define the functions
\begin{equation}\label{e:S->E}
\mathscr{T}^{\beta_1\ldots \beta_i}_{b_1\ldots b_j} (x) = T^{\alpha_1\ldots \alpha_i}_{a_1\ldots a_j} (u(x)) \frac{\partial x_{\beta_1}}{\partial u_{\alpha_1}} \cdots \frac{\partial x_{\beta_i}}{\partial u_{\alpha_i}} \frac{\partial u_{a_1}}{\partial x_{b_1}} \cdots \frac{\partial u_{a_j}}{\partial x_{b_j}}\, .
\end{equation}
It is easy to check that the functions above do not depend on the chosen coordinates and thus can be defined globally on $\Sigma$. Conversely, if we have global functions $\mathscr{T}$ as above on $\Sigma$, we can ``reconstruct a tensor'' using, in local coordinates, the reverse formulae
\begin{equation}\label{e:E->S}
T^{\alpha_1\ldots \alpha_i}_{a_1\ldots a_j} (u) = \mathscr{T}^{\beta_1\ldots \beta_i}_{b_1\ldots b_j} (x (u)) \frac{\partial u_{\alpha_1}}{\partial x_{\beta_1}} \cdots \frac{\partial u_{\alpha_i}}{\partial x_{\beta_i}} \frac{\partial x_{b_1}}{\partial u_{a_1}} \cdots \frac{\partial x_{b_j}}{\partial u_{a_j}}\, .
\end{equation}
Given these transformation rules and the smoothness of the maps $x\mapsto u (x)$ and $u \mapsto x(u)$, we easily conclude the estimates
\begin{align}
\|D^k T\|_0 &\leq C \sum_{b_1, \ldots , b_j, \beta_1, \ldots , \beta_i} \|\mathscr T^{\beta_1\ldots \beta_i}_{b_1\ldots b_j}\|_k\, ,\\
\|D^k \mathscr T^{\beta_1\ldots \beta_i}_{b_1\ldots b_j}\|_0 &\leq C \|T\|_k\, ,
\end{align}
for a constant $C = C (n,i,j,k)$ which is independent of the tensor $T$. 

Thus, if we have defined a suitable family of smoothing operators $\mathcal{S}_\varepsilon$ on functions over $\Sigma$, we can extend them to tensors with the following algorithm: given a tensor $T$ we produce the functions $\mathscr{T}^{\beta_1\ldots \beta_i}_{b_1\ldots b_j}$ using formula \eqref{e:S->E}; we then apply the smoothing operator
to each function, getting the functions $\mathcal{S}_\varepsilon \mathscr{T}^{\beta_1\ldots \beta_i}_{b_1\ldots b_j}$; we finally use the latter to define $\mathcal{S}_\varepsilon T$ through formula \eqref{e:E->S}. Observe that each of these operations is linear in $T$. 

As a second step we reduce the problem of regularizing functions over $\Sigma$ to that of regularizing functions over $\mathbb R^{2n}$ by a simple extension argument. More precisely, consider a smooth cut-off function $\varphi: \mathbb R^+\to \mathbb R$, which is identically $1$ on $[0, \eta[$, strictly decreasing on $[\eta, 2\eta]$ and identically $0$ on $[2\eta, \infty[$. Given a function $f$ on $\Sigma$, we then extend it to a function $\tilde{f}$ on $V_{3\eta}$ setting $\tilde{f} (x) = \varphi (|x - \pi (x)|) f (\pi (x))$ and subsequently to $\mathbb R^{2n}$ by setting it identically $0$ outside $V_{2\eta}$. 
Again, by the smoothness of $\pi$, it is easy to check that we have the estimate
\[
\|D^k \tilde{f}\|_0 \leq C \|f\|_k\, 
\]
for some constant $C = C(k)$,
where this time $D^k \tilde{f}$ denotes the usual (Euclidean) $k$th derivative and $\|\cdot\|_0$ is the usual maximum norm of a continuous compactly supported function on $\mathbb R^{2n}$.
Conversely, if $\tilde{f}\in C^k_c (\mathbb R^{2n})$, we have 
\[
\|D^k (\tilde{f}|_\Sigma)\|_0 \leq C \|\tilde{f}\|_k = \sum_{i\leq k} \|D^i \tilde f\|_0\, .
\]
Thus, if we can find a suitable regularization operator $\mathcal{R}_\varepsilon$ on $C^k_c (\mathbb R^{2n})$ which satisfies the properties analogous to (a), (b), (c), and (d), we achieve the corresponding desired operator on $C^k(\Sigma)$ via the rule $\mathcal{S}_\varepsilon f = (\mathcal{R}_\varepsilon \tilde{f})|_\Sigma$ (notice again that two points are crucial: the linearity of the maps $f\mapsto \tilde{f}$ and $\tilde{f}\mapsto \tilde{f}_\Sigma$ and the relation $f = \tilde{f}|_\Sigma$). 

We now come to the operator $\mathcal{R}_\varepsilon$ regularizing functions on $\mathbb R^{2n}$, which is 
the convolution with a suitably chosen mollifier $\varphi$ in the Schwartz class $\mathscr{S}$. More precisely, assuming that $m=2n$ and that $\varphi \in \mathscr{S} (\mathbb R^m)$ has integral $1$, we define $\varphi_\varepsilon (x) = \varepsilon^{-m} \varphi(\frac{x}{\varepsilon})$ and set
\[
[\mathcal{R}_\varepsilon f] (x) = f * \varphi_\varepsilon (x) = \int f (x-y) \varphi_\varepsilon (y)\, dy = \frac{1}{\varepsilon^m} \int f (x-y) \varphi \left(\frac{y}{\varepsilon}\right)\, dy\, .
\]
The analog of property (a) is
\begin{equation}\label{e:(a')}
\mbox{$\mathcal{R}_\varepsilon$ maps $C_c (\mathbb R^m)$ into $\mathscr{S} (\mathbb R^m)$ and depends smoothly on $\varepsilon$.}
\end{equation}
The latter is, however, a very standard fact for convolutions. Estimate (b) is also a classical property.
Indeed, given a multiindex $I= (i_1, \ldots , i_m)\in \mathbb N^m$, let $|I|= i_1 + \cdots + i_m$ and
\[
\partial^I f = \frac{\partial^{|I|} f}{\partial x_1^{i_1} \partial x_2^{i_2}\cdots \partial x^{i_m}_m}\, .
\]
If we fix natural numbers $r\geq s$ and consider a multiindex $I$ with $|I| = r$, we can obviously write it as $I= I'+J$ where $|I'|= s$ and $|J| =r-s$. The usual properties of convolutions yield then the following estimate
\[
\| \partial^I (\mathcal{R}_\varepsilon f)\|_0 = \|(\partial^{I'} f)* (\partial^J \varphi_\varepsilon)\|_0 \leq\|\partial^{I'} f\|_0 \|\partial^J \varphi_\varepsilon\|_{L^1}
\leq \|D^s f\|_0 \varepsilon^{s-r} \|\partial^J \varphi\|_{L^1}\, .
\]
Thus, if we define $C := \min_{|J| = r-s} \|\partial^J \varphi\|_{L^1}$, we achieve 
\begin{equation}\label{e:(b')}
\|\partial^I (\mathcal{R}_\varepsilon f)\|_0 \leq C \varepsilon^{r-s} \|D^s f\|_0 \qquad \mbox{when $s\leq r$.}
\end{equation}
Coming to (c), we use elementary calculus to give a formula for $\mathcal{R}'_\varepsilon := \frac{\partial}{\partial \varepsilon} \mathcal{R}_\varepsilon$:
\begin{align*}
\mathcal{R}'_\varepsilon f (x)= \int f (x-y) \left[- \frac{m}{\varepsilon^{m+1}} \varphi \left(\frac{y}{\varepsilon}\right) - 
\frac{1}{\varepsilon^m} \nabla \varphi \left(\frac{y}{\varepsilon}\right)\cdot \frac{y}{\varepsilon^2}\right]\, dy\, .
\end{align*}
If we set $\psi (y) := - m\varphi (y) - \nabla \varphi (y) \cdot y$ and $\psi_\varepsilon (y) = \varepsilon^{-m} \psi (\frac{y}{\varepsilon})$, 
we conclude the identity
\begin{equation}
\mathcal{R}'_\varepsilon f = \varepsilon^{-1} f* \psi_\varepsilon\, .
\end{equation}
Note that even $\psi$ belongs to the Schwartz class. 
Hence,  by the argument given above, the following inequality
\begin{equation}\label{e:figata}
\|D^r (\mathcal{R}'_\varepsilon f)\|_0 \leq C \varepsilon^{s-r-1} \|D^s f\|_0 
\end{equation}
is certainly valid for $r\geq s$. However, the crucial point of estimate (c) is its validity even in the range  $r <s$! In order to achieve this stronger bound we need to choose a specific mollifier $\varphi$: more precisely we require that:
\begin{equation}\label{e:figata2}
\forall k\in \mathbb N \qquad \exists \vartheta^{(k)} \in \mathscr{S}\quad\mbox{such that}\quad \frac{\partial^k \vartheta^{(k)}}{\partial x_1^k} = \psi\, .
\end{equation} 
With this property, for $s>r$ we can integrate by parts $k= s-r$ times to achieve the identity
\[
\mathcal{R}'_\varepsilon f = \varepsilon^{s-r-1} \frac{\partial^{s-r} f}{\partial x_1^{s-r}} * \vartheta^{(s-r)}_\varepsilon\, ,
\]
and, applying the same argument used for \eqref{e:(b')}, we conclude \eqref{e:figata}. 

In order to find a kernel $\varphi$ such that \eqref{e:figata2} holds, we compute first the Fourier transform of $\psi$:
\[
\hat\psi (\xi) = - m \hat\varphi (\psi) - \sum_j \left(-\frac{1}{i} \frac{\partial}{\partial \xi_j}\right) \left (i\xi_j \hat\varphi (\xi)\right) =
\nabla \hat\varphi (\xi)\cdot \xi.
\]
Assume $\hat \varphi\in C^\infty_c (\mathbb R^m)$ and equals $(2\pi)^{\frac{m}{2}}$ in a neighborhood of $0$. Then
$\varphi$ belongs to $\mathscr{S}$ and has integral $1$. Moreover $\hat\psi$ vanishes in a neighborhood of the origin, and thus $(i\xi_1)^{-k} \hat\psi$ belongs to $\mathscr{S}$. But then, if we let $\vartheta^{(k)}$ be the inverse Fourier transform of the latter function, we conclude that $\vartheta^{(k)}\in \mathscr{S}$ and that
$\frac{\partial^k \vartheta^{(k)}}{\partial x_1^k} = \psi$.

To complete the proof, we finally show the analog of estimate (d), namely
\begin{equation}
\|D^r (f - \mathcal{R}_\varepsilon f)\|_0 \leq C \varepsilon^{s-r} \|D^s f\|_0 \qquad \mbox{when $s\geq r$.} 
\end{equation}
For $s=r$ it is an obvious outcome of \eqref{e:(b')}. For $s>r$, we instead integrate \eqref{e:figata} in $\varepsilon$:
\begin{align*}
\|D^r (f -\mathcal{R}_\varepsilon f)\|_0 & \leq \int_0^\varepsilon \left\|D^r \left( \mathcal{R}'_\delta f\right)\right\|_0\, d\delta
\leq C \|D^s f\|_0 \int_0^\varepsilon \delta^{s-r-1}\, d\delta = C \varepsilon^{s-r} \|D^s f\|_0\, 
\end{align*}
(note that $s-r-1\geq 0$ under our assumptions!). 
\end{proof}

\section{A smooth path to prove the perturbation theorem}\label{s:path}

Recalling Section \ref{s:perturbation}, we wish to construct 
\begin{itemize}
\item[(i)] a path $[t_0, \infty) \ni t\mapsto h (t)$ joining $0$ to $h$
\item[(ii)] and a path $[t_0, \infty) \ni t\mapsto w (t)$ joining $w_0$ to $\bar u$ 
\end{itemize}
such that 
\begin{equation}\label{e:path2}
\frac{d}{dt} w (t)^\sharp e = \dot h (t)\, .
\end{equation}
Recall moreover that we have reduced \eqref{e:path2} to solving \eqref{e:normal}--\eqref{e:linearization2} for the 
``velocity'' $\dot w$ of $w$, at least in local coordinates. Assuming that $w(t)$ is a free map for every $t$, we can use Lemma~\ref{l:Gram} to find, in a given coordinate patch, a ``canonical'' solution of the linear system \eqref{e:normal}-\eqref{e:linearization}: more precisely we can write 
\begin{equation}\label{e:notazione_davide_1}
\dot w_\alpha := \mathcal{L}_\alpha^{ij} (Dw, D^2 w) h_{ij}
\end{equation}
where  $\mathcal{L}_\alpha^{ij} (A, B)$
is a suitable collection of functions which
depend smoothly (in fact analytically) upon the entries $A$ and $B$. This defines a linear operator $\mathcal{L} (Dw, D^2w)$ from the space of $(0,2)$ tensors 
over the coordinate patch $U$ into the space of maps $\dot w : U \to \mathbb R^N$. Next, we wish to extend this operator to the whole manifold $\Sigma$: the crucial point is that, although derived in a coordinate patch, the formula above does not depend on the chosen coordinate patch.

\begin{lemma}[Existence of the operator $\mathscr{L}$]
Assume $w: \Sigma \to \mathbb R^N$ is $C^2$ and free. 
Given any $(0,2)$ tensor $\bar h$ and any coordinate patch, the map $\mathcal{L} (Dw, D^2 w) \bar h$ defined above does not depend on the coordinates and the process defines, therefore, a global (linear) operator $\mathscr{L} (w)$ from the space of smooth symmetric $(0,2)$ tensors over $\Sigma$ into the space of smooth maps $C^\infty (\Sigma, \mathbb R^N)$.
\end{lemma}
\begin{proof}
Observe that, for each fixed $p\in \Sigma$, the linear space of vectors $z = \dot w (p)$ satisfying the system \eqref{e:normal}--\eqref{e:linearization2} is independent
of the choice of coordinates (in other words, although the coefficients in the system might change, the solution set remains the same: this follows from straightforward computations!).
Since, however, according to Lemma~\ref{l:Gram} the vector $[\mathcal{L} (Dw, D^2 w) h] (p)$ is the (unique) element of minimal norm in such vector space, it turns out that
it is independent of the coordinates chosen to define $\mathcal{L} (Dw, D^2 w) \bar h$. 
\end{proof}

Having defined the operator $\mathscr{L} (w)$ we can rewrite \eqref{e:path2} as a ``formal system of ordinary differential equations'' 
\begin{equation}\label{e:fake_Picard}
\left\{
\begin{array}{l}
\dot w (t) = \mathscr{L} (w (t)) \dot h (t)\, ,\\ \\
w (t_0) = w_0\, .
\end{array}
\right.
\end{equation}
The problem with this approach is that
the operator $\mathscr{L}$ ``loses derivatives'' in its nonlinear entry $w$, namely although it defines the velocity $\dot w$ at order $0$, it depends on first and second
derivatives of $w$. Hence, if $w, \dot h\in C^k$, then $\mathscr{L} (w) h$ is, a priori, only in $C^{k-2}$. There is therefore no classical functional analytic setting to solve \eqref{e:fake_Picard} in the usual way, namely no Banach space where we can apply a Picard--Lindel\"of or a Cauchy--Lipschitz iteration. 

In order to get around this (very discouraging) issue, Nash considered the regularized problem
\begin{equation}\label{e:ODE}
\left\{
\begin{array}{l}
\dot w (t) = \mathscr{L} (\mathcal{S}_{t^{-1}} w (t)) \dot h (t)\\ \\
w (t_0) = w_0\, .
\end{array}
\right.
\end{equation}
However, $h (t)$ must now be chosen carefully and, in fact, it will be chosen depending upon $w(t)$, so that the complete system will be given by the coupling of
\eqref{e:ODE} with a second equation relating $w (t)$ and $h (t)$. In order to describe the latter, we introduce a function $\psi\in C^\infty (\mathbb R)$ which is:
\begin{itemize}
\item[(a)] identically equal to $0$ on the negative real axis; 
\item[(b)] identically equal to $1$ on $[1, \infty)$;
\item[(c)] everywhere nondecreasing. 
\end{itemize}
The path $h$ is then linked to $w$ through the relation
\begin{equation}\label{e:integral}
h (t) = \mathcal{S}_{t^{-1}} \left[ \psi (t- t_0) h  + \int_{t_0}^t [2 d (\mathcal{S}_{\tau^{-1}} w (\tau) - w (\tau))] \odot d \dot w (\tau)\, \psi (t-\tau)\, d\tau\right]\, .
\end{equation}
From now on the system \eqref{e:ODE}--\eqref{e:integral} will be called {\em Nash's regularized flow equations}. 

In order to gain some insight in the latter complicated relation, assume for the moment that we are able to find an initial value $t_0$ and a smooth curve $t \mapsto (w (t), h (t))$ in $C^3$ satisfying \eqref{e:ODE}--\eqref{e:integral} over $[t_0, \infty)$. In particular, when we refer to a ``smooth solution'' of the regularized flow equations, we understand that $\mathcal{S}_{t^{-1}} w (t)$ is a free map for every $t$ in the domain of definition.
 
Assume further that $w (t)$ converges in $C^2$ to some $\bar u$ for $t\uparrow\infty$ and that the integrands in the following computations all decay sufficiently fast, so that we can integrate over the whole halfline $[t_0, \infty)$. The relation \eqref{e:ODE} implies that
\begin{equation}\label{e:formal1}
2 d (\mathcal{S}_{t^{-1}} w (t)) \odot d \dot w (t) = \dot h (t)\, .
\end{equation}
Integrating the latter identity between $t_0$ and $\infty$, we then get
\begin{equation}\label{e:formal2}
\int_{t_0}^\infty 2 d (\mathcal{S}_{\tau^{-1}} w(\tau)) \odot d\dot w (\tau)\, d\tau = h(\infty) - h (t_0) = h\, .
\end{equation}
Letting $t\to \infty$ in \eqref{e:integral} and using that $\mathcal{S}_{t^{-1}}$ converges to the identity, we conclude
\[
h = h(\infty) = h + \int_{t_0}^\infty 2 d (\mathcal{S}_{\tau^{-1}} w(\tau) - w (\tau)) \odot d \dot w (\tau) d\tau\, ,
\]
implying therefore 
\begin{equation}\label{e:formal3}
\int_{t_0}^\infty 2 d (\mathcal{S}_{\tau^{-1}} w(\tau)) \odot d\dot w (\tau)\, d\tau = \int_{t_0}^\infty 2 d w (\tau) \odot d \dot w (\tau) d\tau\, .
\end{equation}
Combining the latter equality with \eqref{e:formal2} we achieve
\begin{equation}\label{e:formal4}
\int_{t_0}^\infty 2 d w (\tau) \odot d \dot w (\tau)d\tau = h\, .
\end{equation}
On the other hand, the integrand in the left-hand side is precisely $\frac{d}{d\tau} w(\tau)^\sharp e$, and thus we immediately conclude
\begin{equation}\label{e:formal5}
\bar{u}^\sharp e - w_0^\sharp e = w(\infty)^\sharp e - w(t_0)^\sharp e = h\, ,
\end{equation}
namely that $\bar u$ is the map in the conclusion of Theorem~\ref{t:perturbation}.

\medskip

In order to carry out the program above, we obviously have to ensure that
\begin{itemize}
\item[(a)] The regularized flow equations, namely the pair \eqref{e:ODE}--\eqref{e:integral}, is locally solvable; more precisely, if there is a solution in the interval $[t_0, t_1]$, it can be prolonged to some larger open interval $[t_0, t')$.
\item[(b)] We have uniform estimates ensuring the global solvability, namely any smooth solution on $[t_0, t')$ can be smoothly prolonged to the closed interval $[t_0, t']$. 
\end{itemize}
The combination of (a) and (b) would then imply the existence of a global solution on $[t_0, \infty)$. We further have to ensure that
\begin{itemize}
\item[(c)] The limit $\bar u$ of $w (t)$ for $t\to \infty$ exists in the strong $C^3$ topology, and we have the appropriate decay of the integrands needed to justify the ``formal computations'' \eqref{e:formal1}--\eqref{e:formal5}
\end{itemize}
This last step will make the computations above rigorous and ensure that $\bar u$ is a $C^3$ isometric embedding. In order to complete the proof of Theorem~\ref{t:Ck_1}, we will then only need to show that, when $h\in C^k$, then $u$ is also in $C^k$.

The program above will be carried out in the subsequent sections under the assumption that $t_0$ is sufficiently large and $\|h\|_3$ sufficiently small, depending on the ``initial value'' $w_0$. Moreover, we will follow a somewhat different order. First we tackle a set of a priori estimates which are certainly powerful enough to conclude (b) and (c), cf.~Proposition~\ref{p:a priori}. We then examine the local existence of the solution, which combined with the estimates of Proposition~\ref{p:a priori} will immediately imply both global solvability and convergence to an isometry, cf.~Proposition~\ref{p:existence}. Finally,
the higher differentiability of $\bar u$ is achieved in Proposition~\ref{p:higher}.

\section{A priori estimates for solutions of Nash's regularized flow equations}

We start by fixing one important constant: $\varepsilon>0$ will be chosen so that
\begin{equation}\label{e:def_epsilon_0}
\mbox{if}\; \|u-w_0\|_2 \leq 4\varepsilon\quad \mbox{then $u$ is a free embedding.}
\end{equation}
Our main a priori estimates are summarized in the following proposition, which is indeed the core of Nash's approach.

\begin{proposition}[A priori estimates]\label{p:a priori}
For any $t_0$ sufficiently large there is $\delta (t_0)>0$ such that, if $\|h\|_3 \leq \delta$, then the following holds. Consider any solution $w$ of \eqref{e:ODE}--\eqref{e:integral} over an interval $I$ (with left endpoint $t_0$ and which might be closed, open or infinite) . If
\begin{align}
&\|w (t)  - w_0\|_3 + t^{-1}\|w(t) -w_0\|_4 \leq  2 \varepsilon\, ,\label{e:starting1}\\
&t^4 \|\dot h (t)\|_0 + \|\dot h (t)\|_4 \leq 2\, ,\label{e:starting2}
\end{align}
then indeed we have the improved bounds
\begin{align}
&\|w(t) -w_0\|_3 + t^{-1}\|w(t) -w_0\|_4 \leq \varepsilon\, ,\label{e:a priori1}\\
&t^4 \|\dot h (t)\|_0 + \|\dot h (t)\|_4\leq 1\, .\label{e:a priori2}
\end{align}
Moreover, 
\begin{equation}\label{e:a priori3}
t^4 \|\dot w(t)\|_0 + \|\dot w (t)\|_4 \leq C_0\, ,
\end{equation}
and, if $I = [t_0, \infty)$, then there is a function $\delta (s)$ with $\lim_{s\to \infty} \delta (s) =0$ such that
\begin{align}
\|w (t) - w(s)\|_3 \leq \delta (s) \qquad \mbox{for all $t\geq s\geq t_0$.}\label{e:Cauchy}
\end{align}
\end{proposition}

Before coming to the proof we recall here a few useful estimates.

\begin{lemma}\label{l:interpolation}
If $T$ is a smooth $(i,j)$ tensor on $\Sigma$ and $r < \sigma < s$ are three natural numbers, then there is a
constant $C= C(r,s, \sigma, i, j)$ such that
\begin{equation}\label{e:interpolation}
\|T\|_\sigma \leq C \|T\|_r^\lambda \|T\|_s^{1-\lambda}\qquad \mbox{where $\sigma = \lambda r + (1-\lambda) s$.}
\end{equation}
If $\Psi: \Gamma \to \mathbb R^k$ is a smooth map, with $\Gamma \subset \mathbb R^\kappa$ compact and $r$ a natural number, then there is
a constant $C (r, \Psi)$ such that 
\begin{equation}\label{e:composition}
\|\Psi \circ v\|_r \leq C (1 +\|v\|_r) \qquad \mbox{for
every smooth $v: \Sigma \to \Gamma$.}
\end{equation}
For every $r\in \mathbb R$ there is a constant $C(r)$ such that
\begin{equation}\label{e:product}
\|\varphi\psi\|_r \leq C \|\varphi\|_0 \|\psi\|_r + C \|\varphi\|_r\|\psi\|_0 \qquad \mbox{for every $\varphi, \psi\in C^r (\Sigma)$.}
\end{equation}
The inequality extends as well to (tensor) product of tensors, where the constant will depend additionally only on the type of
tensors involved. 
\end{lemma}

The lemma above follows from rather standard and well-known arguments and we will give some explanations and references at the end of section. We underline here 
a crucial consequence, which will be used repeatedly in our arguments.

\begin{remark}\label{r:interpola}
From \eqref{e:interpolation} we easily conclude that, if $\|T (t)\|_k \leq \lambda t^j$ and
$\|T\|_{k+i}\leq \lambda t^{j+i}$, then $\|T\|_{k+\kappa} \leq C \lambda t^{j+\kappa}$ for all intermediate $\kappa \in \{1, \ldots , j-1\}$.\footnote{Nash does not
take advantage of this simple remark and introduces instead a rather unusual notation to keep track of all the estimates for the intermediate norms in the bounds
corresponding to \eqref{e:a priori1}--\eqref{e:a priori3}.}
\end{remark}

\begin{proof}[Proof of Proposition~\ref{p:a priori}]
First of all, if $t_0$ is chosen larger than a fixed constant, we can use \eqref{e:starting1} and Proposition~\ref{p:convolution}(d) to conclude that $\|\mathcal{S}_{t^{-1}} w (t) - w_0\|_2 \leq 4\varepsilon$. In turn,  by \eqref{e:def_epsilon_0}, this implies that, when computing the operator $\mathscr{L}$, the entries of $\mathcal{L}^{ij}_\alpha$ belong to a compact set where the corresponding functions are smooth. Observe moreover that $\|w (t)\|_3 \leq C$, for some constant $C$ depending only upon the initial value $w_0$. 
We can thus apply \eqref{e:composition} and Proposition~\ref{p:convolution} to conclude that 
\begin{equation}\label{e:anche_questa}
\|\mathscr{L} (\mathcal{S}_{t^{-1}} w (t))\|_\kappa \leq C (\kappa) (1+ t^{\kappa-1})
\end{equation} 
where $C(\kappa)$ is a constant which depends only upon $\kappa$. In fact, for $\kappa\geq 1$ we have
\begin{align*}
\|\mathscr{L} (\mathcal{S}_{t^{-1}} w (t))\|_\kappa \stackrel{\eqref{e:composition}}{\leq} C (\kappa) \|\mathcal{S}_{t^{-1}} w (t)\|_{\kappa+2}
{\leq} C (\kappa) \|w (t)\|_3 t^{\kappa-1}\, ,
\end{align*}
where the last inequality follows from Proposition~\ref{p:convolution}(b).
In the case of $\kappa=0$, we use instead the estimate $\|\mathcal{S}_{t^{-1}} w(t)\|_2 \leq C \|w (t)\|_2$ (again cf.~Proposition~\ref{p:convolution}(b)). 

Using now \eqref{e:product}, from \eqref{e:ODE} we conclude that
\begin{align}
\|\dot w (t)\|_0 & \leq \|\mathscr{L} (\mathcal{S}_{t^{-1}} w (t))\|_0 \|\dot h (t)\|_0 \leq C  t^{-4}\, ,\label{e:weaker-1}\\
\|\dot w(t)\|_4 & \leq  \|\mathscr{L} (\mathcal{S}_{t^{-1}} w (t))\|_4\|\dot h (t)\|_0 + C\|\mathscr{L} (\mathcal{S}_{t^{-1}} w (t))\|_0\|\dot h (t)\|_4 \leq C\, .\label{e:weaker-2}
\end{align}
Indeed, this shows \eqref{e:a priori3}.

We next introduce some additional functions in order to make some expressions more manageable. More precisely
\begin{align}
E (t) := & 2 d (\mathcal{S}_{t^{-1}} w (t) - w (t)) \odot d \dot w (t)\, ,\label{e:E(t)}\\
L (t):= & \int_{t_0}^t E (\tau) \psi (t-\tau)\, d\tau\, .
\end{align}
Observe that with the introduction of these two quantities we can rewrite \eqref{e:integral} as 
\begin{equation}\label{e:h}
h (t) = \mathcal{S}_{t^{-1}} [ \psi (t-t_0) h + L (t)]\, .
\end{equation}
Recalling Proposition~\ref{p:convolution}, we have $\|\mathcal{S}_{t^{-1}} w(t) - w(t)\|_1 \leq C t^{-2}\|w (t)\|_3 \leq C t^{-2}$. Observe that $\|\dot w (t)\|_1 \leq C t^{-3}$, which follows  
from \eqref{e:weaker-1} and \eqref{e:weaker-2} because of Remark~\ref{r:interpola} (this is just one of several instances where such remark will be used!). Combining the latter estimate with \eqref{e:product}, we then conclude $\|E(t)\|_0 \leq C t^{-5}$. On the other hand,  
\[
\|\mathcal{S}_{t^{-1}} w(t) - w(t)\|_4 \leq C t\, ,
\]
and hence again from \eqref{e:product} we conclude
\begin{equation}\label{e:bound_E(tau)}
\|E(t)\|_3 \leq C \|\mathcal{S}_{t^{-1}} w(t) - w(t)\|_4\|\dot w (t)\|_1 + C \|\mathcal{S}_{t^{-1}} w(t) - w(t)\|_1 \|\dot w (t)\|_4 \leq C t^{-2}\, .
\end{equation}
The latter inequality yields
\begin{equation}\label{e:a priori4}
\|L (t)\|_3 \leq \int_{t_0}^t \|E (\tau)\|_3\, d\tau \leq C t_0^{-1}\, .
\end{equation}
Next, we compute
\[
\dot h(t) = \underbrace{\left(\textstyle{\frac{d}{dt}} \mathcal{S}_{t^{-1}}\right) [ \psi (t-t_0) h + L (t)]}_{=:P(t)} + \mathcal{S}_{t^{-1}} [ \psi' (t-t_0) h + \dot L (t)]\, .
\]
First, we observe that $\psi' (t-t_0)$ vanishes for $t> t_0+1$ and $t<t_0$. Hence
\begin{equation}\label{e:pezzo1}
\|\psi' (t-t_0) \mathcal{S}_{t^{-1}} h\|_4  \leq \left\{\begin{array}{ll} 
C t_0 \delta \qquad &\mbox{for $t\in [t_0, t_0+1]$,}\\ \\
0 &\mbox{otherwise.}
\end{array}\right.
\end{equation}
For the same reason (and because $\psi (0)=0$) we can estimate
\begin{align}
\|\dot L (t)\|_0 \leq &\int_{\max\{t_0, t-1\}}^t \|E (\tau)\|_0\, d\tau  \leq C t^{-5} \, ,\label{e:pezzo2}\\
\|\dot L (t)\|_3 \leq &\int_{\max \{t_0, t-1\}}^t \|E (\tau)\|_3\, d\tau\leq C t^{-2} \, .\label{e:pezzo2_bis}
\end{align}
Next, recalling that $\mathcal{S}' _\varepsilon := \frac{d}{d \varepsilon} \mathcal{S}_\varepsilon$, we have
\[
\textstyle{\frac{d}{dt}} \mathcal{S}_{t^{-1}} = - t^{-2} \mathcal{S}'_{t^{-1}}\, .
\]
Hence, using Proposition~\ref{p:convolution}(c) and \eqref{e:a priori4}, it is straightforward to show that
\begin{equation}\label{e:pezzo3}
t^4 \|P(t)\|_0 + \|P(t)\|_4 \leq C (\|h(t)\|_3+ \|L (t)\|_3) \leq C \delta + C t_0^{-1}\, ,
\end{equation}
where $C$ is independent of $\delta$.  Combining \eqref{e:pezzo1}, \eqref{e:pezzo2}, \eqref{e:pezzo2_bis}, and \eqref{e:pezzo3} we get
\begin{equation}\label{e:doth_bound}
t^4 \|\dot{h} (t)\|_0 + \|\dot h(t)\|_4 \leq C t^{-1} + C \delta (1+ t_0^5) + C t_0^{-1} \leq C t_0^{-1} + C \delta t_0^5\, .
\end{equation} 
Therefore, choosing first $t_0$ large enough and then $\delta \leq \delta_0 (t_0)$ sufficiently small, we conclude a bound which is even stronger than \eqref{e:a priori2}: the left-hand side can be made smaller than any fixed $\eta>0$. 

The estimate on $\|w(t) -w_0\|_4$ in \eqref{e:a priori1} is an obvious consequence of the one above on $\|\dot h (t)\|_4$ through integration of \eqref{e:ODE}: it suffices to choose $\eta$ smaller than a given constant. 
The proof of the remaining parts of \eqref{e:a priori1} and \eqref{e:Cauchy} require instead a subtler argument. However, notice also that
we just need to accomplish \eqref{e:Cauchy}, since $C_0$ is a constant claimed to be independent of $t_0$.

In order to get \eqref{e:Cauchy} we integrate \eqref{e:ODE} and then integrate by parts:
\begin{align}
& w (t) - w(s)\nonumber\\
&\qquad =  \int_s^t \mathscr{L} (\mathcal{S}_{\tau^{-1}} (w (\tau))) \dot h (\tau)\, d\tau\nonumber\\
& \qquad =  - \int_s^t \underbrace{\left[\frac{d}{d\tau} \mathscr{L} (\mathcal{S}_{\tau^{-1}} (w (\tau)))\right]}_{=: D(\tau)} (h (\tau)- h(t))\, d\tau
+  \mathscr{L} (\mathcal{S}_{t^{-1}} (w (s))) (h (t) - h(s))\, .\label{e:verbatim1}
\end{align}
First of all, integrating the bound \eqref{e:a priori2} on $\dot h (t)$, we obviously conclude
\begin{equation}\label{e:verbatim2}
\|h (t) - h(s)\|_0 \leq C s^{-3} \qquad \mbox{for all $t\geq s\geq t_0$.}
\end{equation}
Next, assuming that $t \geq s \geq t_0+1$, we have $\psi (s-t_0) = \psi (t-t_0) =1$ and we can thus compute
\begin{align}
h (t) - h (s) & =  (\mathcal{S}_{t^{-1}} h - \mathcal{S}_{s^{-1}} h) + \underbrace{\mathcal{S}_{t^{-1}} \int_s^t E (\tau)\, \psi (t-\tau)\, d\tau}_{\mbox{(I)}}\nonumber\\
&\qquad + \underbrace{\mathcal{S}_{t^{-1}} \int_{s-1}^s E (\tau)\, (\psi (t-\tau) - \psi (s-\tau))\, d\tau}_{\mbox{(II)}}
+ \underbrace{\vphantom{\int_{s-1}^s} (\mathcal{S}_{t^{-1}} - \mathcal{S}_{s^{-1}}) L(s)}_{\mbox{(III)}}\, .\label{e:verbatim3}
\end{align}
Note next that
\[
\|\mbox{(I)}+ \mbox{(II)}\|_3 \leq C \int_{s-1}^t \| E (\tau)\|_3\, d\tau \leq C \int_{s-1}^\infty \tau^{-2}\, d\tau \leq C s^{-1}\, .
\]
For what concerns (III) note that the bound \eqref{e:bound_E(tau)} on $\|E (\tau)\|_3$ implies that
\[
L (\infty) := \int_{t_0}^\infty E (\tau)\, d\tau
\] 
is well defined, it belongs to $C^3$, and it satisfies the following decay estimate:
\begin{equation}\label{e:verbatim3.5}
\|L (\infty) - L(s)\|_3 \leq C s^{-1}\, .
\end{equation}
Thus we can bound
\[
\|\mbox{(III)}\|_3\leq C s^{-1} + \|\mathcal{S}_{s^{-1}} L (\infty) - \mathcal{S}_{t^{-1}} L (\infty)\|_3\, ,
\]
which in turn leads to 
\begin{equation}\label{e:verbatim4}
\|h (t) - h(s)\|_3 \leq C s^{-1} + \|\mathcal{S}_{s^{-1}} L (\infty) - \mathcal{S}_{t^{-1}} L (\infty)\|_3\ + \|\mathcal{S}_{t^{-1}} h - \mathcal{S}_{s^{-1}} h \|_3\, .
\end{equation}
Using the fact that $\mathcal{S}_{t^{-1}}$ converges to the identity for $t\to \infty$, we reach
\begin{equation}
\|h (t) - h (s)\|_3 \leq \tilde{\delta} (s) \qquad \mbox{for all $t\geq s$,}
\end{equation}
where $\tilde{\delta} (s)$ is a function such that $\lim_{s\to \infty} \tilde{\delta} (s) = 0$.
Using \eqref{e:anche_questa}, \eqref{e:verbatim2} and \eqref{e:verbatim4},  we conclude
\begin{align}\label{e:intermedia}
\|w(t) - w(s)\|_3  \leq & \bar \delta (s) + C \int_s^t (\|D (\tau)\|_3 \tau^{-3} + \|D (\tau)\|_0)\, d\tau\, ,
\end{align}
for some function $\bar \delta (s)$ which converges to $0$ as $s$ goes to $\infty$. 

In order to estimate carefully $D(t)$, we pass to local coordinates. Recalling the notation $\mathcal{L}^{ij}_\alpha = \mathcal{L}^{ij}_\alpha (A,B)$ of \eqref{e:notazione_davide_1}
we compute
\begin{align}
& \textstyle{\frac{d}{dt}} \mathcal{L}_\alpha^{ij} (D \mathcal{S}_{t^{-1}} w(t), D^2 \mathcal{S}_{t^{-1}} w (t))\nonumber \\
& \qquad = \underbrace{D_A \mathcal{L}_\alpha^{ij} (D \mathcal{S}_{t^{-1}} w (t), D^2 \mathcal{S}_{t^{-1}} w(t))}_{D' (t)} \circ \left(-t^{-2} D \mathcal{S}'_{t^{-1}} w(t)
+ \mathcal{S}_{t^{-1}} D \dot w (t)\right)\nonumber\\ 
&\qquad\qquad+  \underbrace{D_B  \mathcal{L}_\alpha^{ij} (D \mathcal{S}_{t^{-1}} w (t), D^2 \mathcal{S}_{t^{-1}} w(t))}_{D'' (t)}
\circ \left(-t^{-2} D^2 \mathcal{S}'_{t^{-1}} w(t)
+ \mathcal{S}_{t^{-1}} D^2 \dot w (t)\right)\, ,\label{e:verbatim5}
\end{align}
where $\circ$ denotes a suitable product structure. 
Now, as already argued for $\mathscr{L} (\mathcal{S}_{t^{-1}} (w(t))$, for any natural number $\kappa$ we have
\begin{equation}\label{e:verbatim6}
\|D' (t)\|_\kappa + \|D'' (t)\|_\kappa \leq C(\kappa) (1 + t^{\kappa-1})\, . 
\end{equation}
Moreover, having derived the bound $\|w (t)\|_4 \leq C t$, we can 
take advantage of Proposition~\ref{p:convolution} to get
\begin{align}\label{e:verbatim7}
\|D(t)\|_0 \leq C \left(t^{-3} \|w (t)\|_4 + \|\dot w (t)\|_2\right) \leq C t^{-2}\, .
\end{align}
In order to estimate the $C^3$ norm, we use \eqref{e:product}, \eqref{e:verbatim6} and argue similarly to get:
\begin{align}
\|D (t)\|_3 \leq C t^2 \left(\|\dot w(t)\|_2 + t^{-3} \|w(t)\|_4\right) + C \left(\|w (t)\|_4 + t\|\dot w (t)\|_4\right) \leq C t\label{e:verbatim8}\, .
\end{align}
Inserting the latter two inequalities in \eqref{e:intermedia}, we clearly conclude \eqref{e:Cauchy} and complete the proof. 
\end{proof}

\begin{proof}[Proof of Lemma~\ref{l:interpolation}] First of all, we observe that it suffices to prove all the claims for functions and in a local
coordinate patch: hence, without loss of generality we can just prove the claim for functions on balls of $\mathbb R^n$.

\medskip

{\bf Proof of \eqref{e:interpolation}.} By the classical extension theorems, it suffices to prove the inequality for functions defined
on the whole $\mathbb R^n$ (under the assumptions that all norms are finite!). In such a case we will in fact have the stronger inequality
\[
\|D^\sigma v\|_0 \leq C \|D^r v\|_0^\lambda \|D^s v\|_0^{1-\lambda}\, .
\]
Clearly, it suffices to prove the inequality in the particular case where $r = 0 <\sigma  <s$, where it takes the form
\[
\|D^\sigma v\|_0 \leq C\|D^s v\|_0^{\sigma/s} \|v\|_0^{1 - \sigma/s}\, .
\]
If $v\equiv 0$, then there is nothing to prove. If $D^s v \equiv 0$, since the function is bounded, then we have $D^\sigma v \equiv 0$ and again the
inequality is trivial.  Otherwise, recall that we have the following elementary bound, with a constant $C$ independent of $v$.
\[
\|D^\sigma v\|_0 \leq C\|D^s v\|_0 + C \|v\|_0\, .
\]
However, since we can rescale the function to $v_\varepsilon (r) = v (\varepsilon r)$, we also have the validity of 
\[
\|D^\sigma v\|_0 \leq C \varepsilon^{s-\sigma}\|D^s v\|_0 + C \varepsilon^{-\sigma}\|v\|_0\, ,
\]
with the very same constant $C$, i.e. independently of $\varepsilon > 0$. 
Choosing $\varepsilon = \|v\|_0^{1/s}\|D^s v\|_0^{-1/s}$ we conclude the proof.

\medskip

{\bf Proof of \eqref{e:composition}.} Again we can assume that the domain of the function is $\mathbb R^n$. 
Denoting by $D^j$ any partial derivative of order $j$, the chain rule can be written symbolically as
\begin{equation}\label{e:chainrule}
D^m(\Psi\circ v)=\sum_{l=1}^m(D^l\Psi)\circ v\sum_{\sigma}C_{l,\sigma}(Dv)^{\sigma_1}(D^2v)^{\sigma_2}\dots(D^mv)^{\sigma_m}
\end{equation}
for some constants $C_{l,\sigma}$,
where the inner sum is over $\sigma=(\sigma_1,\dots,\sigma_m)\in\mathbb N^m$ such that
\begin{equation*}
\sum_{j=1}^m\sigma_j=l,\quad \sum_{j=1}^mj\sigma_j=m.
\end{equation*}
From \eqref{e:interpolation} we have 
\[
\|u\|_j\leq C_h\|u\|_0^{1-\frac{j}{m}}\|u\|_m^{\frac{j}{m}} \qquad \mbox{ for $m\geq j\geq 0$} 
\]
(without loss
of generality we assume both $\|u\|_0$ and $\|u\|_m$ nonzero, otherwise the inequality is trivial: thus we can use \eqref{e:interpolation} also
for the ``extreme cases'' $\sigma =r$ and $\sigma =s$!). Inserting the latter inequality in \eqref{e:chainrule}, we easily achieve \eqref{e:composition}.

\medskip

{\bf Proof of \eqref{e:product}.} Using the notation above we write the Leibniz rule as
\[
D^m (\varphi \psi) = \sum_{j=0}^m \underbrace{C_{j,m} D^j \varphi D^{m-j} \psi}_{S_j}\, .
\]
For each summand we use \eqref{e:interpolation} and Young's inequality to write 
\[
\|S_j\|_0 \leq C \|\varphi\|_0^{1-j/m}\|\varphi\|_m^{j/m}\|\psi\|_0^{j/m}\|\psi\|_m^{1-j/m} \leq C\|\varphi\|_0\|\psi\|_m + C\|\varphi\|_m\|\psi\|_0\, .\qedhere
\]
\end{proof}

\section{Global existence and convergence to an isometry}

In this section we combine the bounds in Proposition~\ref{p:a priori} with a local solvability argument to show that there is a global solution to Nash's regularized flow equations. 

\begin{proposition}\label{p:existence}
There exist $t_0$ and $\delta$ such that, if $\|h\|_3\leq \delta$,  
then there is a  solution $t\mapsto w (t)$ of \eqref{e:ODE}--\eqref{e:integral} on $[t_0, \infty)$ which satisfies the bounds \eqref{e:a priori1}--\eqref{e:Cauchy} for every $t$. Moreover, for $t\to \infty$, $w(t)$ converges in $C^3$ to a free embedding $\bar u$ with $\bar u^\sharp e = w_0^\sharp e + h$.
\end{proposition}
\begin{proof}
The whole point lies in the following: 
\begin{itemize}
\item[(Loc)] assume $J = [t_0, t_1]$ is some closed interval (possibly trivial, namely, with $t_1=t_0$) over which we have a solution of \eqref{e:ODE}--\eqref{e:integral} satisfying the bounds \eqref{e:a priori1}--\eqref{e:Cauchy}. Then the solution can be prolonged on some open interval $[t_0, t_2[\supset [t_0, t_1]$ to a solution which satisfies the bounds \eqref{e:starting1}--\eqref{e:starting2}. 
\end{itemize}
The statement (Loc) and Proposition~\ref{p:a priori} easily imply the global existence claimed in the proposition. Indeed, if we let $[t_0, T)$ be the maximal interval over which there is a solution satisfying \eqref{e:a priori1}--\eqref{e:Cauchy}, the statement (Loc) with $t_1=t_0$ and the a priori estimates in Proposition~\ref{p:a priori} imply that $T> t_0$, since for $t_1=t_0$ we can simply set $w(t_0) = w$, $\dot h (t_0) =0$ and
all the bounds \eqref{e:a priori1}--\eqref{e:Cauchy} would be trivially true.  Moreover, if $T< \infty$, then the bounds in Proposition~\ref{p:a priori}
imply that the solution can be smoothly extended to $[t_0, T]$ and (Loc) contradicts the maximality of $T$, establishing the global existence. The convergence to a $C^3$ $\bar u$ follows from the bound \eqref{e:Cauchy}. In turn we have the bound
\[
\|d w (t) \odot d \dot w (t)\|_0 + \|d (\mathcal{S}_{t^{-1}} w (t)) \odot d \dot w (t)\|_0 \leq C t^{-4}\, ,
\]
so that all the integrals used in \eqref{e:formal2}--\eqref{e:formal5} converge in the uniform norm and define continuous functions. 
The computations in \eqref{e:formal2}--\eqref{e:formal5} are thus rigorous and yield $\bar u^\sharp e = w^\sharp e + h$. 

Hence, in what follows we will focus on the proof of (Loc).

\medskip

First of all, we rewrite \eqref{e:ODE}--\eqref{e:integral} in terms of  a fixed point for an integral operator on $(w, \lambda) := (w, \dot h)$. We start by writing 
\begin{equation}\label{e:integral1}
w (t) = w_0 + \int_{t_0}^t \mathscr{L} (\mathcal{S}_{\tau^{-1}}w(\tau)) \lambda (\tau)\, d\tau =: w_0 + \int_{t_0}^t \mathscr{W} (w (\tau), \lambda (\tau))\, d\tau\, .
\end{equation}
We then rewrite the function $E(t)$ of \eqref{e:E(t)} as
\begin{equation}
E (t) = 2 d (\mathcal{S}_{t^{-1}} w (t) - w(t)) \odot d (\mathscr{L} (\mathcal{S}_{t^{-1}} w(t)) \lambda (t)) =: \mathscr{E} (w(t), \lambda (t))\, .
\end{equation}
Finally,
\begin{align}
\lambda (t) & = \frac{d}{dt} \left\{ \mathcal{S}_{t^{-1}} \left[ \psi (t-t_0) h + \int_{t_0}^t \mathscr{E} (w(\tau), \lambda (\tau))\, \psi (t-\tau)\, d\tau\right]\right\}\nonumber\\
& =  \underbrace{\psi' (t-t_0) \mathcal{S}_{t^{-1}} h - t^{-2}\psi (t-t_0)\mathcal{S}'_{t^{-1}} h}_{=\mu (t)}   - t^{-2} \mathcal{S}'_{t^{-1}}  
\int_{t_0}^t \mathscr{E} (w(\tau), \lambda (\tau))\, \psi (t-\tau)\, d\tau\nonumber\\
 &\qquad + \mathcal{S}_{t^{-1}} \int_{t_0}^t \mathscr{E} (w (\tau), \lambda (\tau)) \psi' (t-\tau)d\tau\label{e:operatore}\, .
\end{align}
Observe now that the operator $\mathscr{W}$ is smooth on $C^4$,
because of the regularization of $\mathcal{S}_t$ (cf.~the proof of Proposition~\ref{p:a priori}). The operator $\mathscr{E}$ is locally Lipschitz from $C^4$ to $C^3$ (cf.~the proof of Proposition~\ref{p:a priori}) because it loses one derivative, but on the other hand the operators $\mathcal{S}_t$ and $\mathcal{S}'_t$ in front of the integrals in the above expressions regularize again from $C^3$ to $C^4$.
Hence the local existence in (Loc) follows from classical fixed point arguments.

\medskip

We briefly sketch the details for the reader's convenience. We consider an interval $J=[t_0, t_1]$ as in (Loc) and $t_2>t_1$, whose choice will be specified later. We consider a pair $(w, \lambda)\in C (J, C^4)$ which solves \eqref{e:integral1}--\eqref{e:operatore} and satisfies 
\begin{align}
\|w(t) -w_0\|_3 + t^{-1}\|w(t) -w_0\|_4 & \leq \varepsilon\, ,\label{e:parte1}\\
t^4 \|\lambda (t)\|_0 + \|\lambda (t)\|_4 &\leq 1\, .\label{e:parte2}
\end{align}
(and in case $t_0=t_1$ we simply set $w (t_0)=w_0$ and $\lambda (t_0) =0$). We consider next the space $X$ of pairs $(\underline w , \underline \lambda)\in C ([t_0, t_2], C^4)$ such that
\begin{itemize}
\item[(a)] $w = \underline{w}$ and $\lambda = \underline \lambda$ on the interval $J$;
\item[(b)] the following inequalities hold:
\begin{align}
\|\underline w(t) -w_0\|_3 + t^{-1}\|\underline w(t) -w_0\|_4 & \leq 2 \varepsilon\, ,\label{e:arriva1}\\
t^4 \|\underline \lambda (t)\|_0 + \|\underline \lambda (t)\|_4 & \leq 2\, .\label{e:arriva2}
\end{align}
\end{itemize}
On $X$ we consider the norm $\|(\underline w, \underline \lambda)\|_{4,0} := \max_{t\in [t_0, t_2]} (\|\underline w(t)\|_4 + \|\underline \lambda (t)\|_4)$. $X$ is clearly a complete metric space. We then consider the transformation $\mathscr{A}:X \to C ([t_0, t_2], C^4)$ given by $(\underline w, \underline h)\mapsto \mathscr{A} (\underline w, \underline h) = (\tilde{w}, \tilde{h})$ through the following formulas:
\begin{align*}
\tilde{w} (t) = &\; w_0 + \int_{t_0}^t \mathscr{W} (\underline{w} (\tau), \underline{\lambda} (\tau))\, d\tau\, ,\\
\tilde{\lambda} (t) = &\; \mu (t)  - t^{-2} \mathcal{S}'_{t^{-1}}  
\int_{t_0}^t \mathscr{E} (\underline{w} (\tau), \underline{\lambda} (\tau))\, \psi (t-\tau)\, d\tau
+ \mathcal{S}_{t^{-1}} \int_{t_0}^t \mathscr{E} (\underline{w} (\tau), \underline{\lambda} (\tau)) \psi' (t-\tau)d\tau\, .
\end{align*}
Now, if we assume $t_2 \leq t_1+1$, then $\max_t \|\mathscr{W} (\underline{w} (t), \underline{\lambda} (t)\|_4 \leq C$, because
of the estimates \eqref{e:arriva1}--\eqref{e:arriva2}. Hence we can estimate
\begin{align}
\|\tilde{w} (t) - \underline w (t_1)\|_0 \leq \int_{t_1}^{t_2} \|\mathscr{W} (\underline{w} (\tau), \underline{\lambda} (\tau)\|_4 \, d\tau \leq C (t_2-t_1)
\quad \forall t\geq t_1\, .
\end{align}
Similarly, since $\sup_t \|\mathscr{E} (\underline{w} (t), \underline{\lambda} (t))\|_3 \leq C$ and recalling the estimates of
Proposition~\ref{p:convolution}, we conclude that
\[
\|\tilde{\lambda} (t) - \underline{\lambda} (t_1)\|_4\leq \|\mu (t) - \mu (t_1)\|_4 + C (t_2-t_1) \qquad \forall t\geq t_1\, .
\]
From \eqref{e:parte1}--\eqref{e:parte2} and the smoothness of the map $\mu$, it is easy to see that \eqref{e:arriva1}--\eqref{e:arriva2} is valid for the pair $(\tilde{w}, \tilde{\lambda})$ provided $t_2-t_1$ is smaller than a certain threshold.
In particular, for $t_2-t_1$ small enough the operator $\mathscr{A}$ maps $X$ into itself.

It remains to show the contraction property. Consider two pairs $(w_1, \lambda_1), (w_2, \lambda_2)\in X$ and
$(\tilde{w}_i, \tilde{\lambda}_i) = \mathscr{A} (w_i, \lambda_i)$. Then, using the properties of the operators $\mathcal{S}_{t^{-1}}$ and $\mathcal{S}'_{t^{-1}}$ we easily conclude
\begin{align}
\|\tilde{w}_1 (t) - \tilde{w}_2 (t)\|_{4,0} & \leq \int_{t_1}^{t_2} \|\mathscr{W} (w_1 (\tau), \lambda_1 (\tau)) - \mathscr{W} (w_2 (\tau), \lambda_2 (\tau))\|_4\, d\tau\, ,\\ 
\|\tilde{\lambda}_1 (t) - \tilde{\lambda}_2 (t)\|_{4,0} & \leq  C \int_{t_1}^{t_2} \|\mathscr{E} (w_1 (\tau), \lambda_1 (\tau)) -
\mathscr{E} (w_2 (\tau), \lambda_2 (\tau))\|_3\, d\tau\, .
\end{align}
In turn, recalling the Lipschitz regularity of the operators $\mathscr{W}$ and $\mathscr{E}$, we easily achieve
\begin{align*}
\|\mathscr{A} (w_1, \lambda_1) - \mathscr{A} (w_2, \lambda_2)\|_{4,0} & = \|(\tilde{w}_1, \tilde{\lambda}_1) - (\tilde{w}_2, \tilde{\lambda}_2\|_{4,0}\nonumber\\
& \leq C (t_2-t_1)\|(w_1, \lambda_1)-(w_2, \lambda_2)\|_{4,0}\, .
\end{align*}
Again, it suffices to choose $t_2-t_1$ smaller than a certain threshold to conclude that $\mathscr{A} : X\to X$ is a contraction. 
\end{proof}

\section{Higher regularity of the map $\bar u$}

Finally, in this section we complete the proof of Theorem~\ref{t:perturbation} by showing the following result.

\begin{proposition}\label{p:higher}
The map $\bar u$ of Proposition~\ref{p:existence} belongs to $C^k$ if $h\in C^k$ for $k\geq 4$. 
\end{proposition}
\begin{proof}
The proof will be by induction on $k$. Assume that, under the assumption $h\in C^k$, we have shown that
\begin{align}
&\|w (t)\|_k + t^{-1} \|w (t)\|_{k+1} \leq C\, ,\label{e:k_1}\\
&t^{k+1}\|\dot h (t)\|_0 + \|\dot h (t)\|_{k+1}\leq C\label{e:k_2}\, ,
\end{align}
for some constant $C$ independent of $t$. We will then show that, under the assumption that
$h\in C^{k+1}$, the same set of estimates hold with $k+1$ in place of $k$, namely
\begin{align}
\|w (t)\|_{k+1} + t^{-1} \|w (t)\|_{k+2} & \leq C'\, ,\label{e:k+1_1}\\
t^{k+2}\|\dot h (t)\|_0 + \|\dot h (t)\|_{k+1} &\leq C'\label{e:k+1_2}\, ,
\end{align}
with a constant $C'$ which might be worse than $C$, but depends only on $k$ and $t_0$ (the latter is, however,
fixed in the statement of the proposition). 
Indeed the estimate for $\|w (t)\|_{k+1}$ will come from the following stronger claim: there is a function $\delta (s)$ which converges to $0$
as $s\to\infty$ and such that
\begin{equation}\label{e:finale}
\|w (t) - w (s)\|_{k+1}\leq \delta (s) \qquad \mbox{for all $t\geq s\geq t_0$.}
\end{equation}
The claim obviously would complete the proof of the proposition, because it clearly shows that $w(t)$ converges in $C^{k+1}$ as $t\uparrow\infty$. Hence, in the rest of the proof we will focus on showing \eqref{e:k+1_1}, \eqref{e:k+1_2}, and \eqref{e:finale}.

\medskip

We start by estimating $\dot w (t)$ using \eqref{e:ODE} and recalling the same arguments of the proof of Proposition~\ref{p:a priori}: from \eqref{e:k_1}, \eqref{e:k_2}, and Proposition~\ref{p:convolution} we conclude the bounds which are the analog of \eqref{e:weaker-1} and \eqref{e:weaker-2}, namely
\begin{equation}\label{e:appunti(1)}
t^{k+1}\|\dot w (t)\|_0 + \|\dot w (t)\|_{k+1} \leq C\, .
\end{equation}
We next estimate the function $E(t)$ of \eqref{e:E(t)}, again using the arguments of Proposition~\ref{p:convolution}. First, by Proposition~\ref{p:convolution}(c) and \eqref{e:k_1} we get
\begin{equation}
t^{k} \|\mathcal{S}_{t^{-1}} w (t) - w(t)\|_1 + \|\mathcal{S}_{t^{-1}} w (t) - w(t)\|_{k+1} \leq C t\, .
\end{equation}
Then, using \eqref{e:product} we conclude the bounds which are the analog of \eqref{e:bound_E(tau)}, namely
\begin{equation}\label{e:appunti(2)}
t^{k} \|E(t)\|_0 + \|E(t)\|_k \leq C t^{-2}\, .
\end{equation}
We next recall the computation for $\dot h (t)$:
\begin{align}
\dot h (t) =\;& \underbrace{\vphantom{\int_{t_0}^t}- \frac{\psi (t-t_0)}{t^2} \mathcal{S}'_{t^{-1}} h + \psi' (t-t_0) \mathcal{S}_{t^{-1}} h}_{=:A(t)} 
\underbrace{- \frac{1}{t^2} \mathcal{S}'_{t^{-1}}\overbrace{\int_{t_0}^t E (\tau) \psi (t- \tau)\, d\tau}^{=: L(t)}}_{=:B(t)}\nonumber\\
& + \underbrace{\mathcal{S}_t \int_{\max \{t_0, t-1\}}^t E (\tau) \psi' (t-\tau)\, d\tau}_{=: C(t)}\, .
\end{align}
Now, using that $h\in C^{k+1}$, Proposition~\ref{p:convolution}(c), and the fact that $\psi' (t-t_0)$ vanishes for $t-t_0>1$, we easily conclude that
\begin{equation}\label{e:A(t)}
t^{k+2}\|A (t)\|_0 + \|A (t)\|_{k+2} \leq C\, ,
\end{equation}
where the constant $C$ depends on $k$ and $t_0$ (which are both fixed). As for $C(t)$, we can use \eqref{e:appunti(2)} and Proposition~\ref{p:convolution}(b) to conclude
\begin{equation}\label{e:C(t)}
t^{k+2}\|C (t)\|_0 + \|C(t)\|_{k+2} \leq C\, .
\end{equation}
The estimate on $B(t)$ turns out to be more delicate. First notice that, by \eqref{e:appunti(2)}, we certainly conclude that $\|L(t)\|_k \leq C$. Using now Proposition~\ref{p:convolution}(c) we get however the weaker estimate
\begin{equation}\label{e:B(t)_weak}
\|B (t)\|_{k+2} \leq C t\, .
\end{equation}
We can now go back in the argument for \eqref{e:appunti(1)} and recover $\|\dot w (t)\|_{k+2} \leq C t^2$. In turn, plugging this information in the derivation of \eqref{e:appunti(2)}
we get $\|E (t)\|_{k+1} \leq C t^{-1}$. The latter bound can be used to estimate $\|L (t)\|_{k+1} \leq C \log t$ which in turn, using Proposition~\ref{p:convolution}(c), improves \eqref{e:B(t)_weak} to
\begin{equation}\label{e:B(t)_almost}
\|B (t)\|_{k+2} \leq C \log t\, .
\end{equation}
We can now iterate the whole process to reach, respectively,
\begin{align*}
\|\dot h (t)\|_{k+2} \leq \;& C \log t\, ,\\
\|\dot w (t)\|_{k+2} \leq \;& C \log t\, ,\\
\|w (t)\|_{k+2} \leq \;& C t \log t\, ,\\
\|E (t)\|_{k+1} \leq \;& C t^{-2} \log t\, .
\end{align*}
Since however $t^{-2} \log t$ is integrable on $[t_0, \infty)$, we achieve the desired bound $\|B (t)\|_{k+2} \leq C$ and indeed, using again Proposition~\ref{p:convolution}(c),
\begin{equation}\label{e:B(t)}
t^{k+2} \|B (t)\|_0 + \|B (t)\|_{k+2} \leq C\, .
\end{equation}
Clearly \eqref{e:A(t)}, \eqref{e:B(t)} and \eqref{e:C(t)} yield \eqref{e:k+1_2}. As already argued several times, we directly conclude $\|\dot w (t)\|_{k+2}\leq C$ and
$\|w (t)\|_{k+2} \leq C t$, namely \eqref{e:k+1_1}. Besides, following the same reasoning as above we also conclude the following useful bound: 
\begin{equation}\label{e:appunti(2)_improved}
t^{k+1} \|E(t)\|_0 + \|E(t)\|_{k+1} \leq C t^{-2}\, .
\end{equation}

Thus the only bound which remains to show is \eqref{e:finale}:
the argument, however, follows almost verbatim the one for \eqref{e:Cauchy}. We briefly sketch the details. First, we recall the computation in \eqref{e:verbatim1}. 
Then, using the bound \eqref{e:k+1_2} we derive the analog of \eqref{e:verbatim2}, namely 
\begin{equation}\label{e:verba_2}
\|h (t) - h(s)\|_0 \leq C s^{-k-1} \qquad \mbox{for all $t\geq s\geq t_0$.}
\end{equation}
Similarly, using \eqref{e:verbatim3} and \eqref{e:appunti(2)_improved} we derive 
\begin{align}
& \|h (t) - h(s)\|_{k+1}\nonumber\\
& \quad \leq C s^{-1} + \|\mathcal{S}_{t^{-1}} h - \mathcal{S}_{s^{-1}} h\|_{k+1} + \|\mathcal{S}_{t^{-1}} L (\infty) - \mathcal{S}_{t^{-1}} L (\infty)\|_{k+1}
\quad \forall t\geq s\geq t_0\, .\label{e:verba4}
\end{align}
Plugging these inequalities in \eqref{e:verbatim1} and using \eqref{e:anche_questa}, we derive the existence of a function $\bar\delta (s)$ which converges to $0$ as $s\to\infty$ and such that
\begin{equation}\label{e:intermedia_bis}
\|w (t) - w(s)\|_{k+1} \leq \bar\delta (s) +  C \int_s^t (\|D (\tau)\|_{k+1} \tau^{-k-1} + \|D (\tau)\|_0)\, d\tau\, .
\end{equation}
This replaces the analogous estimate \eqref{e:intermedia}, where $D(t)$ is the quantity defined in \eqref{e:verbatim1}. The estimate $\|D (\tau)\|_0 \leq \tau^{-2}$ of \eqref{e:verbatim7} is 
certainly valid here as well. In order to estimate $\|D (t)\|_{k+1}$ we first recall the computations in \eqref{e:verbatim5} and the quantities $D'(t)$ and $D'' (t)$ introduced there. Using the better bounds $\|w (t)\|_{k+2} \leq C t$ and \eqref{e:k+1_1}, the estimate in \eqref{e:verbatim8} can in fact be improved to
\begin{equation}\label{e:verba8}
\|D (t)\|_{k+1} \leq C t\, .
\end{equation}
Inserting the inequalities just found for $\|D(\tau)\|_0$ and $\|D (\tau)\|_{k+1}$ in \eqref{e:intermedia_bis}, we immediately conclude \eqref{e:finale}, which completes our proof. 
\end{proof}

\section{The nonclosed case}

The proof of Corollary~\ref{c:Ck_2} uses a construction very similar to that employed
Corollary~\ref{c:Nash+Whitney} to show the existence of a short embedding of a noncompact manifold. 

\begin{proof}[Proof of Corollary~\ref{c:Ck_2}] Consider an open covering $\{U_\ell\}_\ell$ as in Lemma~\ref{l:covering} and let $\mathcal{C}_i$ be the corresponding classes. As in the proof of Corollary~\ref{c:Nash+Whitney}, fix a family $\{\varphi_\ell\}_\ell$ of smooth functions with the properties that $\varphi_\ell \in C^\infty_c (U_\ell)$ and for every $p\in \Sigma$ there is at least one $\varphi_\ell$ which equals $1$ on a neighborhood of $p$. Moreover, having ordered $\{U_\ell\}_\ell$ we fix a (strictly) decreasing number of parameters $\varepsilon_\ell$, converging to $0$.

Next consider the map $v^0: \Sigma \to \mathbb R^{2 (n+1)}$ defined in the following way: for each $i\in \{1, \ldots , n+1\}$ and every $p\in \Sigma$, set
\begin{align*}
v^0_{2 (i-1) +1} (p) = \varepsilon^2_\ell \varphi_\ell (p) \quad \mbox{and} \quad v^0_{2i} (p) = \varepsilon_\ell \varphi_\ell (p)\, 
\end{align*}
if $p$ is contained in some $U_\ell \in \mathcal{C}_i$, otherwise we set them equal to $0$. As already shown in the proof of Corollary~\ref{c:Nash+Whitney}, the latter map is well-defined, and we let $h: = (v^0)^\sharp e$. Provided we choose the $\varepsilon_\ell$ sufficiently small, we have $g-h > 0$. 

For each $U_\ell$ fix a smooth map $\Phi_\ell$ which maps $U_\ell$ diffeomorphically on the standard sphere $\mathbb S^n\setminus \{N\}$, where $N$ denotes the north pole. We extend it to a smooth map on the whole manifold $\Sigma$ by defining $\Phi_\ell \equiv N$
on $\Sigma \setminus U_\ell$. If $\sigma$ denotes the standard metric on $\mathbb S^n$, we then select a sequence $\eta_\ell$ of sufficiently small positive numbers such that the tensor
\[
\tilde{g} := g-h - \sum_\ell  \eta_\ell \Phi_\ell^\sharp \sigma\, 
\]
is still positive definite. 
For each $U_\ell$ consider the tensor
$g_\ell := \varphi_\ell^2 \left(\sum_\ell \varphi_\ell^2\right)^{-1} \tilde{g}$, so that
\[
\sum_\ell g_\ell = \tilde{g}\, .
\]
Observe that, since $\Phi_\ell$ is a diffeomorphism on the support of $g_\ell$, which in turn is contained in $U_\ell$, the $(0,2)$ tensor
$\bar g_\ell := (\Phi_\ell^{-1})^\sharp g_\ell$ is well-defined on $\mathbb S^n \setminus \{N\}$ and can be extended smoothly
to $\mathbb S^n$ by setting it equal to $0$. Thus there is an isometric embedding $w^\ell$ of $\mathbb S^n$ into $\mathbb R^{N_0}$
such that $(w^\ell)^\sharp e = \bar g_\ell + \eta_\ell \sigma$. By applying a translation we can assume that $w^\ell$ maps the north pole
$N$ in $0$. Thus, $u^\ell := w^\ell\circ \Phi_\ell$ is a smooth map on $\Sigma$ which vanishes identically outside $U_\ell$ and such that
\[
(u^\ell)^\sharp e = g_\ell + \eta_\ell \Phi_\ell^\sharp \sigma\, .
\]
Now, for each $i\in \{1, \ldots, n+1\}$ we define the map $v^i: \Sigma \to \mathbb R^{N_0}$ setting
$v^i (p) = \varphi_\ell (p) u^\ell (p)$
if $p$ belongs to some $U_\ell \in \mathcal{C}_i$ and $0$ otherwise. Finally, let $u = v^0\times v^1 \times \ldots \times v^{n+1}$. Then it is obvious from the construction and from Remark~\ref{r:product} that $u$ is an isometry:
\[
u^\sharp e = (v^0)^ \sharp e + \sum_\ell g_\ell + \sum_\ell \eta_\ell \Phi_\ell^\sharp \sigma = h + \tilde{g} + \sum_\ell \eta_\ell \Phi_\ell^\sharp \sigma = g\, .
\]
It follows therefore that $u$ is necessarily an immersion.
The argument of Corollary~\ref{c:Nash+Whitney} finally shows that $u$ is injective and completes the proof.
Observe that, if we set instead 
\[
\tilde{g} := g - \sum_\ell \eta_\ell \Phi_\ell^\sharp \sigma\, , 
\]
and define analogously the maps $w^\ell$, $u^\ell$ and $v^i$ with $i\in \{1, \ldots , n+1\}$, 
the resulting map $\bar u = v^1 \times \ldots \times v^{n+1}$ is an isometric immersion of $\Sigma$: the only property which is
lost compared to $u$ is indeed the injectivity. 
\end{proof}

\chapter{Continuity of solutions of parabolic equations}

\section{Introduction} 

In 1958 Nash published his fourth masterpiece \cite{Nash1958}, a cornerstone in the theory of partial differential equations. His main theorem regarded bounded solutions of linear second-order parabolic equations with uniformly elliptic nonconstant coefficients. More precisely, equations of the form
\begin{equation}\label{e:parabolic}
\partial_t u = {\rm div}_x  (A (x,t) \nabla u)\, , 
\end{equation}
where:
\begin{itemize}
\item[(a)]  the unknown $u$ is a function of time $t$ and space $x\in \mathbb R^n$;
\item[(b)] $\partial_t u$ denotes the time partial derivative
$\frac{\partial u}{\partial t}$;
\item[(c)] $\nabla u$ denotes the spatial gradient, namely the vector
\[
\nabla u (x,t) = (\partial_1 u (x,t), \ldots , \partial_n u (x,t)) = \left( \frac{\partial u}{\partial x_1} (x,t), \ldots , \frac{\partial u}{\partial x_n} (x,t)\right)\, ,
\]
\item[(d)] and ${\rm div}_x V$ denotes the (spatial) divergence of the vector field $V$, namely
\[
{\rm div}_x V (x,t) = \partial_1 V_1 (x,t) + \ldots + \partial_n V_n (x,t)\, .
\]
\end{itemize}
Following the Einstein's summation convention on repeated indices, we will often write
\[
{\rm div}_x (A\nabla u) = \partial_i (A_{ij} \partial_j u)\, .
\]
\begin{ipotesi}\label{a:ellipticity} In this chapter the coefficients $A_{ij}$ will always satisfy the following requirements:
\begin{itemize}
\item[(S)] Symmetry, namely $A_{ij} = A_{ji}$;
\item[(M)] Measurability, namely each $(x,t) \mapsto A_{ij} (x,t)$ is a (Lebesgue) measurable function;
\item[(E)] Uniform ellipticity, namely there is a $\lambda \geq 1$ such that 
\begin{equation}\label{e:ellipticity}
\lambda^{-1} |v|^2 \leq A_{ij} (x,t) v_i v_j \leq \lambda |v|^2\qquad \mbox{$\forall (x,t)\in \mathbb R^n \times \mathbb R$ and $\forall v\in \mathbb R^n$.}
\end{equation}
\end{itemize}
\end{ipotesi}

Clearly, since the coefficients $A_{ij}$ are not assumed to be differentiable, we have to specify a suitable notion of solution for \eqref{e:parabolic}.

\begin{definition}
In what follows, the term {\em solution} of \eqref{e:parabolic} in an open domain $\Omega\subset \mathbb R^n \times \mathbb R$ will denote a locally summable function $u$ with locally square summable distributional derivatives $\partial_j u$ satisfying the identity
\begin{equation}\label{e:distributional}
\int u (x,t) \partial_t \varphi (x,t)\, dx\, dt = \int \partial_i \varphi (x,t) A_{ij} (x,t) \partial_j u (x,t)\, dx\, dt \qquad \forall \varphi\in C^\infty_c (\Omega)\, .
\end{equation}
\end{definition}
The following is then Nash's celebrated H\"older continuity theorem. As usual we denote by $\|f\|_\infty$ the (essential) supremum of the measurable function $f$ and, in case $f$ coincides with a continuous function a.e., we state pointwise inequalities omitting the ``almost everywhere'' specification.

\begin{theorem}[Nash's parabolic regularity theorem]\label{t:main_reg}
There are positive constants $C$ and $\alpha$ depending only upon $\lambda$ and $n$ with the following property. If the matrix $A$ satisfies Assumption~\ref{a:ellipticity} and $u$ is a bounded distributional solution of \eqref{e:parabolic} in $\mathbb R^n\times (0, \infty)$, then the following estimate holds for all $t_2\geq t_1> 0$ and all $x_1, x_2\in \mathbb R^n$:
\begin{equation}\label{e:Holder_est_par}
|u (x_1, t_1) - u(x_2, t_2)|\leq C \|u\|_\infty \left[\frac{|x_1-x_2|^\alpha}{t_1^{\sfrac{\alpha}{2}}} + 
\left(\frac{t_2-t_1}{t_1}\right)^{\frac{\alpha}{2 (1+\alpha)}}\right] \, .
\end{equation}
\end{theorem}

From the above theorem, Nash derived a fundamental corollary in the case of second-order elliptic equations
\begin{equation}\label{e:elliptic}
{\rm div}_x (A \nabla v) =0\, ,
\end{equation}
where the measurable coefficients $A_{ij}$ do not depend on $t$. 

\begin{definition}
If $\Omega$ is an open domain of $\mathbb R^n$, the term distributional solution $v$ of \eqref{e:elliptic} in $\Omega$ will denote a locally summable function $v$ with locally square summable distributional derivatives $\partial_j u$ satisfying the identity
\[
\int \partial_i v (x) A_{ij} (x) \partial_j \varphi (x)\, dx = 0 \qquad\qquad \forall \varphi\in C^\infty_c (\Omega)\, .
\]
\end{definition}

The following theorem is nowadays called De Giorgi--Nash theorem, since indeed De Giorgi proved it\footnote{In fact, De Giorgi's statement is stronger, since in his theorem $\|v\|_\infty$ in \eqref{e:Holder_est_ell} is replaced by the $L^2$ norm of $v$ (note that the power of $r$ should be suitably adjusted: the reader can easily guess the correct exponent using the invariance of the statement under the transformation $u_r (x) = u (rx)$).} independently
of Nash in \cite{DeGiorgi} (see \cite{DeGiorgiSel} for the English translation).

\begin{theorem}[De Giorgi--Nash]\label{t:DG-Nash}
There are positive constants $C$ and $\beta$ depending only upon $\lambda$ and $n$ with the following property. If the matrix $A$ satisfies Assumption~\ref{a:ellipticity} and $v$ is a bounded distributional solution of \eqref{e:elliptic} in $B_{3r} (z)\subset \Omega$, then the following estimate holds for every $x,y\in B_{r} (z)$:
\begin{equation}\label{e:Holder_est_ell}
|v (x)-v(y)|\leq C \|v\|_\infty r^{-\beta} |x-y|^\beta\, .
\end{equation}
\end{theorem}

Theorem~\ref{t:DG-Nash} was sufficient to give a positive answer to Hilbert's XIXth problem, namely the regularity of scalar minimizers of uniformly convex Lagrangians in any dimension, cf.~\cite[Teorema III]{DeGiorgi}. The case $n=2$ had been previously settled by Morrey in \cite{Morrey} and it was also known that the H\"older continuity of the first derivative of the minimizer would suffice to conclude its full regularity, see \cite{Hopf,Morrey2}.  The De Giorgi-Nash theorem closed the gap.\footnote{Indeed, it was known that the first partial derivatives of the minimizer satisfy a uniformly elliptic partial differential equation with measurable coefficients. De Giorgi's stronger version of Theorem~\ref{t:DG-Nash} would then directly imply the desired H\"older estimate. Nash's version was also sufficient, because a theorem of Stampacchia guaranteed the local boundedness of the first partial derivatives, cf.~\cite{Stampacchia}.}

The De Giorgi--Nash H\"older continuity theorem is false for elliptic {\em systems}, as it was noticed by De Giorgi in \cite{DeGiorgi2}.
In fact, for vectorial problems in the calculus of variations Ne\v{c}as proved later the existence of nondifferentiable minimizers of smooth uniformly convex functionals when both the domain and the target have sufficiently large dimension. The methods of Ne\v{c}as were refined further in \cite{HLN} and \cite{SY}, and recently the paper \cite{MoSa} used a different construction to show the existence
of a nondifferentiable minimizer when the target is $2$-dimensional and the domain $3$-dimensional. Since Morrey's work shows the regularity for planar minimizers even in the vectorial case, the latter example is in the lowest possible dimensions. Finally, in \cite{SY2} it was shown that if the domain is $5$-dimensional, vectorial minimizers might even be unbounded!

\medskip

Various authors rewrote, simplified and pushed further the De Giorgi--Nash theory.
The two most important contributors are probably Moser \cite{Moser1961}
and Aronson \cite{Ar}.
Moser introduced the versatile Moser iteration, 
based on the study of the time-evolution of successive powers of the solution,
which simplifies the proof (and avoids the explicit use of the entropy functional $Q$, see Definition~\ref{d:funzionali}). 
Moser further proved what is usually called Harnack inequality (although a more appropriate name in this
case would probably be ``Moser--Harnack''). For positive solutions $v$ of \eqref{e:elliptic}, the inequality 
is the estimate 
\[ 
\sup_{B_r (x)} v \leq C \inf_{B_{2r} (x)} v,
\]
where the constant $C$ only depends on $r$, the dimension $n$ and the ellipticity constant $\lambda$.

Aronson established a Gaussian type bound on the associated fundamental solution 
$S (x,t, \bar x, \bar t)$ (cf.~Theorem~\ref{t:fund_sol}), more precisely he bounded the latter from above and
from below with functions of the form
\[ \frac{K}{(t-\bar t)^{n/2}}\, e^{- B|x-\bar x|^2/(t-\bar t)}
\]
(Nash established the (weaker) upper bound with $K (t-\bar t)^{-\sfrac{n}{2}}$, cf.~Proposition~\ref{p:stime}). 

These three results, namely the H\"older continuity, the Moser--Harnack inequality,
and the Gaussian type bounds, are all connected and in some sense equivalent.
Fine expositions of this, as well as clever 
rewritings/simplifications/improvements of the proofs, 
can be found in Bass \cite[Ch. 7]{Ba:book}, \cite{Ba:Nash}
and Fabes and Stroock \cite{FS}.

\medskip

Most of the chapter will be dedicated to Nash's proof of Theorem~\ref{t:main_reg}, whereas Theorem~\ref{t:DG-Nash} will be derived from Theorem~\ref{t:main_reg} in the last section.

\section{Preliminaries and main statements}

Nash's approach to Theorem~\ref{t:main_reg} follows initially the well-known path of proving ``a priori estimates''. More precisely, standard arguments reduce Theorem~\ref{t:main_reg} to the following weaker version. In the rest of our discussion, we will use ``smooth'' to denote $C^\infty$ functions. All the statements will indeed hold under much less restrictive regularity assumptions, namely the existence and continuity of a suitable number of derivatives needed to justify the computations contained in the arguments. On the other hand, since such precise results are not needed later, in order to keep the presentation less technical we will ignore the issue. 

\begin{theorem}(A priori estimate)\label{t:main_reg_2}
Theorem~\ref{t:main_reg} holds under the additional assumptions that
\begin{itemize}
\item[(A1)] $A_{ij}$ is smooth on $\mathbb R^n \times \mathbb R$ for all $i,j = 1, \ldots n$;
\item[(A2)] $A_{ij} = \delta_{ij}$ outside of a compact set $K\times [0, T]$;
\item[(A3)] $u$ is smooth.
\end{itemize}
\end{theorem} 

Observe a crucial point: it is well known (and it was well known at the time Nash wrote his note) that the assumptions (A1)--(A3) imply the smoothness of any solution of \eqref{e:parabolic}, but the crucial point in Theorem~\ref{t:main_reg_2} is that the constants $C$ and $\alpha$ 
of \eqref{e:Holder_est_par} are {\em independent} of $A$ (more precisely, they depend only on the dimension $n$ and the constant $\lambda$ in \eqref{e:ellipticity}). We will focus on Theorem~\ref{t:main_reg_2} for most of the subsequent sections and only at the end, in Section \ref{s:para_tecnico}, we will
show how to conclude Theorem~\ref{t:main_reg} from it.\footnote{Nash does not provide any argument nor reference, he only briefly mentions that Theorem~\ref{t:main_reg} follows from Theorem~\ref{t:main_reg_2} using a regularization scheme and the maximum principle. Note that a derivation of the latter under the weak regularity assumptions of Theorem~\ref{t:main_reg} is, however, not entirely trivial: in Section \ref{s:para_tecnico} we give an alternative argument based on a suitable energy estimate.}

\medskip

Under the assumptions (A1)--(A3) of Theorem~\ref{t:main_reg_2} we take advantage of the existence of fundamental solutions. More precisely, we recall the following theorem (see \cite[Ch.~1.6]{Friedman}).

\begin{theorem}\label{t:fund_sol}
Under the assumptions of Theorem~\ref{t:main_reg_2}
there is a smooth map 
\[
(x,t,\bar x, \bar t) \mapsto S (x,t, \bar x, \bar t)
\] 
defined for $x, \bar x\in \mathbb R^n$ and $t>\bar t$ with the following properties:
\begin{itemize}
\item[(a)] The map $(x,t) \mapsto S (x,t, \bar x, \bar t) = T (x,t)$ is a classical solution of \eqref{e:parabolic} on $\mathbb R^n \times (\bar t, \infty)$.
\item[(b)] $T (\cdot, t)$ and $\partial_t^k T (\cdot, t)$ belong to the Schwartz space of rapidly decreasing smooth functions $\mathscr{S} (\mathbb R^n)$ and the corresponding seminorms can be bounded uniformly when $t$ belongs to a compact subset of $(\bar t, \infty)$.
\item[(c)] $T>0$ and $\int T (x,t)\, dx = 1$ for every $t>\bar t$.
\item[(d)] $T (\cdot, t)$ converges, in the sense of measures, to the Dirac mass $\delta_{\bar x}$ as $t\downarrow \bar t$, namely 
\[
\lim_{t\downarrow \bar t} \int T (x,t) \varphi (x)\, dx = \varphi (\bar x) 
\]
for any bounded continuous test function $\varphi$. Moreover, for any ball $B_r (\bar x)$, the function $T (\cdot, t)$ converges to $0$ on
$\mathbb R^n \setminus B_r (\bar x)$ with respect to all the seminorms of the Schwartz space.  
\item[(e)] For any $u$ bounded smooth solution of \eqref{e:parabolic} on $\mathbb R^n \times [\bar t, T[$ we have the representation formula
\begin{equation}\label{e:representation}
u (x,t) = \int S (x,t, y, \bar t) u (y, \bar t)\, d y\, .
\end{equation}
Vice versa, given a bounded smooth $u_0 (y) =: u (y, \bar t)$ the formula above gives the unique solution on $[\bar t, \infty[$ subject to the corresponding initial condition.
\item[(f)] The properties above hold for the map $(\bar x, \bar t)\mapsto S(x,t, \bar x, \bar t) = \bar T (\bar x, \bar t)$ on the domain $\mathbb R^n \times (-\infty, t)$, which therefore is a (backward in time) fundamental solution of the adjoint equation 
\begin{equation}\label{e:adjoint}
-\partial_{\bar t} \bar T = \partial_{\bar x_j} (A_{ij} \partial_{\bar x_i} \bar T)\, .
\end{equation}
\end{itemize}
\end{theorem}

Except for the smoothness, the existence of a map $S$ with all the properties listed above is given in \cite[Ch.~1]{Friedman}
(note that point (f) is proved in \cite[Th.~15]{Friedman}). The latter reference shows that $S$ has continuous first-order derivatives (in time and space) and continuous second-order derivatives in space when the coefficients $A_{ij}$ are $C^2$ (in fact $C^{1, \alpha}$, cf.~\cite[Th~10]{Friedman}). Decay properties for the function and its first-order space derivatives are then showed in 
\cite[Th~11]{Friedman}. The higher regularity (and the decay of higher derivatives) when the coefficients $A_{ij}$ are smooth and constant outside of a compact set, follows easily from the arguments given in \cite{Friedman},
and we have stated it only for completeness: indeed the arguments of Nash do not really need this additional information.

In the remaining sections we will derive several bounds on the map $S$ which will finally  lead to a proof of Theorem~\ref{t:main_reg_2}
through the representation formula \eqref{e:representation}. Three very relevant quantities which we will compute on the fundamental solutions are the energy, the entropy and the first moment.

\begin{definition}\label{d:funzionali}
Under the assumptions of Theorem~\ref{t:main_reg_2} let $T (x,t) := S (x,t,0,0)$, where $S$ is the map of Theorem~\ref{t:fund_sol}. 
We then introduce
\begin{itemize}
\item[(i)] The {\em energy} $E(t) := \int T (x,t)^2\, dx$.
\item[(ii)] The {\em entropy} $Q (t) := - \int T (x,t) \log T (x,t)\, dx$.
\item[(iii)] The {\em first moment} $M (t) := \int T (x,t) |x|\, dx$. 
\end{itemize}
\end{definition}

On each of these quantities (which by Theorem~\ref{t:fund_sol} are smooth on $(0, \infty)$)
Nash derives subtle crucial estimates, which we summarize in the following proposition.

\begin{proposition}[Bounds on the energy, the entropy and the moment]\label{p:stime}
Under the assumptions of Theorem~\ref{t:main_reg_2} there are positive constants $C_1, C_2, C_3$ and $C_4$, depending only upon $\lambda$ and $n$, such that the following holds. If $T, E, Q$ and $M$ are as in Definition~\ref{d:funzionali}, then
\begin{align}
&E (t) \leq C_1  t^{-\sfrac{n}{2}}\, ,\label{e:energy_bound}\\
&\|T (\cdot, t)\|_\infty \leq C_2 t^{-\sfrac{n}{2}}\, ,\label{e:infty_bound}\\
&Q(t) \geq - C_3 + \frac{n}{2} \log t\, ,\label{e:entropy_bound}\\
&C_4^{-1} t^{\sfrac{1}{2}} \leq M (t) \leq C_4 t^{\sfrac{1}{2}}\, .\label{e:moment_bound}
\end{align}
\end{proposition}

The last bound is in fact the cornerstone of Nash's proof. With it he derives subsequently what he calls {\em $G$ bound}.

\begin{definition}\label{d:G}
Let $T$ be as in Definition~\ref{d:funzionali} and consider the ``normalization'' $U$ of the fundamental solution:
$U (y, t) := t^{\sfrac{n}{2}} T ( t^{\sfrac{1}{2}} y, t)$. For any $\delta\in ]0,1[$ the $G_\delta$-functional is
\begin{equation}\label{e:G}
G_\delta (t) = \int e^{-|y|^2} \log (U (y, t) + \delta)\, dy\, .
\end{equation}
\end{definition}

\begin{proposition}[$G$ bound]\label{p:G_bound}
Under the assumptions of Theorem~\ref{t:main_reg_2} there are constants $C_5$ and $\delta_0$, depending only upon $\lambda$ and $n$, such that the following holds. If $G_\delta$ is as in Definition~\ref{d:G}, then 
\begin{equation}\label{e:G_bound}
G_\delta (t) \geq - C_5 (-\log \delta)^{\sfrac{1}{2}} \qquad \mbox{for all $\delta <\delta_0$.}
\end{equation}
\end{proposition}

In turn Proposition~\ref{p:G_bound} will be used in an essential way to compare fundamental solutions for
different source points. Observe in fact that the integrand defining $G_\delta$
is rather negative at those points $\xi$ which are close to $0$ (the ``source'' of the fundamental solution) and where at the
same time the value of $U$ is low. Our goal, namely bounding $G_\delta (t)$ from below by $- C (- \log \delta)^{\sfrac{1}{2}}$, is thus to gain control on such ``bad points''.  In particular Proposition~\ref{p:G_bound} allows to derive the central ``overlap estimate'' for fundamental solutions, namely the following result.

\begin{proposition}[Overlap estimate]\label{p:overlap}
Under the assumptions of Theorem~\ref{t:main_reg_2} there are positive constants $C$ and $\alpha$, depending only upon $\lambda$ and $n,$ such that, if $S$ is the map of Theorem~\ref{t:fund_sol}, then
\begin{equation}\label{e:overlap}
\int |S (x,t, x_1, \bar t) - S (x,t, x_2, \bar t)|\, dx \leq C \left(\frac{|x_1-x_2|}{(t-\bar t)^{\sfrac{1}{2}}}\right)^\alpha \qquad
\mbox{for all $t> \bar t$.}
\end{equation}
\end{proposition}

The H\"older estimate in space for a bounded solution $u$ is a direct consequence of the overlap estimate and of \eqref{e:representation}, whereas the estimate in time will follow from additional considerations taking into account the other bounds derived above.

After collecting some elementary inequalities in the next section, we will proceed, in the subsequent three sections, to prove the three Propositions~\ref{p:stime},~\ref{p:G_bound}, and~\ref{p:overlap}. We will then show in Section \ref{s:conclusione} how Theorem~\ref{t:main_reg_2} follows.

\section{Three elementary inequalities}

In deriving the estimates claimed in the previous section we will use three ``elementary'' inequalities on functions. All of them have been generalized in various ways in the subsequent literature and hold under less restrictive assumptions than those stated here: the statements given below are just sufficient for our purposes and I have tried to keep them as elementary as possible.

The first inequality is nowadays known as ``Nash's inequality''. In \cite{Nash1958} Nash credits the proof to Elias Stein.

\begin{lemma}[Nash's inequality]\label{l:stein}
There is a constant $C$, depending only upon $n$, such that the following inequality holds for any function $v\in \mathscr{S} (\mathbb R^n)$:
\begin{equation}\label{e:stein}
\left(\int_{\mathbb R^n} |v(x)|^2\, dx \right)^{1+\sfrac{2}{n}}\leq C \left(\int_{\mathbb R^n} |\nabla v (x)|^2\, dx\right) \left(\int_{\mathbb R^n} |v (x)|\right)^{\sfrac{4}{n}}\, . 
\end{equation} 
\end{lemma}

The second is a Poincar\'e-type inequality in a ``Gaussian-weighted'' Sobolev space.

\begin{lemma}[Gaussian Poincar\'e inequality]\label{l:poincare}
The following inequality holds for any bounded $C^1$ function $f$ on $\mathbb R^n$ with bounded derivatives and which satisfies
the constraint $\int e^{-|\xi|^2} f (\xi)\, d\xi = 0$:
\begin{equation}\label{e:poincare}
2 \int_{\mathbb R^n} e^{-|\xi|^2} f^2 (\xi)\, d\xi \leq \int_{\mathbb R^n} e^{-|\xi|^2} |\nabla f (\xi)|^2\, d\xi\, .
\end{equation}
\end{lemma}

The proof of the final inequality in \cite{Nash1958} is credited to Lennart Carleson:

\begin{lemma}[Carleson's inequality]\label{l:carleson}
There is a positive constant $c$, depending only on $n$, such that the following inequality holds for any positive function $T\in \mathscr{S} (\mathbb R^n)$ with $\int_{\mathbb R^n} T (x)\, dx=1$:
\begin{equation}\label{e:carleson}
\int_{\mathbb R^n} |x| T (x)\, dx \geq c \exp \left[ - \frac{1}{n} \int_{\mathbb R^n} T (x) \log T (x)\, dx \right]\, .
\end{equation}
 \end{lemma}

\begin{proof}[Proof of Lemma~\ref{l:stein}]
Consider the Fourier transform\footnote{In order to simplify the notation we omit the domain of integration when it is the entire space.} $\hat {v}$ of $v$:
\[
\hat v (\xi) := (2\pi)^{-\sfrac{n}{2}} \int e^{ix\cdot \xi}\, v (x)\, dx\, .
\]
Recalling the Plancherel identity and other standard properties of the Fourier transform we achieve
\begin{align}
\int |v (x)|^2\, dx & = \int |\hat v (\xi)|^2\, d\xi\label{e:Plancherel}\\
\int |\nabla v (x)|^2\, dx & = \int |\xi|^2 |\hat v (\xi)|^2\, d\xi\label{e:F-derivata}\\
|\hat v (\xi)| & \leq (2\pi)^{-\sfrac{n}{2}} \int |v (x)|\, dx \qquad \forall \xi\in \mathbb R^n\, .\label{e:infty-1}
\end{align}
Using \eqref{e:infty-1} we obviously get
\[
\int_{\{|\xi|\leq \rho\}} |\hat v (\xi)|^2\, d\xi \leq  C \rho^n \left(\int |v (x)|\, dx\right)^2\, ,
\]
whereas using \eqref{e:F-derivata} we have
\[
\int_{\{|\xi|\geq \rho\}} |\hat {v} (\xi)|^2\, d\xi \leq \int \frac{|\xi|^2}{\rho^2} |\hat v (\xi)|^2\, d\xi = \frac{1}{\rho^2} \int |\nabla v (x)|^2\, dx\, .
\]
Equation \eqref{e:Plancherel} and the last two inequalities can be combined to reach 
\begin{equation}\label{e:ottimizza}
\int |v (x)|^2\, dx \leq C \rho^n \left(\int |v (x)|\, dx\right)^2 + \frac{1}{\rho^2} \int |\nabla v (x)|^2\, dx\, ,
\end{equation}
where the constant $C$ is independent of $\rho$. 

Next, the inequality \eqref{e:stein} is trivial if $v$ or $\nabla v$ vanishes identically. Hence, we can assume that both integrals in the right-hand side of \eqref{e:stein} are nonzero. Under this assumption \eqref{e:stein} follows right away from \eqref{e:ottimizza} once we set 
\[
\rho = \left[\frac{\displaystyle\int |\nabla v (x)|^2\, dx}{\left(\displaystyle\int |v (x)|\, dx\right)^2}\right]^{\frac{1}{n+2}}\, .\qedhere
\]
\end{proof}

\begin{proof}[Proof of Lemma~\ref{l:poincare}]
Consider the Hilbert space $H$ of measurable functions $f$ such that $\int e^{-|\xi|^2} f^2 (\xi)\, d\xi < \infty$, with the scalar product
\[
\langle f, g\rangle := \int e^{-|\xi|^2} f (\xi) g (\xi)\, d\xi\, .
\]
It is well known that a Hilbert basis of $H$ is given by suitable products of the Hermite polynomials (cf.~\cite[Sec.~6.1]{Hermite_ref}): if $H_i$ denotes the Hermite polynomial of degree $i$ in one variable, suitably normalized, we define, for any $I = (i_1, \ldots , i_n) \in \mathbb N^n$
\[
H_I (\xi) = H_{i_1} (\xi_1) H_{i_2} (\xi_2) \cdot \ldots \cdot H_{i_n} (\xi_n)\, .
\]
We then have
\begin{align}
\int e^{-|\xi|^2} f^2 (\xi)\, d\xi  & = \sum_I \alpha_I^2\, ,\label{e:alfa}\\
\int e^{-|\xi|^2} (\partial_{\xi_j} f)^2 (\xi)\, d\xi & = \sum_I \beta_{I, j}^2\, ,
\end{align} 
where
\begin{align}
\alpha_I = & \int e^{-|\xi|^2} f (\xi) H_I (\xi)\, d\xi\, ,\\
\beta_{I, j} = & \int e^{-|\xi|^2} \partial_{\xi_j} f (\xi) H_I (\xi)\, d\xi\, .
\end{align}
Integrating by parts and using the relation
\[
\partial_{\xi_j} (e^{-|\xi|^2} H_I (\xi)) = (2 i_j)^{\sfrac{1}{2}} H_I (\xi)
\]
we easily achieve the identity
\[
\sum_{j=1}^n \beta_{I, j}^2 = 2 |I| \alpha_I^2\, .
\]
Therefore we conclude
\begin{align}
\int e^{-|\xi|^2} |\nabla f (\xi)|^2\, d\xi &= 2 \sum_I |I| \alpha_I^2\, .\label{e:2alfa}
\end{align}
Note that $|I|\leq 1$ as soon as $I\neq (0,0,\ldots 0)$. Thus,
the inequality \eqref{e:poincare} is a trivial consequence of \eqref{e:alfa} and \eqref{e:2alfa} provided $\alpha_{(0,0, \ldots , 0)}=0$.
Since the Hermite polynomial $H_0$ is simply constant, the latter condition is equivalent to $\int e^{-|\xi|^2} f (\xi)\, d\xi =0$. 
\end{proof}

\begin{proof}[Proof of Lemma~\ref{l:carleson}]
For every fixed $\lambda\in \mathbb R$, consider the function $\ell (\tau) = \tau \log \tau + \lambda \tau$ on $(0, \infty)$. Observe that the function is convex, it converges to $0$ as $\tau\to \infty$ and converges to $\infty$ as $\tau\to \infty$. Its derivative
$\ell' (\tau) = \log \tau + (1+\lambda)$ vanishes if and only for $\tau_0 = e^{-1-\lambda}$ and moreover $\ell (\tau_0) = - e^{-\lambda -1} < 0$: the latter must thus be the minimum of the function and therefore
\[
\tau \log \tau + \lambda \tau \geq - e^{-\lambda -1} \qquad \mbox{for every positive $\tau$.} 
\]
In particular, for any choice of the real numbers $a>0$ and $b \in \mathbb R$ we have
\begin{equation}\label{e:carleson2}
\int (T(x) \log T (x) + (a |x| +b) T (x))\, dx \geq - e^{-b-1} \int e^{-a|x|}\, dx\, .
\end{equation}
In analogy with the quantities introduced in Definition~\ref{d:funzionali}, we consider the entropy and the moment, namely
\begin{align}
Q &:=- \int T(x) \log T (x)\, dx\, ,\\
M &:= \int |x| T (x)\, dx\, ,\\
\end{align}
and we let $D (n)$ be the dimensional constant
\[
D (n) := \int e^{-|x|}\, dx\, .
\]
Then we can rewrite \eqref{e:carleson2} as
\begin{equation}\label{e:carleson3}
- Q + a M + b \geq - e^{-b-1} a^{-n} D (n)\, 
\end{equation}
(where we have also used that $\int T (x) \, dx =1$). Set $a:= \frac{n}{M} >0$ and $e^{-b} = \frac{e}{D(n)} a^n$. Then \eqref{e:carleson3} turns into
\[
-Q + n - \log \left(\frac{e}{D(n)} \left(\frac{n}{M}\right)^n \right) \geq -1\, .
\]
In turn, the latter is equivalent to
\[
n - n \log n + \log D(n) + n \log M \geq Q\, .
\]
Exponentiating the latter inequality we conclude $M \geq c(n) e^{\sfrac{Q}{n}}$ for some positive constant $c(n)$, which is precisely inequality \eqref{e:carleson}.
\end{proof}

\section{Energy, entropy and moment bounds}

In this section we prove Proposition~\ref{p:stime}.

\begin{proof}[Proof of the energy estimate \eqref{e:energy_bound}] We differentiate $E$ and compute
\begin{align*}
E' (t) & =  2 \int T (x,t) \partial_t T (t,x)\, dx = 2 \int T (x,t) \partial_j (A_{ij} (x,t) \partial_i T (x,t))\, dx \\
& =  - 2 \int \partial_j T (x,t) A_{ij} (x,t) \partial_j T (x,t)\, dx \leq -2 \lambda^{-1} \int |\nabla T (x,t)|^2\, dx\\
& \stackrel{\eqref{e:stein}}{\leq}  - C \left(\int |T (x,t)|^2\, dx \right)^{1+\sfrac{2}{n}} = - C E^{1+\sfrac{2}{n}}\, ,
\end{align*}
where $C$ is a positive constant depending only upon $\lambda$ and $n$. Note moreover that in the last line
we have used $\int T (x,t)\, dx = 1$. Since $E (t)$ is positive for every $t>0$ we conclude that $\frac{d}{dt} E (t)^{-\sfrac{2}{n}} \geq C>0$. By 
Theorem~\ref{t:fund_sol}(d), $\lim_{t\downarrow 0} E (t)^{-1} = 0$ and thus we can integrate the differential inequality to conclude that
\[
E (s)^{-\sfrac{2}{n}} = \int_0^s \frac{d}{dt} E (t)^{-\sfrac{2}{n}}\, dt \geq Cs\, ,
\]
which in turn implies $E (s) \leq C_1 s^{- \sfrac{n}{2}}$, where $C_1$ depends only upon $\lambda$ and $n$. \end{proof}

\begin{proof}[Proof of the uniform bound \eqref{e:infty_bound}] By translation invariance, from the energy estimate we conclude
\[
\int |S(x,t, \bar x, \bar t)|^2\, dx \leq C (t-\bar t)^{-\sfrac{n}{2}}\, .
\]
By Theorem~\ref{t:fund_sol}(f) the above argument applies to the adjoint equation to derive also the bound
\[
\int |S(x,t, \bar x, \bar t)|^2\, d\bar x \leq C (t-\bar t)^{-\sfrac{n}{2}}\, .
\]
On the other hand, using Theorem~\ref{t:fund_sol}(e), we have
\[
T (x,t) = \int S (x,t, \bar x, {\textstyle{\frac{t}{2}}}) T (\bar x, {\textstyle{\frac{t}{2}}})\, d\bar x\, .
\]
Using the Cauchy--Schwarz inequality, we then conclude
\begin{align}
|T (x,t)|^2 & \leq  E (\textstyle{\frac{t}{2}}) \int |S (x,t, \bar x, {\textstyle{\frac{t}{2}}})|^2\, d\bar x \leq C t^{-n}\, .
\end{align}
\end{proof}

\begin{proof}[Proof of the entropy bound \eqref{e:entropy_bound}] The $L^\infty$ bound and the monotonicity of the logarithm gives easily
\[
Q (t) \geq - \log \|T (\cdot, t)\|_\infty \int T (x,t)\, dx = - \log \|T (\cdot, t)\|_\infty \geq - C + \frac{n}{2} \log t\, .
\]
\end{proof}

\begin{proof}[Proof of the moment bound \eqref{e:moment_bound}] The first ingredient is Lemma~\ref{l:carleson}, which gives $M (t) \geq C e^{Q(t)/n}$. Next, differentiating the entropy we get
\begin{align*}
Q' (t) & =  - \int (1+\log T (x,t)) \partial_t T (x,t)\, dx = - \int (1+\log T (x,t)) \partial_j (A_{ij} (x,t) \partial_i T (x,t))\, dx\\
& =  \int \partial_j \log T(x,t) A_{ij} (x,t) \partial_i T(x,t)\, dx\\
& = \int \left(\partial_j \log T(x,t) A_{ij} (x,t) \partial_i \log T(x,t)\right)\, T (x,t)\, dx\, \\
& \geq \lambda^{-1} \int |A (x,t) \nabla \log T (x,t)|^2 T (x,t)\, dx \, .
\end{align*}
Recall that $\int T (x,t)\, dx = 1$ to estimate further
\begin{align*}
Q'(t) & \geq \lambda^{-1} \left(\int |A (x,t) \nabla \log T (x,t)| T (x,t)\, dx\right)^2
= \lambda^{-1}  \left( \int |A (x,t) \nabla T (x,t)|\, dx\right)^2\, .
\end{align*}
Whereas, differentiating the momentum:
\begin{align*}
M ' (t) =& \int |x| \partial_j (A_{ij} (x,t) \partial_i T(x,t))\, dx = - \int \frac{x_j}{|x|} A_{ij} (x,t) \partial_i T (x,t)\, dx\, .
\end{align*}
We thus conclude $|M'(t)|^2 \leq \lambda\, Q' (t)$. 

Let us summarize the inequalities relevant for the rest of the argument, namely the entropy bound \eqref{e:entropy_bound}, Carleson's inequality, and the one just derived:
\begin{align}
Q (t) & \geq - C_3 + \frac{n}{2} \log t\, ,\label{e:e_una}\\
M (t) & \geq C e^{Q(t)/n}\, ,\label{e:e_due}\\
Q' (t)^{\sfrac{1}{2}} & \geq \lambda^{-\sfrac{1}{2}} |M' (t)|\, .\label{e:e_tre}
\end{align}
Recall  moreover that, from Theorem~\ref{t:fund_sol}(d), $\lim_{t\downarrow 0} M (t) =0$. We thus set $M (0)=0$: this information
and the three inequalities above will allow us to achieve the desired bound.

Define $n R (t) = Q (t) +C_3 - \frac{n}{2} \log t$. Observe that $Q' (t) = n R' (t) + \frac{n}{2t}$. Hence we can use \eqref{e:e_due}
and integrate \eqref{e:e_tre} to achieve
\begin{equation}\label{e:e_quattro}
c_1 t^{\sfrac{1}{2}} e^{R (t)} \leq M (t) \leq c_2 \underbrace{\int_0^t \left({\textstyle{\frac{1}{2s}}} + R' (s)\right)^{\sfrac{1}{2}}\, ds}_{=: I(t)}\, .
\end{equation}
Using the concavity of $\xi\mapsto (1+\xi)^{\sfrac{1}{2}}$ on $[-1, \infty)$, we conclude that $(1+\xi)^{\sfrac{1}{2}} \leq 1 + \frac{\xi}{2}$ and thus 
\[
\left({\textstyle{\frac{1}{2s}}} + R' (s)\right)^{\sfrac{1}{2}} \leq \left(\frac{1}{2s}\right)^{\sfrac{1}{2}} \left(1 + \frac{1}{2} 
R' (s) 2s \right) = (2s)^{-\sfrac{1}{2}} + \left({\textstyle{\frac{s}{2}}}\right)^{\sfrac{1}{2}} R' (s) \, .
\]
Hence
\begin{align*}
I(t) & \leq  \int_0^t (2s)^{-\sfrac{1}{2}}\, ds + \int_0^t \left({\textstyle{\frac{s}{2}}}\right)^{\sfrac{1}{2}} R' (s)\, ds
= (2t)^{\sfrac{1}{2}} + \left({\textstyle{\frac{t}{2}}}\right)^{\sfrac{1}{2}} R(t)  - \int_0^t (8s)^{-\sfrac{1}{2}} R (s)\, ds\\
& \leq   (2t)^{\sfrac{1}{2}} + \left({\textstyle{\frac{t}{2}}}\right)^{\sfrac{1}{2}} R(t)\, .
\end{align*}
Inserting the latter inequality in \eqref{e:e_quattro} and dividing by $t^{\sfrac{1}{2}}$ we conclude that
\begin{equation}\label{e:e_cinque}
e^{R(t)} \leq \frac{c_3 M (t)}{t^{\sfrac{1}{2}}} \leq c_4 \left(1 + \frac{R(t)}{2}\right)\, ,
\end{equation}
where $c_3$ and $c_4$ are positive constants (depending only upon $n$ and $\lambda$). Now, the map
\[
\rho \mapsto e^\rho - c_4 \left(1+\frac{\rho}{2}\right)\, 
\]
converges to $\infty$ for $\rho\uparrow \infty$ and thus \eqref{e:e_cinque} implies that $R(t)$ is bounded by a constant 
which depends only upon $\lambda$ and $n$. In turn, again from \eqref{e:e_cinque}, we conclude \eqref{e:moment_bound}. 
\end{proof}

\section{$G$ bound}

In this section we prove Proposition~\ref{p:G_bound}. We will use in an essential way the bounds of Proposition~\ref{p:stime},
especially the moment bound. 

\medskip

We begin by noting the obvious effect of the normalization $U (\xi, t) = t^{\sfrac{n}{2}} T (t^{\sfrac{1}{2}} \xi, t)$.
All the estimates of Proposition~\ref{p:stime}  turn into corresponding ``time-independent'' bounds, which we
collect here:
\begin{align}
 &\int U (\xi, t)\, d\xi = 1\, ,\label{e:L1-norm}\\
 &\int |U (\xi,t)|^2\, d\xi \leq C\, , \label{e:L2-norm}\\
 &\|U (\cdot, t)\|_\infty \leq C\, , \label{e:Linfty_norm}\\
 C^{-1} \leq &\int |\xi| |U (\xi, t)|\, d\xi \leq C\, , \label{e:moment_norm}
\end{align}
for some constant $C$ depending only on $\lambda$ and $n$.

Moreover, the parabolic equation for $T$ transforms into the equation
\begin{equation}\label{e:par_norm}
2 t \partial_t U (\xi, t) = n U (\xi, t) + \xi_i \partial_i U (\xi, t) +2\partial_j (A_{ij} (t^{\sfrac{1}{2}} \xi, t) \partial_i U (\xi,t))\, ,
\end{equation}
and observe that the ``rescaled'' coefficients $\bar A_{ij} (\xi, t) := A_{ij} (t^{\sfrac{1}{2}} \xi, t)$ satisfy the same ellipticity condition
as $A_{ij}$, namely $\lambda^{-1} |v|^2 \leq \bar A_{ij}  v_i v_j \leq \lambda |v|^2$.

Differentiating \eqref{e:G} we achieve
\begin{align}
2 t G_\delta' (t) & =  \int e^{-|\xi|^2} \frac{2t \partial_t U (\xi, t)}{U (\xi, t) + \delta}\, d\xi\nonumber\\
&
\stackrel{\eqref{e:par_norm}}{=}  \underbrace{n\int e^{-|\xi|^2} \frac{U (\xi,t)}{U (\xi, t) + \delta}\, d\xi}_{=: H_1 (t)\geq 0} +
\underbrace{\int e^{-|\xi|^2} \frac{\xi \cdot \nabla U (\xi, t)}{ U(\xi, t) + \delta}\, d\xi}_{=:H_2 (t)}\nonumber\\
&\qquad\qquad +  \underbrace{2 \int e^{-|\xi|^2} \frac{\partial_j (\bar A_{ij} (\xi,t) \partial_i U (\xi,t))}{U (\xi, t) + \delta}\, d\xi}_{=:H_3 (t)}\, .\label{e:G'}
\end{align}
As for $H_2$, integrating by parts we get
\begin{align}
H_2 (t) & =  \int e^{-|\xi|^2} \xi \cdot \nabla (\log (U (\xi, t) + \delta)\, d\xi = - \int e^{-|\xi|^2} (n - 2|\xi|^2) \log (U (\xi, t) + \delta)\, d\xi\nonumber\\
= & - n G_\delta (t) + 2 \int e^{-|\xi|^2}|\xi|^2  \left(\log \delta + \log \left(1 + \delta^{-1} U (\xi, t)\right)\right)\, d\xi\nonumber\\
\geq & - n G_\delta (t) + 2 \log \delta \int |\xi|^2 e^{-|\xi|^2}\, d\xi \geq - n G_\delta (t) + C \log \delta\, .\label{e:H_2_bound}
\end{align}
Finally, integrating by parts $H_3$:
\begin{align}
H_3 (t) & =  - 2 \int \partial_j \left(e^{-|\xi|^2} (U (\xi, t) + \delta)^{-1}\right) \bar A_{ij} (\xi, t) \partial_i U (\xi, t)\, d\xi\nonumber\\
&= 4 \int e^{-|\xi|^2} \xi_j \bar{A}_{ij} (\xi, t) \frac{\partial_i U (\xi, t)}{U (\xi, t) + \delta}\, d\xi
+ 2 \int e^{-|\xi|^2} \frac{\partial_j U (\xi, t) \bar{A}_{ij} (\xi, t) \partial_i U (\xi, t)}{(U (\xi, t) + \delta)^2}\, d\xi\nonumber\\
& =  \underbrace{4\int e^{-|\xi|^2} \xi_j \bar{A}_{ij} (\xi, t) \partial_i \log (U(\xi, t) + \delta)\, d\xi}_{:= H_4 (t)}\nonumber\\
&\qquad + \underbrace{2\int e^{-|\xi|^2} \partial_j \log (U(\xi, t) + \delta) \bar{A}_{ij} (\xi, t) \partial_i \log (U (\xi, t) + \delta)\, d\xi}_{=: H_5 (t)}\, .
\label{e:H_3}
\end{align}
Note first that, by the ellipticity condition, the integrand of $H_5 (t)$ is indeed nonnegative.

Next, for each $(\xi, t)$ consider the quadratic form $\mathcal{A} (v,w) = \bar{A}_{ij} (\xi, t) v_i w_j$. The ellipticity condition guarantees that this is a scalar product. Hence, we have the corresponding Cauchy--Schwarz inequality $|\mathcal{A} (v,w)|^2 \leq \mathcal{A} (v,v) \mathcal{A} (w,w)$. Using this observation, $H_4 (t)$ can be bounded by
\begin{align}
|H_4 (t)| & \leq  4 \int e^{-|\xi|^2} \left(\xi_i \bar{A}_{ij} (\xi,t) \xi_j\right)^{\sfrac{1}{2}} \left(\partial_h \log (U(\xi, t) +\delta) 
\bar{A}_{hk} (\xi, t) \partial_k \log (U (\xi, t) + \delta)\right)^{\sfrac{1}{2}}\, d\xi\nonumber\\
& \leq 4 \left(\int e^{-|\xi|^2} \xi_j \bar{A}_{ij} (\xi, t) \xi_j \, d\xi\right)^{\sfrac{1}{2}} H_5 (t)^{\sfrac{1}{2}}\nonumber\\
& \leq  C  H_5 (t)^{\sfrac{1}{2}}\, .\label{e:bound_H'_3}
\end{align}
Inserting \eqref{e:bound_H'_3}, \eqref{e:H_3} and \eqref{e:H_2_bound} in \eqref{e:G'} we conclude the intermediate inequality
\begin{equation}\label{e:G'_2}
2 t G_\delta' (t) \geq C \log \delta - n G_\delta (t) - C H_5 (t)^{\sfrac{1}{2}} + H_5 (t)\, .
\end{equation}
The moment bound \eqref{e:moment_norm} will be used in a crucial way to prove the following 

\begin{lemma}\label{l:H_5_cattivo}
There are positive constants $\bar G$ and $\bar c$, both depending only upon $\lambda$ and $n$, such that,
if $\delta\leq 1$ and $G_\delta (t) \leq - \bar G$, then $H_5 (t) \geq \bar c (1- G_\delta (t))^2$. 
\end{lemma}

We postpone the proof of the lemma after showing how Proposition~\ref{p:G_bound} follows easily from it and from
the inequality \eqref{e:G'_2}. First of all observe that, under the assumption that $G_\delta (t) \geq - \tilde{G}\geq \bar G$,
if the constant $\tilde{G}$ is chosen sufficiently large, then $H_5 (t) - C H_5 (t)^{\sfrac{1}{2}} \geq \bar{c}{2} G_\delta (t)^2$. Hence, we conclude the existence of positive constants $\tilde {G}$, $\tilde{c}$, $C$ (depending only upon $\lambda$ and $n$) such that
\begin{equation}\label{e:quasi_finito}
2 t G'_\delta (t) \geq \tilde c G_\delta (t)^2 + C \log \delta \qquad \mbox{if $G_\delta (t) \leq - \tilde{G}$ and $\delta\leq 1$.}
\end{equation}
Set therefore $C_5 := \left(\frac{C+1}{\tilde c}\right)^{\sfrac{1}{2}}$ and let $\delta_0 \leq 1$ be such that
\[
C_5 (-\log \delta_0)^{\sfrac{1}{2}} \geq  \tilde{G}\, .
\]
We now want to show that with these choices the estimate of Proposition~\ref{p:G_bound} holds. In fact, assume that $\delta \leq \delta_0$ and that at some point $\tau>0$ we have
\[
G_\delta (\tau) < - C_5 (- \log \delta)^{\sfrac{1}{2}}\, .
\]
By our choice of $\delta_0$ this would imply $G_\delta  (\tau) < - \tilde{G}$, which in turn implies, by \eqref{e:quasi_finito},
\begin{equation}\label{e:scoppia}
2\tau G'_\delta (\tau) \geq - \log \delta\, .
\end{equation}
In  particular, there is an $\varepsilon > 0$ such that $G_\delta $ is increasing on the interval $(\tau-\varepsilon, \tau)$. We then conclude
that $G_\delta (\tau -\varepsilon) < - C_5 (- \log \delta)^{\sfrac{1}{2}}$ and we can proceed further: it can only be that $G_\delta < - C_5 (- \log \delta)^{\sfrac{1}{2}}$ on the whole interval $(0, \tau)$. But then \eqref{e:scoppia} would be valid on $(0, \tau)$ and we would conclude that
\[
\lim_{\tau\downarrow 0} G_\delta (\tau) = -\infty\, ,
\]
contradicting the trivial bound $G_\delta > \log \delta$. 

In order to complete the proof of Proposition~\ref{p:G_bound} it remains to show that Lemma~\ref{l:H_5_cattivo} holds.
 
\begin{proof}[Proof of Lemma~\ref{l:H_5_cattivo}]
Observe that, by the ellipticity condition,
\begin{equation}\label{e:ell_again}
H_5 (t) \geq 2 \lambda^{-1} \int e^{-|\xi|^2} |\nabla \log (U (\xi, t) + \delta)|^2\, d\xi\, .
\end{equation}
We now wish to apply Lemma~\ref{l:poincare}. We set for this reason
\[
f (\xi) := \log (U (\xi, t) + \delta) - \pi^{-\sfrac{n}{2}} \int e^{-|\xi|^2} \log (U (\xi, t) + \delta)\, d\xi
= \log (U (\xi, t) + \delta) - \pi^{-\sfrac{n}{2}} G_\delta (t)\, .
\]
This choice achieves $\nabla f = \nabla \log (U +\delta)$ and $\int e^{-|\xi|^2} f (\xi)\, d\xi =0$. We can thus apply 
Lemma~\ref{l:poincare} which, combined with \eqref{e:ell_again}, gives
\begin{equation}\label{e:H_5_bound}
H_5 (t) \geq 4 \lambda^{-1} \int e^{-|\xi|^2}\left(\log (U (\xi, t) + \delta) - \pi^{-\sfrac{n}{2}} G_\delta (t)\right)^2\, d\xi\, .
\end{equation}
Consider now the following function $g$ on the positive real axis:
\[
g (u) := u^{-1} (\log (u+\delta) - \pi^{-\sfrac{n}{2}} G_\delta (t) )^2\, .
\]
Since $U$ is (strictly) positive, we have
\begin{equation}\label{e:G_below_stupid}
\pi^{-\sfrac{n}{2}} G_\delta (t) > \pi^{-\sfrac{n}{2}} \log \delta \int e^{-|\xi|^2}d\xi = \log \delta\, .
\end{equation}
Moreover $g$ is nonnegative and vanishes only at the only positive point $\bar u$ such that 
\[
\log (\bar u + \delta) = \pi^{-\sfrac{n}{2}} G_\delta (t)\, .
\]
Next, differentiating $g$ we find
\[
g' (u) = - u^{-2} (\log (u+\delta) - \pi^{-\sfrac{n}{2}} G_\delta (t) )^2 + 2 u^{-1} (u+\delta)^{-1} (\log (u+\delta) - \pi^{-\sfrac{n}{2}} G_\delta (t))\, .
\]
Hence the derivative $g'$ vanishes at $\bar u$ and at any other (positive) point $u_m$ which solves
\begin{equation}\label{e:zero_of_g'}
\underbrace{\log (u +\delta) - \pi^{-\sfrac{n}{2}} G_\delta (t) - 2 \frac{u}{u+\delta}}_{=: h (u)} = 0\, .
\end{equation}
The function $h (u)$ is negative for $u\leq \bar u$ and thus any solution of the equation must be larger than $\bar u$. In fact 
\[
h (\delta) =
\log 2 + \log \delta - \pi^{-\sfrac{n}{2}} G_\delta (t) -1 \stackrel{\eqref{e:G_below_stupid}}{\leq} \log 2 -1 < 0\, .
\]
Since $\delta \leq 1$, we certainly conclude that any solution $u_m$ of \eqref{e:zero_of_g'} must be larger than $\delta$. On the other hand,
differentiating $h$ we find
\[
h' (u) = \frac{2u}{(u+\delta)^2} - \frac{1}{u+\delta}\, ,
\]
which is strictly positive for $u\geq \delta$. 

We conclude that there is a unique point $u_m > \bar u$ which satisfies \eqref{e:zero_of_g'}. On the other hand 
\begin{equation}\label{e:asintoto}
\lim_{u\uparrow \infty} g (u) = 0\, .
\end{equation}
Hence $u_m$ must be a local maximum for $g$, and
$g$ is strictly decreasing on $]u_m, \infty[$.

Observe next that
\[
\log u_m < \log (u_m+\delta) \leq \pi^{-\frac{n}{2}} G_\delta (t) + 2\, .
\]
We therefore conclude that
\[
u_m < \exp (2+\pi^{-\frac{n}{2}} G_\delta (t)) =: U_0 (t)\, .
\]
Define
\[
U^* (\xi, t) :=
\left\{
\begin{array}{ll}
U (\xi, t) \qquad & \mbox{if $U (\xi, t) \geq U_0 (t)$,}\\ \\
0 \qquad &\mbox{otherwise.}
\end{array}\right.
\]
Summarizing we can bound
\begin{equation}\label{e:bound_H_5_2}
H_5 (t) \geq c \int e^{-|\xi|^2} g (U^* (\xi, t)) U^* (\xi, t)\, d\xi\, .
\end{equation}
Recalling \eqref{e:Linfty_norm}, we have $\|U^* (\cdot, t)\|_\infty \leq C$. If we set $\bar C = \max \{C, e^3\}$, we have
$\|U^* (\cdot, t)\|_\infty \leq \bar C$ and, at the same time,
$\bar C \geq e^3 \geq U_0 (t) \geq u_m$, because for $G_\delta (t)$ we have the trivial bound
\begin{equation}\label{e:trivial}
G_\delta (t) \leq \int \log (U (\xi,t) + \delta)\, d\xi\, \leq \int U (\xi,t)\, d\xi = 1\, .
\end{equation}
Using the monotonicity of $g$ on $]u_m, \infty[$ we then infer
\begin{equation}\label{e:bound_H_5_3}
H_5 (t) \geq c \int e^{-|\xi|^2} (\log (\bar C+ \delta) - \pi^{-\sfrac{n}{2}} G_\delta (t))^2 U^* (\xi, t)\, d\xi\, ,
\end{equation}
where $c$ is a small but positive constant (depending only on $\lambda$ and $n$) and $\bar C$ is a constant larger than $e^3$,
also depending only on $\lambda$ and $n$. In particular, the trivial bound \eqref{e:trivial} implies
\[
 \log (\bar C+ \delta) - \pi^{-\sfrac{n}{2}} G_\delta (t) = \pi^{-\sfrac{n}{2}} \left( \pi^{\sfrac{n}{2}} \log (\bar C + \delta) - 
G_\delta (t)\right) \geq \pi^{-\sfrac{n}{2}} (1- G_\delta (t)) \geq 0\, ,
\]
and we therefore conclude
\begin{align}
H_5 (t) & \geq c_0 (1- G_\delta (t))^2 \int e^{-|\xi|^2} U^* (\xi, t)\, d\xi\nonumber\\
 & = c_0 (1- G_\delta (t))^2 \underbrace{\int_{|\xi|\geq \exp (2+ G_\delta (t))} e^{-|\xi|^2} U (\xi, t)\, d\xi}_{=:I}\, .\label{e:bound_H_5_4}
\end{align}
Clearly, in order to complete the proof of the lemma we just need to show the existence of positive constants $\bar G$ and $\bar c$ such that
\[
G_\delta (t)\leq - \bar G \qquad \implies \qquad I \geq \bar c\, .
\]
Under the assumption $G_\delta (t)\leq - \bar G$, for any $\mu >0$ we can write
\[
I \geq e^{-\mu^2} \int_{\mu \geq |\xi|\geq \exp (2- \bar G)} U (\xi, t)\, d\xi
= e^{-\mu^2} \left(1 - \int_{|\xi|\leq \exp (2 - \bar G)} U (\xi,t)\, d\xi - \int_{|\xi|\geq \mu} U (\xi, t)\, d\xi \right)\, .
\]
Using \eqref{e:Linfty_norm} we have
\[
 \int_{|\xi|\leq \exp (2 - \bar G)} U (\xi,t)\, d\xi \leq C (\exp (2-\bar G))^n
\]
for a constant $C$ depending only on $n$ and $\lambda$. In particular, if we choose $\bar G$ large enough we can assume that the integral above is bounded by $\frac{1}{4}$. Next, using \eqref{e:moment_norm} we get
\[
\int_{|\xi|\geq \mu} U (\xi, t)\, d\xi \leq \frac{1}{\mu} \int U (\xi, t) |\xi|\, d\xi \leq \frac{C}{\mu}\, .
\]
Thus, it suffices to fix $\mu$ large enough so that the latter integral is also smaller than $\frac{1}{4}$. With such choice, $G_\delta (t) \leq - \bar G$ implies $I \geq \frac{1}{2} e^{-\mu^2}$, which thus completes the proof. 
\end{proof}

\section{Overlap estimate}

We are now ready to prove Proposition~\ref{p:overlap}. First of all we notice that, without loss of generality, we can assume
$\bar t =0$. We thus consider two fundamental solutions $S (x,t,x_1, 0)$ and $S (x, t, x_2, 0)$. Fix for the moment a positive time $t$ and set $\xi_i := x_i t^{-\sfrac{1}{2}}$ and

\[
U_i (\xi) := t^{\sfrac{n}{2}} S (t^{\sfrac{1}{2}}\xi, t, x_i, 0)\, .
\] 
By Proposition~\ref{p:G_bound} we have
\begin{equation}\label{e:G_i}
\int e^{-|\xi-\xi_i|^2} \log (U_i (\xi) + \delta)\, d\xi  \geq - C_5 (-\log \delta)^{\sfrac{1}{2}}
\end{equation}
for all $\delta\leq \delta_0$. In particular, in the rest of this paragraph we will certainly assume $\delta\leq 1$. 

We then add the two inequalities above to get
\begin{equation}\label{e:sum}
\int \left[e^{-|\xi-\xi_1|^2} \log (U_1 (\xi) + \delta) + e^{-|\xi-\xi_2|^2} \log (U_2 (\xi) + \delta)\right]\, d\xi  
\geq - 2 C_5 (-\log \delta)^{\sfrac{1}{2}}\quad \forall \delta\leq \delta_0\, .
\end{equation}
Let 
\begin{align*}
U_+  (\xi) &:= \max \{U_1 (\xi), U_2 (\xi)\}\, ,\\ 
U_- (\xi) &:= \min \{U_1 (\xi), U_2 (\xi)\}\, ,\\
f_+ (\xi) &:= \max \{ \exp (- |\xi-\xi_1|^2), \exp (-|\xi-\xi_2|^2)\}\, ,\\ 
f_- (\xi) &:= \min \{ \exp (- |\xi-\xi_1|^2), \exp (-|\xi-\xi_2|^2)\}\, .
\end{align*}
Recalling the elementary bound $ac + bd \leq \max \{a,b\}\max \{c,d\} + \min \{a,b\} \min \{c, d\}$ we then conclude
\begin{equation}\label{e:rearrang}
\int \left[f_+ (\xi) \log (U_+ (\xi) +\delta) + f_- (\xi) \log (U_- (\xi) +\delta)\right]\, d\xi \geq - 2C_5 (-\log \delta)^{\sfrac{1}{2}}\, . 
\end{equation}
Since $\delta\leq 1$, we have
\[
\log (U_+ (\xi) + \delta) \leq U_+ (\xi) \leq U_1 (\xi) + U_2 (\xi)\, ,
\]
and consequently we can bound
\begin{equation}\label{e:sciocca}
\int f_+ (\xi) \log (U_+ (\xi) + \delta)\, d\xi \leq \int (U_1 (\xi)+ U_2 (\xi))\, d\xi \leq 2\, .
\end{equation}
Next, we bound
\[
\log (U_- (\xi) + \delta) = \log \delta + \log (1 + \delta^{-1} U_- (\xi)) \leq \log \delta + \delta^{-1} U_- (\xi)\, ,
\]
and thus
\begin{equation}\label{e:furba}
\int f_- (\xi) \log (U_- (\xi) + \delta)\, d\xi \leq \log \delta \int f_- (\xi)\, d\xi + \delta^{-1} \int U_- (\xi)\, d\xi\, .
\end{equation}

Now, observe that $\int f_- (\xi) \, d\xi$ is simply a function $w$ of $|\xi_1-\xi_2|$, which is positive and decreasing. 
Thus, combining \eqref{e:rearrang}, \eqref{e:sciocca}, and \eqref{e:furba} we achieve
\begin{equation}\label{e:da_sotto}
\int U_- (\xi) \, d\xi\geq \max_{\delta \leq \delta_0} \delta 
\left[-2- w (|\xi_1-\xi_2|) \log \delta - 2 C_5 (- \log \delta)^{\sfrac{1}{2}}\right] = : \phi (|\xi_1-\xi_2|)\, .
\end{equation}
The function $\phi$ is nonnegative and decreasing. Considering the rescaling which defined
the $U_i$'s we then conclude
\begin{equation}\label{e:overlap_1}
\int \min \{S (x,t, x_1, 0), S (x,t, x_2,0)\}\, dx = \int U_- (\xi)\, d\xi \geq \phi \left(\frac{|x_1-x_2|}{t^{\sfrac{1}{2}}}\right)\, ,
\end{equation}
Next, recall the elementary identity
\[
|\sigma -\tau| = \sigma + \tau - 2 \min \{\sigma, \tau\}\, ,
\]
valid for every positive $\sigma$ and $\tau$. In particular, we can combine it with \eqref{e:overlap_1} to conclude
\begin{align}
\frac{1}{2} \int |S (x,t, x_1, 0) - S (x,t, x_2, 0)| \, dx & = 1 -  \int \min \{S (x,t, x_1, 0), S (x,t, x_2,0)\}\, dx\nonumber\\
& \leq 1 - \phi \left(\frac{|x_1-x_2|}{t^{\sfrac{1}{2}}}\right):= \psi \left(\frac{|x_1-x_2|}{t^{\sfrac{1}{2}}}\right)\, ,\label{e:overlap_2}
\end{align}
where $\psi$ is a positive increasing function strictly smaller than $1$ everywhere. Observe, moreover, that with the same argument we easily achieve 
\begin{align}
\frac{1}{2} \int |S (x,t, x_1, \bar t) - S (x,t, x_2, \bar t)| \, dx \leq \psi \left(\frac{|x_1-x_2|}{(t-\bar t)^{\sfrac{1}{2}}}\right)\, \label{e:overlap_3}\, ,
\end{align}
whenever $t\geq \bar t$. 

We will pass from \eqref{e:overlap_2} to \eqref{e:overlap} through an iterative argument. In order to implement such argument
we introduce the functions
\begin{align}
T_a (x,t) = &  \max \{S (x,t, x_1, 0) - S (x,t, x_2,0), 0\}\, ,\\
T_b (x,t) = & \max \{S (x,t, x_2, 0) - S (x,t, x_1, 0), 0\}\, ,
\end{align}
and
\[
A (t) := \int T_a (x,t)\, dx = \int T_b (x,t)\, dx = \frac{1}{2} \int |S (x,t,x_1, 0) - S (x, t, x_2, 0)|\, dx\, .
\]
Note, moreover, that although we have defined $A$ only for $t> \bar t$, from the first identity in the derivation of \eqref{e:overlap} and the properties of the fundamental solution, it is easy to see that $\lim_{t\downarrow 0} A (t) =1$. 

Furthermore, let $T^*_a (x,t, \bar t)$ and $T_b^* (x,t, \bar t)$ be the solutions of \eqref{e:parabolic} with respective initial data
$T_a (x, \bar t)$ and $T_b (x, \bar t)$ at $t$. Note therefore the identities
\begin{align}
T_a^* (x, t, \bar t) = & \int S (x,t, y, \bar t) T_a (y, \bar t)\, dy = \int S (x,t,y, \bar t) \underbrace{T_a (y, \bar t) T_b (z, \bar t) A(\bar t)^{-1}}_{=: \chi (y,z,\bar t)}\, dy\, dz\, ,\\
T_b^* (x,t, \bar t) = & \int S (x,t, z, \bar t) T_b (z, \bar t)\, dz = \int S (x,t, z, \bar t) \chi (y,z,\bar t)\, dy\, dz\, .
\end{align}
Moreover, $T_a^* (x,\bar t,\bar t) - T_b^* (x,\bar t,\bar t) = S(x,\bar t,x_1, 0) - S (x,\bar t, x_2, 0)$ and thus 
\[
T_a^* (x,t, \bar t) - T_b^* (x,t, \bar t) = S (x,t,x_1, 0) - S (x,t, x_2, 0) \qquad  \mbox{for every $t\geq \bar t$.}
\] 
We therefore conclude the inequality
\begin{equation}\label{e:propagazione}
|S (x,t, x_1, 0) - S (x,t, x_2, 0)| \leq \int |S(x,t, z, \bar t) - S (x,t,y, \bar t)| \chi (y,z, \bar t)\, dy\, dz\, .
\end{equation}

Note that, in principle, $A(t, \bar t)$ is defined for $t> \bar t$. On the other hand, it follows easily from the first equality in \eqref{e:overlap_2}, that $\lim_{t\downarrow \bar t} A (t, \bar t) =1$. 
Integrating \eqref{e:propagazione} we then obtain
\begin{equation}\label{e:da_iterare}
A (t) \leq \int \psi \left(\frac{|y-z|}{(t-\bar t)^{\sfrac{1}{2}}}\right) \chi (y,z,\bar t)\, dy\, dz\, \quad \forall t> \bar t\, .
\end{equation}
Observe in particular that
\begin{equation}\label{e:monotona}
A (t) < \int \chi (y,z, \bar t)\, dy \, dz = A (\bar t)\, . \qquad \forall t>\bar t\, ,
\end{equation}
namely $A$ is strictly monotone decreasing. 

Let $\varepsilon := \phi (1) = 1 -\psi (1)$ and define $\sigma := 1- \frac{\varepsilon}{4}$. For each natural number $k\geq 1$ we let
$t_k$ be the first time such that $A (t_k)\leq \sigma^k$, if such time exists. Since
\[
A( |x_1-x_2|^2) \leq \psi (1) = 1- \varepsilon < \sigma\, ,
\]
we have the inequality
\begin{equation}\label{e:t_1}
t_1 \leq |x_1-x_2|^2 \, .
\end{equation}
We wish to derive an iterative estimate upon $t_{k+1}- t_k$. 

In order to do so, we let $x_0 := \frac{x_1+x_2}{2}$ and define the moments
\begin{align}
M_a (t) & :=  \int |x-x_0| T_a (x,t)\, dx\, ,\\
M_b (t) &:= \int |x-x_0| T_b (x,t)\, dx\, ,\\
M_k  & := \max \{M_b (t_k), M_a (t_k)\}\, .
\end{align}
Strictly speaking the moments are not defined for $t=0$. However since the functions converge to $0$ as $t\downarrow 0$, we set $M_a (0) = M_b (0) = 0$. 
Observe that
\[
\int_{|y-x_0|\geq 2 \sigma^{-k} M_k} T_a (y,t_k)\, dy \leq  \frac{\sigma^k}{2 M_k} \int T_a (y, t_k) |y-x_0|\, dy \leq  \frac{\sigma^k}{2}\, .
\]
Moreover, an analogous estimate is valid for $T_b$. Since the total integral of $T_a (y, t_k)$ (respectively $T_b (z, t_k)$) is in fact $A (t_k) = \sigma^k$, we conclude
\begin{align}
\int_{|y-x_0|\leq 2 \sigma^{-k} M_k} T_a (y)\, dy \geq  \frac{\sigma^k}{2}\, ,\label{e:momento_100}\\
\int_{|z-x_0|\leq 2 \sigma^{-k} M_k} T_b (z)\, dz \geq  \frac{\sigma^k}{2}\, .\label{e:momento_101}
\end{align}
Consider the domain $\Omega_k := \{(y,z): |y-x_0|\leq 2 \sigma^{-k} M_k, |z-x_0|\leq 2 \sigma^{-k}M_k\}$ and its complement $\Omega_k^c$. 
Observe that on $\Omega_k$ we have $|y-z|\leq 4 \sigma^{-k} M_k$. Thus for $t' > t_k$ we can use \eqref{e:da_iterare} to estimate
\begin{align}
A (t') & \leq  \int_{\Omega_k^c} \chi (y,z,t_k)\, dy\, dz + \psi \left(4 \sigma^{-k} M_k (t'-t_k)^{-\sfrac{1}{2}}\right) \int_{\Omega_k} \chi (y,z, t_k)\, dy\, dz\nonumber\\
& \leq \int \chi (y,z,t_k)\, dy\, dz - \left[1-\psi \left(4 \sigma^{-k} M_k (t'-t_k)^{-\sfrac{1}{2}}\right)\right] \int_{\Omega_k} \chi (y ,z,t_k)\, dy\, dz\nonumber\\
& \leq  A (t_k) -  \left[1-\psi \left(4 \sigma^{-k} M_k (t'-t_k)^{-\sfrac{1}{2}}\right)\right] A (t_k)^{-1} \left(\frac{\sigma^k}{2}\right)^2\nonumber\\
& = \sigma^k \left[\frac{3}{4} + \frac{1}{4} \psi \left(4 \sigma^{-k} M_k (t'-t_k)^{-\sfrac{1}{2}}\right)\right]\, .\label{e:stima_t_k+1}
\end{align}
If we set
\[
t':= t_k+16 \sigma^{-2k} M_k^2\, ,
\]
then
\[
\psi \left(4 \sigma^{-k} M_k (t'-t_k)^{-\sfrac{1}{2}}\right) = \psi (1) = 1 - \varepsilon\, ,
\]
and \eqref{e:stima_t_k+1} gives
\[
A(t')\leq \sigma^k \left(1-\frac{\varepsilon}{4}\right) = \sigma^{k+1}\, .
\]
We thus infer the recursive estimate
\begin{equation}\label{e:t_k-ricorsivo}
t_{k+1}\leq t_k + 16 \sigma^{-2k} M_k^2\, .
\end{equation}
We wish next to estimate $M_k$. Observe that
\begin{align*}
T_a (x, t') & =  \max \{S (x,t', x_1, 0) - S (x, t', x_2, 0), 0\} = \max \{T_a^* (x,t', t) - T_b^* (x,t',t), 0\}\\
&  \leq T_a^* (x,t', t) = \int S (x,t', y,t) T_a (y,t)\, dy\, . 
\end{align*}
Now,
\begin{align*}
M_a (t') & = \int |x-x_0| T_a (x, t')\, dx \leq \int (|x-y| + |y-x_0|) S (x,t', y,t) T_a (y,t)\, dy\, dx\\
& = \int |y-x_0| T_a (y, t)\, dy + \int T_a (y, t) \int |x-y| S (x,t', y, t)\, dx\, dy\, .  
\end{align*}
Using the moment bound we then infer
\[
M_a (t') \leq M_a (t) + A(t) C_4 (t'-t)^{\sfrac{1}{2}}\, .
\]
This, and the analogous bound on $M_b (t')$, leads to the recursive estimate
\[
M_{k+1} \leq M_k + \sigma^{k+1} C_4 (t_{k+1} - t_k)^{\sfrac{1}{2}} \leq M_k (1+4 C_4)\, .
\]
Clearly, since $t_0 =0$ and $M_0 = M_a (t_0) = M_b (t_0) = \frac{|x_1-x_2|}{2}$, we have
\begin{equation}\label{e:M_k}
M_k \leq \frac{|x_1-x_2|}{2} (1+ C_4)^k\, .
\end{equation}
Thus the recursive bound \eqref{e:t_k-ricorsivo} becomes
\begin{equation}\label{e:t_k-ricorsivo-2}
t_{k+1} \leq t_k + 4 |x_1-x_2|^2 \big[\underbrace{\sigma^{-2} (1+C_4)^2}_{B}\big]^k\, .
\end{equation}
Summing \eqref{e:t_k-ricorsivo-2} and taking into account that $t_1 \leq |x_1-x_2|^2$ we clearly reach
\begin{equation}\label{e:esponenziale_t_k}
t_{k+1} \leq 4 |x_1-x_2|^2 \frac{B^{k+1}-1}{B-1}\leq 4 |x_1-x_2|^2 B^{k+1}\, ,
\end{equation}
where $B$ is a constant larger than $2$ which depends only on $\lambda$ and $n$ (if $B$ as defined in \eqref{e:t_k-ricorsivo-2} is smaller than $2$, we can just enlarge it by setting it equal to $2$).

We next set $t_0=0$ (and recall that $A (0):=\lim_{t\downarrow 0} A(t) = 1$).
Hence, for any $t\geq 0$ there is a unique natural number $k$ such that
\[
t_k \leq t < t_{k+1}\, .
\]
We then conclude
\begin{align}
\int |S (x,t, x_1, 0) - S (x,t, x_2, 0)|\, dx = A(t) \leq A (t_k) \leq \sigma^k \qquad \forall t \geq t_k\, .
\end{align}
Observe on the other hand that
\[
k +1 \geq - (\log B)^{-1} \log \frac{4 |x_1-x_2|^2}{t}\qquad \mbox{for all $t\geq t_k$.}
\]
If we set $ \alpha:= - 2 (\log B)^{-1} \log \sigma$, which is a positive number depending therefore only upon $\lambda$ and $n$, we reach
the estimate
\begin{equation}
\int |S (x,t, x_1, 0) - S (x,t, x_2, 0)|\, dx \leq \sigma^{-1} 4^{\sfrac{\alpha}{2}} \left(\frac{|x_1-x_2|}{t^{\sfrac{1}{2}}}\right)^\alpha\, .
\end{equation}
This is exactly the desired estimate, and hence the proof of Proposition~\ref{p:overlap} is finally complete.

\section{Proof of the a priori estimate}\label{s:conclusione}

First of all observe that, by Theorem~\ref{t:fund_sol}(f), \eqref{e:overlap} can also be used to prove
\begin{equation}\label{e:overlap_duale}
\int |S (x_1,t, y, \bar t) - S (x_2,t, y, \bar t)|\, dy \leq C \left(\frac{|x_1-x_2|}{(t-\bar t)^{\sfrac{1}{2}}}\right)^\alpha \qquad
\mbox{for all $t> \bar t$.}
\end{equation}
This easily gives the H\"older continuity of any solution $u$ through Theorem~\ref{t:fund_sol}(e):
\begin{align}
|u (x_1, t) - u (x_2, t)| & \leq \int |S (x_1,t, y, 0) - S (x_2,t, y, 0)| |u (y, 0)|\, dy\nonumber\\ 
& \leq C \|u\|_\infty \left(\frac{|x_1-x_2|}{t^{\sfrac{1}{2}}}\right)^\alpha\, .\label{e:space}
\end{align}
As for the time continuity, we use
\[
u (x, t) - u (x,s) = \int S (x,t,y,s) u (y,s)\, dy - u (x,s) \int S (x,t, y,s)\, dy
\]
to estimate
\begin{align}
|u (x,s) - u (x,t)| &\leq \int S (x,t, y,s) |u (y, s) - u (x,s)|\, dy\nonumber\\
& \leq \underbrace{\int_{|y-x|\leq \rho} S (x,t,y,s) |u (y, s) - u(x,s)|\, dy}_{=I_1}\nonumber\\ 
& \qquad+  \underbrace{\int_{|y-x|\geq \rho} S (x,t,y,s) |u(y, s) - u(x,s)|\, dy}_{=I_2}\, ,
\end{align}
where $\rho>0$ will be chosen later. Using \eqref{e:space} (and the fact that the integral of the fundamental solution equals $1$), we can estimate
\begin{equation}\label{e:I1}
I_1 \leq C \|u\|_\infty s^{-\sfrac{\alpha}{2}} \rho^\alpha\, .
\end{equation}
For $I_2$ we use the moment bound \eqref{e:moment_bound}:
\begin{equation}\label{e:I2}
I_2 \leq 2\rho^{-1} \|u\|_\infty \int |y-x| S (x, t,y,s)\, dy \leq C \|u\|_\infty \rho^{-1} (t-s)^{\sfrac{1}{2}}\, .
\end{equation}
We thus get
\[
|u(t,x) - u (s,x)|\leq C \|u\|_\infty \left(\rho^\alpha s^{-\sfrac{\alpha}{2}} + (t-s)^{\sfrac{1}{2}} \rho^{-1}\right)\, .
\]
Choosing $\rho^{1+\alpha} = s^{\sfrac{\alpha}{2}} (t-s)^{\sfrac{1}{2}}$ we conclude
\begin{equation}\label{e:time}
|u(t,x) - u (s,x)|\leq C \|u\|_\infty \left(\frac{t-s}{s}\right)^{\frac{\alpha}{2 (1+\alpha)}}\, .
\end{equation}
The combination of \eqref{e:space} and \eqref{e:time} gives Theorem~\ref{t:main_reg_2}.

\section{Proof of Nash's parabolic regularity theorem}\label{s:para_tecnico}

In order to conclude Theorem~\ref{t:main_reg} from Theorem~\ref{t:main_reg_2}, fix measurable coefficients $A_{ij}$ satisfying 
Assumption~\ref{a:ellipticity} and a bounded distributional solution $u$ on $\mathbb R^n \times (0, \infty)$. Without loss of generality we can assume that the $A_{ij}$ are defined also for negative times, for instance we can set $A_{ij} (x,-t) = A_{ij} (x,t)$ for every $x$ and every $t>0$. Next, we observe that, if $\varphi$ is a smooth compactly supported nonnegative convolution kernel in $\mathbb R^n \times \mathbb R$, the regularized coefficients $B^\varepsilon_{ij} = A_{ij} * \varphi_\varepsilon$ satisfy Assumption~\ref{a:ellipticity} with the same constant $\lambda$ in \eqref{e:ellipticity}. Consider moreover a cutoff function $\psi^\varepsilon$ which is nonnegative, compactly supported in $B_{2\varepsilon^{-1}} \times (-2\varepsilon, 2\varepsilon^{-1})$, identically equal to $1$ on $B_{\varepsilon^{-1}}\times (-\varepsilon^{-1} , \varepsilon^{-1})$ and never larger than $1$. If we set $A^\varepsilon_{ij} = \psi^\varepsilon B^\varepsilon_{ij} + (1-\psi^\varepsilon) \delta_{ij}$, again the matrix $A^\varepsilon$ satisfies Assumption~\ref{a:ellipticity} with the same $\lambda$ as the matrix $A$. Note also that
\begin{equation}\label{e:convergence}
\lim_{\varepsilon \to 0} \|A^\varepsilon_{ij} - A_{ij}\|_{L^1 (B_R (0) \times (-R, R))} = 0 \qquad \mbox{for every $R>0$.}
\end{equation}
We now wish to construct solutions $u^\varepsilon$ to the ``regularized'' parabolic problem
\begin{equation}\label{e:parabolic_reg}
\partial_t u^\varepsilon = {\rm div}_x (A^\varepsilon \nabla u^\varepsilon)\, ,
\end{equation}
which converge to our fixed solution $u$ of the limiting equation \eqref{e:parabolic}. In order to do so, we fix a smooth mollifier $\chi$ and a family of cut-off functions $\beta^\varepsilon$  {\em in space}. Such pair is the ``spatial analog'' of the pair $(\varphi , \psi^\varepsilon)$ used to regularize $A$. For every time $s$ we define the regularized time-slice
\[
\bar{u}^{\varepsilon,s} (x) := [u (\cdot, s) * \chi_\varepsilon] (x) \beta^\varepsilon (x)\, .
\]
By classical parabolic theory, there is a unique smooth solution $u^{\varepsilon, s}$ of 
\eqref{e:parabolic_reg} on $\mathbb R^n \times [s, \infty[$ subject to the initial condition $u^{\varepsilon, s} (\cdot ,s) = \bar{u}^{\varepsilon ,s}$: in fact this statement follows easily from
Theorem~\ref{t:fund_sol}. Moreover, by the classical maximum principle (cf.~for instance \cite{Friedman}) we have
\begin{equation}\label{e:maximum_principle}
\|u^{\varepsilon, s}\|_\infty \leq \|\bar u^{\varepsilon, s}\|_\infty \leq \|u\|_\infty\, .
\end{equation}
The key to pass from Theorem~\ref{t:main_reg_2} to Theorem~\ref{t:main_reg} is then the following lemma. 

\begin{lemma}\label{l:approx}
For almost every $s>0$, $u^{\varepsilon, s}$ converges weakly$^*$ in $L^\infty (\mathbb R^n \times (s, \infty))$ to $u$. 
\end{lemma}

We will turn to the lemma in a moment. With its aid Theorem~\ref{t:main_reg} is a trivial corollary of Theorem~\ref{t:main_reg_2} and of the estimate \eqref{e:maximum_principle}. Indeed
the solutions $u^{\varepsilon,s}$ will satisfy the uniform estimate 
\begin{equation}\label{e:Holder_uniform}
|u^{\varepsilon,s} (x_1, t_1) - u^{\varepsilon,s} (x_2, t_2)|\leq C \|u\|_\infty \left[\frac{|x_1-x_2|^\alpha}{(t_1-s)^{\sfrac{\alpha}{2}}} + 
\left(\frac{t_2-t_1}{t_1-s}\right)^{\frac{\alpha}{2 (1+\alpha)}}\right] \, ,
\end{equation}
for all $t_2\geq t_1>s >0$ and all $x_1, x_2\in \mathbb R^n$. By the Ascoli--Arzel\`a Theorem the family $u^{\varepsilon,s}$ is precompact in $C^0$, and up to subsequences will then converge uniformly to a H\"older function $u^s$ on any compact set $K\subset \mathbb R^n \times (s, \infty)$: by  Lemma~\ref{l:approx} $u^s$ will coincide with $u$ for almost every $s$ and we will thus conclude
\begin{equation}\label{e:Holder_con_s}
|u (x_1, t_1) - u (x_2, t_2)|\leq C \|u\|_\infty \left[\frac{|x_1-x_2|^\alpha}{(t_1-s)^{\sfrac{\alpha}{2}}} + 
\left(\frac{t_2-t_1}{t_1-s}\right)^{\frac{\alpha}{2 (1+\alpha)}}\right] \, .
\end{equation}
Letting now $s$ go to $0$ we achieve Theorem~\ref{t:main_reg}. 

\begin{proof}[Proof of Lemma~\ref{l:approx}] {\bf Step 1.} First we will prove that \eqref{e:distributional} can in fact be upgraded to the following stronger statement for almost every pair of times $t>s$:
\begin{align}
\int u (x,t) \varphi (x,t)\, dx & = \int_s^t \int u (x, \tau) \partial_t \varphi (x,\tau)\, dx\, d\tau
- \int_s^t \int  \partial_i \varphi (x,\tau) A_{ij} (x,\tau) \partial_j u (x,\tau)\, dx\, d\tau \nonumber\\
&\qquad +  \int u (x,s) \varphi (x,s)\, dx \qquad \forall \varphi\in C^\infty_c (\mathbb R^n \times (0, \infty))\, .\label{e:test}
\end{align}
The argument is standard, but we will include it for the reader's convenience. In particular we will prove that
\eqref{e:test} holds for every pair $s<t$ satisfying the property
\begin{equation}\label{e:approx_cont}
\lim_{\varepsilon \to 0} \frac{1}{\varepsilon} \left[\int^s_{s-\varepsilon} \int_{B_R}  |u(x, \tau) - u(x,s)|\, dx\, d\tau
+ \int^{t+\varepsilon}_t \int_{B_R} |u(x, t) - u (x, \tau)|\, dx\, d\tau \right] = 0 
\end{equation}
for all $R>0$. By standard measure theory implies, any time that we fix $R \in \mathbb N$, \eqref{e:approx_cont} holds for almost every $s<t$. 

On the other hand,
to pass from \eqref{e:distributional} to \eqref{e:test} using \eqref{e:approx_cont} we just argue with the following classical procedure:
\begin{itemize}
\item[(i)] We fix a monotone $\chi\in C^\infty (\mathbb R)$ which is identically $1$ on $]-\infty, 0]$ and identically $0$ on $]1, \infty[$. 
\item[(ii)] We test \eqref{e:distributional} with $\varphi (x,\tau) \chi (\frac{\tau - t}{\varepsilon})\chi (\frac{s-\tau}{\varepsilon})$.
\item[(iii)] We let $\varepsilon$ go to $0$. 
\end{itemize}

\medskip

{\bf Step 2.} Next, using \eqref{e:maximum_principle} and the weak$^*$ compactness of bounded sets in $L^\infty$, we can assume the convergence of $u^{\varepsilon, s}$, up to subsequences, to some $L^\infty$ function $u^s$. We wish to show that $u^s$ has first-order distributional derivatives $\partial_j u^s$ which are locally square summable. In order to do so, we borrow some ideas from
\cite{Aronson} and consider the function
\[
h (x,t) := - \frac{\alpha |x|^2}{t}\, ,
\] 
where $\alpha >0$ will be chosen in a moment.
We use the equation \eqref{e:parabolic_reg} to derive the following equality:
\begin{align}
&\int e^{h (x,t)} |u^{\varepsilon,s} (x, t)|^2 \, dx + 2 \int_s^t \int e^{h (x,\tau)} \partial_j u^{\varepsilon, s} (x, \tau) A^\varepsilon_{ij} (x, \tau)
\partial_i u^{\varepsilon,s} (x, \tau)
\, dx\, d\tau\nonumber\\
& \qquad =  \int_s^t \int e^{h (x, \tau)} \left[\partial_t h (x, \tau) |u^{\varepsilon,s} (x, \tau)|^2 - 2 u^{\varepsilon,s} (x, \tau) \partial_j u^{\varepsilon,s}
(x,\tau) A_{ij}^{\varepsilon} (x, \tau) \partial_i h (x, \tau)\right]\, dx\, d\tau\nonumber\\
&\qquad\quad + \int e^{h(x,s)} |u^{\varepsilon,s} (x,s)|^2\, dx \, . \label{e:4.102}
\end{align}
Note that, for each fixed $\varepsilon$ the solution $u^{\varepsilon,s}$ is smooth and all derivatives are bounded, by standard regularity theory for linear parabolic differential equations, see for instance \cite[Sec.~7.2.3]{Evans}. Thus all the integrals above are finite and the equality above follows from usual calculus formulae.

Now, observe that the last integral in \eqref{e:4.102} is bounded by $C \|u\|_\infty^2$ for some constant $C= C (\alpha, s)$. 
Using the ellipticity of $A^\varepsilon_{ij}$ we can thus estimate
\begin{align*}
&\int e^{h (x,t)} |u^{\varepsilon,s} (x, t)|^2 \, dx + 2 \lambda^{-1} \int_s^t \int e^{h (x, \tau)} |\nabla u^{\varepsilon ,s} (x,\tau)|^2\, dx\, d\tau\\
&\qquad \leq  \int_s^t\int e^{h(x, \tau)} \left[ \partial_t h (x, \tau) |u^{\varepsilon, s} (x, \tau)|^2+ 2 \lambda |u^{\varepsilon, s} (x, \tau)| |\nabla u^{\varepsilon, s} (x, \tau)||\nabla h (x, \tau)|\right]\, dx\, d\tau\nonumber\\
&\qquad\quad+ C \|u\|_\infty^2\, .
\end{align*}
The weight $h$ has the following fundamental property:
\begin{equation}\label{e:esponenziale}
\partial_t h = - \frac{1}{4\alpha} |\nabla h|^2\, . 
\end{equation}
Thus, it suffices to choose $\alpha$ small, depending only upon $\lambda$, to conclude, via Young's inequality, 
\begin{align*}
&\int e^{h (x,t)} |u^{\varepsilon,s} (x, t)|^2 \, dx + 2 \lambda^{-1} \int_s^t \int e^{h (x, \tau)} |\nabla u^{\varepsilon ,s} (x,\tau)|\, dx\, d\tau\\
&\qquad \leq \lambda \int_s^t \int e^{h (x, \tau)} |\nabla u^{\varepsilon ,s} (x,\tau)|^2\, dx\, d\tau + C \|u\|_\infty^2\, .
\end{align*}
The latter inequality gives an upper bound on 
\[
\int_s^t \int e^{h (x, \tau)} |\nabla u^{\varepsilon ,s} (x,\tau)|^2\, dx\, d\tau
\]
which depends upon $\|u\|_\infty$ and $\lambda$, but not upon $\varepsilon$. We thus infer a uniform bound
for $\|\nabla u^{\varepsilon, s}\|_{L^2 (B_R (0)\times (s, \infty))}$ for every positive $R$. In turn such bound implies that the partial derivatives $\partial_j u^s$ are locally square summable and that $\partial_j u^{\varepsilon, s}$ converge (locally) weakly in $L^2$ to $\partial_j u^s$
(again up to subsequences, which we do not label for notational convenience).  

\medskip

{\bf Step 3.} Passing to the limit in the weak formulation of \eqref{e:parabolic_reg} and using that the initial data $u^{\varepsilon, s} (\cdot, s)$ converges (locally in $L^1$) to $u (\cdot, s)$, we then infer the corresponding of \eqref{e:test} for every $t>s$
(in this case we need no restriction upon $t$ because we know that $u^s$ converges locally uniformly!), namely, the
validity of
\begin{align}
\int u^s (x,t) \varphi (x,t)\, dx = & \int u^s (x,s) \varphi (x,s)\, dx + \int_s^t \int u^s  (x, \tau) \partial_\tau \varphi (x,\tau)\, dx\, d\tau\nonumber\\
& - \int_s^t \int  \partial_i \varphi (x,\tau) A_{ij} (x,\tau) \partial_j u^s (x,\tau)\, dx\, d\tau \label{e:test2}
\end{align}
for every test function $\varphi\in C^\infty_c (\mathbb R^n \times (0, \infty))$.
If we consider $w := u - u^s$ we then subtract \eqref{e:test2} from \eqref{e:test} to conclude the following identity for almost every pair $t\geq s$ and for every test $\varphi\in C^\infty_c (\mathbb R^n \times (0, \infty))$:
\begin{align}
\int w (x,t) \varphi (x,t)\, dx = & \int_s^t \int w  (x, \tau) \partial_\tau \varphi (x,\tau)\, dx\, d\tau\nonumber\\
& - \int_s^t \int  \partial_i \varphi (x,\tau) A_{ij} (x,\tau) \partial_j w (x,\tau)\, dx\, d\tau \, .\label{e:test3}
\end{align}
Our goal is to use the latter integral identity, which is a weak form of \eqref{e:parabolic} with initial data $w (\cdot, s) = 0$, to derive
that $w=0$ almost everywhere: this would imply that $u=u^s$ almost everywhere and thus complete the proof of the lemma.

\medskip

{\bf Step 4.} In order to carry on the above program we wish to test \eqref{e:test3} with $\varphi = e^h w$, but we must face two difficulties:
\begin{itemize}
\item[(i)] $w$ is not smooth enough. Indeed the first-order partial derivatives in space are locally square summable and pose no big difficulties, but note that in \eqref{e:test3} there is a term with a partial derivative in time, which for $e^h w$ is not even a summable function.  
\item[(ii)] $e^{h} w$ is not compactly supported in space (the assumption of being compactly supported in time can be ignored, since all domains of integration are bounded in time). 
\end{itemize}
In order to remove these two problems we fix a cutoff function $\chi\in C^\infty_c (\mathbb R^n)$ and a compactly supported smooth kernel {\em in space only}, namely, a nonnegative $\gamma \in C^\infty_c (\mathbb R^n)$ with integral $1$. We then consider the spatial regularization
\[
w * \gamma_\varepsilon (x,\tau) = \int w (y,\tau) \gamma \left(\frac{x-y}{\varepsilon}\right)\, dy\, ,
\]
and define the test function $\varphi := \chi^2 e^{h} w * \gamma_\varepsilon$.  The map $x\mapsto w *\gamma_\varepsilon (x, t)$ is smooth for every fixed $t$ and moreover $\|\nabla (w * \gamma _\varepsilon) (\cdot ,t)\|_\infty \leq C \|w\|_\infty \varepsilon^{-1}$. To gain regularity in time we can use the weak form of the equation to show that, in the sense of distributions,
\begin{equation}\label{e:time_der_dist}
\partial_t (w * \gamma_\varepsilon) = ({\rm div}_x (A \nabla w)) *\gamma_\varepsilon = (A_{ij} \partial_j w) * \partial_i \gamma_\varepsilon\, .
\end{equation}
Since $\partial_t w$ is locally square summable, we conclude that $\partial_t (w*\gamma_\varepsilon)$ is a locally bounded measurable function
and thus that $w*\gamma_\varepsilon$ is locally Lipschitz in the space-time domain $\mathbb R^n \times (0, \infty)$. Hence the test function $\varphi := \chi^2 e^{h} w * \gamma_\varepsilon$ is Lipschitz and compactly supported and, although the test function in our definition of distributional solution is assumed to be smooth, it is easy check that, nonetheless, \eqref{e:test3} holds for our (possibly less regular) choice. 
Inserting such $\varphi$ in \eqref{e:test3}, and using \eqref{e:time_der_dist}, we then achieve
\begin{align*}
& \int e^{h (x,t)} w (x, t) w*\gamma_\varepsilon (x,t) \chi^2 (x)\, dx\nonumber\\
&\qquad =  \int_s^t \int e^{h(x, \tau)} \partial_t h (x, \tau) w (x, \tau) w*\gamma_\varepsilon (x, \tau) \chi^2 (x)\, dx\, d\tau\\
&\qquad\quad+ \underbrace{\int_s^t \int e^{h (x, \tau)} w (x, \tau) [(A_{ij} \partial_j w) * \partial_i \gamma_\varepsilon] (x, \tau) \chi^2 (x)\, dx\, d\tau}_{=:(I)}\\
&\qquad \quad- \int_s^t \int e^{h (x, \tau)} \partial_i w (x, \tau) A_{ij} (x, \tau) \chi (x)\cdot \\
&\qquad\qquad\cdot [\partial_j w*\gamma_\varepsilon (x, \tau) \chi (x) +
w*\gamma_\varepsilon (x, \tau) (\partial_j h (x, \tau) \chi (x) + 2 \partial_j \chi (x))]\, dx\, d\tau\, .  
\end{align*}
Next, assuming that $\gamma$ is a symmetric kernel, we can use the standard identity 
\[
\int (f*\gamma) (x) g (x)\, dx= \int f (x) (g*\gamma) (x)\, dx
\] 
to conclude
\[
(I) = -  \int_s^t \int e^{h (x, \tau)} \partial_j w (x, \tau) A_{ij} (x, \tau) 
[(\chi^2 \partial_i w + \chi^2 w \partial_i h + 2 w \chi \partial_i \chi) * \gamma_\varepsilon] (x, \tau)\, dx\, d\tau\, .
\]
Letting $\varepsilon$ go to $0$ we then conclude
\begin{align*}
& \int e^{h(x, t)}  w^2 (x, t) \chi^2 (x)\, dx \\
& \qquad = - 2 \int_s^t \int e^{h(x, \tau)} \chi^2 (x) \partial_i w (x, \tau) A_{ij} (x, \tau) \partial_j w (x, \tau)\, dx\, d\tau\\
&\qquad\quad + \int_s^t \int e^{h (x, \tau)} \chi^2 (x) w^2 (x, \tau) \partial_t h (x, \tau)\, dx\, d\tau\\
&\qquad\quad - 2 \int_s^t \int e^{h (x, \tau)} w(x, \tau)  \chi (x)  \partial_i w (x, \tau) A_{ij} (x, \tau) (2 \partial_j \chi  (x) + \chi (x) \partial_j h (x, \tau))\, dx\, d\tau\, .
\end{align*}
Using now the ellipticity of $A_{ij}$ and \eqref{e:esponenziale} we achieve
\begin{align*}
& \int e^{h(x, t)}  w^2 (x, t) \chi^2 (x)\, dx\nonumber\\
&\qquad \leq - 2\lambda^{-1} \int_s^t \int e^{h(x, \tau)} \chi^2 (x) |\nabla w (x, \tau)|^2\, dx\, d\tau\\
& \qquad\quad - (4\alpha)^{-1} \int_s^t \int e^{h (x, \tau)} \chi^2 (x) w^2 (x, \tau) |\nabla h (x, \tau)|^2\, dx\, d\tau\\
& \qquad\quad + 2 \lambda \int_s^t \int e^{h (x, \tau)} |w (x, \tau)| |\nabla w (x, \tau)| (\chi^2 (x) |\nabla h (x, \tau)| + 2 |\chi (x)||\nabla \chi (x)|)\, dx\, d\tau\, .
\end{align*}
From the latter we recover, using Young's inequality, 
\begin{align*}
& \int e^{h(x, t)} w^2 (x, t) \chi^2 (x)\, dx\\
&\qquad \leq 
- (4 \alpha)^{-1} \int_s^t \int e^{h (x, \tau)} \chi^2 (x) w^2 (x, \tau) |\nabla h (x, \tau)|^2\, dx\, d\tau\\
&\qquad\quad + C (\lambda) \int_s^t \int e^{h (x, \tau)} \chi^2 (x) w^2 (x, \tau) |\nabla h (x, \tau)|^2\, dx\, d\tau\\
&\qquad\quad  + C (\lambda) \int_s^t \int e^{h(x, \tau)} w^2 (x, \tau) |\nabla \chi (x)|^2\, dx\, d\tau\, ,
\end{align*}
where $C (\lambda)$ is a constant which only depends on $\lambda$. Hence, choosing $\alpha$ sufficiently small, depending only on
$\lambda$, we conclude
\begin{align}\label{e:energia_Ar}
\int e^{h (x,t)} w^2 (x, t) \chi^2 (x)\, dx &\leq C (\lambda) \int_s^t \int e^{h(x, \tau)} w^2 (x, \tau) |\nabla \chi (x)|^2\, dx\, d\tau\, .
\end{align}
Next, consider a cut-off function $\beta \in C^\infty_c (B_2)$ which is identically $1$ on $B_1$ and, for any $R>0$, set
$\chi (x) := \beta (\frac{x}{R})$. Insert the latter  in \eqref{e:energia_Ar}. Using that $|\nabla \chi (x)|\leq C R^{-1}$ and the fact that $e^h w^2$ is integrable, when we let $R\uparrow \infty$ we conclude
\[
\int e^{h(x,t)} w^2 (x,t)\, dx \leq 0\, .
\]
This implies that $w (\cdot, t)\equiv 0$ for almost every $t\geq s$ and thus concludes the proof.  
\end{proof}

\section{Proof of the De Giorgi--Nash theorem}

By standard Sobolev space theory, cf.~\cite[Sec.~7.2\&7.3]{Evans}, $v|_{B_{3r} (z)}$ is the unique minimum of the energy functional
\begin{equation}
\mathscr{E} (w) := \int_{B_{3r} (z)} \partial_i w (x) A_{ij} (x) \partial_j w (x)\, dx 
\end{equation}
among those functions $w\in W^{1,2} (B_{3r} (z))$ such that $w-v \in W^{1,2}_0 (B_{3r} (z))$. If we first extend $A$ and $v$ and we then regularize them by convolution to $A^\varepsilon$ and $v^\varepsilon$, we can consider the corresponding solutions of the regularized elliptic equations, using the same arguments of the last section (a proof of the regularity of the solution can be found, for instance, in \cite[Sec.~6.3]{Evans}. By the maximum principle (cf.~again \cite[Sec.~6.4]{Evans}), we will have $\|v^\varepsilon\|_\infty \leq \|v\|_\infty$ and $v^\varepsilon$ will be a minimizer of the corresponding regularized energy functional.
Since $\|v^\varepsilon\|_{W^{1,2} (B_{3r} (z))}$ would be uniformly bounded, we can assume, after extraction of a convergent subsequence, that $v^\varepsilon$ converges weakly in $W^{1,2} (B_{3r} (z))$ to some $\bar v$, which in turn is a distributional solution 
of \eqref{e:elliptic} subject to the constraint $\bar v - v \in W^{1,2}_0 (B_{3r} (z))$. As such, $\bar v$ must be a minimizer of the same variational problem as $v|_{B_{3r} (z)}$, which we already know to be unique. Thus $\bar v = v|_{B_{3r} (z)}$ and so it suffices to prove Theorem~\ref{t:DG-Nash} under the a priori assumption that $A$ and $v$ are smooth. 
Moreover, by rescaling $v$ to $\bar v (x) := v (rx +z)$, we can assume that $r=1$ and $z=0$.

\medskip

Under these additional assumptions, we can consider $v (x,t) := v(x)$ as a stationary smooth solution of the parabolic problem
\begin{equation}\label{e:para_ell}
\partial_t u (x,t) = \partial_j (A_{ij} (x) \partial_i u (x,t))
\end{equation}
on $C_3 := B_3 \times (0, \infty)$. Theorem~\ref{t:DG-Nash} is then a simple corollary of Theorem~\ref{t:main_reg_2}
and the following proposition, which is a direct outcome of the theory developed by Nash.

\begin{proposition}[$L^\infty$ estimate for the initial-boundary value problem]\label{p:stima_2_sol}
There is a constant $C$ depending only upon $n$ and $\lambda$ with the following property. Assume that $A_{ij} (x,t)$ satisfies the assumptions of
Theorem~\ref{t:main_reg_2} and $w: \overline B_2 \times [0, \infty[\to \mathbb R$ is a smooth bounded solution of \eqref{e:parabolic} with $w (x,0)=0$ for every $x$. Then
\begin{equation}\label{e:bordo_lontano}
\| w(\cdot, t)\|_{L^\infty (B_1)} \leq C \|w\|_\infty t^{\sfrac{1}{2}}\, .
\end{equation}
\end{proposition}

With Proposition~\ref{p:stima_2_sol} at hand, it is easy to conclude Theorem~\ref{t:DG-Nash}. Indeed, multiply $v$ by a smooth cut-off function $\varphi\in C^\infty_c (B_3)$ taking values in $[0,1]$ and identically $1$ on $B_2$. Extend $\varphi v$ smoothly on $\mathbb R^n$ by setting it equal to $0$ on $\mathbb R^n \setminus B_3$. Let $z$ be the solution on $\mathbb R^n \times \mathbb [0, \infty)$ of 
\eqref{e:para_ell} with $z (\cdot, 0) = \varphi v$. Note that $\|z\|_\infty \leq \|v\|_\infty$ by the maximum principle. We can apply 
Proposition~\ref{p:stima_2_sol} to $w (\cdot, t):= z (\cdot, t) -  v (\cdot)$ to conclude 
\begin{equation}\label{e:quasi_ell_1}
\|v - z(\cdot, t)\|_{L^\infty (B_1)} \leq 2 C\|v\|_\infty t^{\sfrac{1}{2}}\, .
\end{equation}
On the other hand, by Theorem~\ref{t:main_reg_2} we have
\begin{equation}\label{e:quasi_ell_2}
|z (x_1, t) - z (x_2, t)| \leq C \|v\|_\infty \frac{|x_1-x_2|^\alpha}{t^{\sfrac{\alpha}{2}}}\, .
\end{equation}
In particular, for $x_1, x_2\in B_1$, we can combine the last two inequalities to conclude
\begin{equation}\label{e:ottimizza_2}
|v (x_1) - v (x_2)| \leq C \|v\|_\infty \left( t^{\sfrac{1}{2}} +  \frac{|x_1-x_2|^\alpha}{t^{\sfrac{\alpha}{2}}}\right) \qquad \forall t>0\, .
\end{equation}
Choose now $t^{\sfrac{1}{2} + \sfrac{\alpha}{2}} = |x_1-x_2|^\alpha$ to conclude that
\begin{equation}\label{e:fatto}
|v (x_1) - v (x_2)|\leq C \|v\|_\infty |x_1-x_2|^{\alpha / (1+\alpha)}\, .
\end{equation}
So, to complete the proof of Theorem~\ref{t:main_reg_2} we only need to show Proposition~\ref{p:stima_2_sol}.

\begin{proof}[Proof of Proposition~\ref{p:stima_2_sol}] Consider any smooth solution $u$ of \eqref{e:parabolic} in 
$C_2 := \bar B_2\times [0, \infty[$.
The boundary values on $\partial C_2$ determine then the solution through a representation formula of the form
\[
u (x, t) := \int_{\partial C_2} u(\xi) \rho (x,t, \xi)\, d\xi\, ,
\]
where the integral is taken with respect to the standard surface measure on the boundary $\partial C_2$, cf.~\cite[Sec.~1.4]{Friedman}. If we set $\xi = (y (\xi), \tau (\xi))$, then the kernel $\rho (x,t, \xi)$ satisfies the conditions
\begin{itemize}
\item[(i)] $\int \rho (x,t, \xi)\, d\xi=1$;
\item[(ii)] $\rho (x,t, \xi)\geq 0$;
\item[(iii)] $\rho (x,t,\xi) =0$ if $t\leq \tau (\xi)$.
\end{itemize}
Since the fundamental solutions $S(x,t, x_0, t_0)$ with $t_0<0$ are also smooth solutions of the parabolic equation in the cylinder $C_2$, we reach the identity
\begin{equation}
S (x,t, x_0, t_0) = \int_{\partial C_2} S (y (\xi), \tau (\xi), x_0, t_0) \rho (x,t, \xi)\, d\xi\, .
\end{equation}
Multiplying by $|x-x_0|$ and integrating we then have
\begin{equation}
\int |x-x_0| S (x,t, x_0, t_0)\, dx_0 = \int\int_{\partial C_2} |x-x_0| S (y (\xi), \tau (\xi), x_0, t_0) \rho (x,t, \xi)\, d\xi\, dx_0\, .
\end{equation}
In particular, using the moment bound \eqref{e:moment_bound} we conclude
\begin{equation}\label{e:mom_ell}
\int \int_{\partial C_2} (|x-y (\xi)| - |x_0 - y (\xi)|)  S (y (\xi), \tau (\xi), x_0, t_0) \rho (x,t,\xi)\, d\xi\, dx_0\leq C_4 (t-t_0)^{\sfrac{1}{2}}\, .
\end{equation}
From the latter inequality, using again the moment bound, we achieve
\begin{align}
& \int_{\partial C_2} |x- y (\xi)|\rho (x,t, \xi)\, d\xi\nonumber\\
&\qquad =  \int \int_{\partial C_2} |x-y (\xi)| S (y (\xi), \tau (\xi), x_0, t_0) \rho (x,t,\xi)\, d\xi\, dx_0\nonumber\\
&\qquad \leq   C_4 (t-t_0)^{\sfrac{1}{2}} + \int_{\partial C_2} \int |x_0 - y (\xi)|S (y (\xi), \tau (\xi), x_0, t_0) \, dx_0\, \rho (x,t, \xi)\, d\xi\nonumber\\
&\quad\;\;\stackrel{\eqref{e:moment_bound}}{\leq} 2 C_4 (t-t_0)^{\sfrac{1}{2}}\, .\label{e:mom_ell_2}
\end{align}
Letting $t_0$ go to $0$, we thus conclude
\begin{equation}\label{e:mom_ell_3}
\int_{\partial C_2} |x- y (\xi)|\rho (\xi, x,t)\, d\xi \leq 2 C_4 t^{\sfrac{1}{2}}\, .
\end{equation}
Let $\mathcal{L}: = \partial C_2 \setminus B_2 (0)\times \{0\}$ and observe that $|x-y (\xi)|\geq 2-|x|$ if $\xi\in \mathcal{L}$. Thus, using \eqref{e:mom_ell_3} and the fact that $\rho\geq 0$, we conclude
\begin{equation}\label{e:norm_ell_4}
2 C_4 t^{\sfrac{1}{2}} \geq \int_{\mathcal{L}} |x- y (\xi)|\, \rho (x, t, \xi)\, d\xi \geq (2 - |x|) \int_{\mathcal{L}} \rho (x,t, \xi)\, d\xi\, .
\end{equation}
Consider now a solution $w$ as in the proposition. Since $w =0$ on $B_2 (0) \times \{0\}$, for any $(x,t)$ we have
\begin{equation}
|w (x,t)|\leq \int_{\mathcal{L}} \rho (x,t,\xi) |w (\xi)|\, d\xi \stackrel{\eqref{e:norm_ell_4}}{\leq} \frac{C t^{\sfrac{1}{2}}}{2-|x|} \|w\|_\infty\, .
\end{equation}
The latter inequality for $x\in B_1 (0)$ obviously implies \eqref{e:bordo_lontano}.
\end{proof}

\chapter{The other papers in pure mathematics}

\section{A path space and Stiefel--Whitney classes} In 1955 Whitney communicated to the Proceedings of the National Academy of Sciences a two pages note of Nash, \cite{Nash1955}, where he gives a very direct proof of the topological invariance of the Stiefel--Whitney classes of smooth manifolds, a theorem proved three years before by Thom (cf.~\cite{Thom}). For the definition of Stiefel--Whitney classes of a smooth vector bundle we refer to \cite{MilSta}: given a differentiable manifold its Stiefel--Whitney classes are then the corresponding classes of the tangent bundle and the theorem of Thom shows that such classes are a topological invariant. In fact, Thom derived this consequence from a stronger theorem, namely that the homotopy type of a tangent bundle as fiber space over a topological manifold $M$ is the same for any differentiable structure on $M$. Nash shows that this conclusion can be inferred from the definition of an appropriate path space $X$ of the topological manifold $M$, where, loosely speaking, the tangent bundles can be embedded.

\begin{definition}\label{d:path_space}
Given a topological manifold $M$, $X$ is the space of continuous mappings $\gamma: [0,1]\to M$ which do not ``recross'' the starting point $\gamma (0)$. $X$ is endowed with the topology induced by uniform convergence and with a natural projection map $\pi: X \to M$ defined by $\pi (\gamma) := \gamma (0)$. 
\end{definition}

Given a differentiable structure on $M$, we can define on its tangent bundle a smooth Riemann tensor $g$ and use it to ``embed the tangent bundle in $X$'' (more precisely, we will embed the sphere bundle in $X$, see below). To this aim, first of all we
assume, by suitably modifying $g$, that
\begin{itemize}
\item[(I)] any pair of points in the Riemannian manifold $(M,g)$ with geodesic distance no larger than $1$ can be joined by a unique geodesic segment of length $1$.
\end{itemize} 
Hence we can consider the subset $G$ of $X$ consisting of those paths which are geodesic segments with length $1$ parametrized with arc-length. Of course, the sphere bundle on $M$ given by the tangent vectors $v$ with unit length 
is isomorphic to $G$ as fiber bundle over $M$. Nash's main observation can then be stated as follows. 

\begin{theorem}\label{t:SW_classes}
If (I) holds, then 
$G$ is a fiber deformation retract of $X$, i.e., there is a continuous map $\Phi : [0,1]\times X \to X$ such that 
\begin{itemize}
\item[(a)] $\Phi (0, \gamma) = \gamma$ for every $\gamma \in X$;
\item[(b)] $\Phi (1, \gamma) \in G$ for every $\gamma \in X$;
\item[(c)] $\Phi (1, \eta) = \eta$ for every $\eta \in G$;
\item[(d)] $\pi (\Phi (s, \gamma)) = \pi (\gamma)$ for every $\gamma \in X$ and every $s\in [0,1]$.  
\end{itemize}
\end{theorem}

The proof, which Nash sketches very briefly, is an elementary exercise.

\section{Le probl\`eme de Cauchy pour les \'equations\\ diff\'erentielles d'un fluide g\'en\'eral}

In 1962, four years after his last masterpiece on the continuity of solutions to parabolic equations, Nash published a twelve pages paper in French, whose aim was to prove the short-time existence of smooth solutions to the compressible 
Navier--Stokes equations for a viscous heat-conducting fluid. More precisely he considers the following system of five partial differential equations, in the unknowns $\rho, v$ and $T$ which represent, respectively, the density, the velocity and the temperature of the fluid and are therefore functions of the time $t$ and the space $x\in \mathbb R^3$:
\begin{equation}\label{e:comp_NS}
\left\{
\begin{array}{l}
\partial_t \rho + {\rm div}_x\, (\rho v) = 0\, ,\\ \\
\rho \partial_t v_i + \rho \left[v_j \partial_j v_i\right] + \partial_i p = \partial_j \sigma_{ij} + \rho F_i\, , \\ \\
\partial_t T + v_j \partial_j T = \frac{1}{\rho T S_T} \left[{\rm div}\, (\varkappa \nabla T) + \rho^2 T S_\rho {\rm div}\, v\right] + \frac{2\eta}{\rho T S_T} \mathcal{S} (v)_{ij} \mathcal{S} (v)_{ij} + \frac{\zeta}{\rho T S_T}  ({\rm div}\, v)^2\, .
\end{array}
\right.
\end{equation}
In the system above\footnote{The first two equations are the first two equations from \cite[p.~487, (1)]{Nash1962} whereas the third should correspond to \cite[p.~488, (1c)]{Nash1962}. The latter is derived by Nash from the third equation in 
\cite[p.~487, (1)]{Nash1962}, which in turn corresponds to the classical conservation law for the entropy, see, for instance, \cite[(49.5)]{Landau}. The third equation of \cite[p.~487, (1)]{Nash1962} contains two typos, which disappear in \cite[p.~488, (1c)]{Nash1962}. The latter however contains another error: Nash has $\eta$ and $\zeta$ in place of  $\frac{\eta}{\rho T S_T}$ and $\frac{\zeta}{\rho T S_T}$, but it is easy to see that this would not be consistent with the way he describes its derivation.  

Nash's error has no real consequence for the rest of the note, since he treats the coefficients in front of $\mathcal{S} (v)_{ij} \mathcal{S} (v)_{ij}$
and  $({\rm div}\, v)^2$ as arbitrary real analytic functions of $\rho$ and $T$ and the same holds for $\frac{\eta}{\rho T S_T}$ and $\frac{\zeta}{\rho T S_T}$ under the assumption $S_T \neq 0$. The latter inequality is needed in any case even to treat Nash's ``wrong'' equation for $T$.}:
\begin{itemize}
\item[(i)] We use Einstein's convention on repeated indices;
\item[(ii)] The pressure $p$ is a function of the density $\rho$ and the temperature $T$;
\item[(iii)] $\sigma_{ij}$ is the Cauchy stress tensor, given by the formula
\begin{equation}
\sigma_{ij} = \eta \left(\partial_i v_j + \partial_j v_i\right) + \left(\zeta - {\textstyle{\frac{2}{3}}} \eta\right) {\rm div}\, v\; \delta_{ij} \, ,
\end{equation}
with $\eta$ and $\zeta$ (the viscosity coefficients) which are functions of $\rho$ and $T$;
\item[(iv)] $F= (F_1, F_2, F_3)$ is the external force acting on the fluid;
\item[(v)] $\varkappa$, the heat conductivity, is a function of the temperature $T$ and the density $\rho$;
\item[(vi)] The entropy $S$ is a function of $\rho$ and $T$, whereas $S_T$ and $S_\rho$ are the corresponding partial derivatives with respect to $T$ and $\rho$;
\item[(vii)] $\mathcal{S} (v)$ is the traceless part of the symmetrized derivative of $v$, more precisely
\begin{equation}
\mathcal{S} (v)_{ij} = \frac{1}{2} \left[ \partial_i v_j + \partial_j v_i - {\textstyle{\frac{2}{3}}} {\rm div}\, v\; \delta_{ij}\right]\, .
\end{equation}
\end{itemize}
The functions $\eta, \zeta, \varkappa, S$ and $p$ are thus known and determined by the thermodynamical properties of the fluid. They display a rather general behavior, although they must obey some restrictions: we refer to the classical textbook \cite{Landau} for their physical meaning and for the derivation of the equations. In his paper Nash assumes that all the functions $\eta, \zeta, \varkappa, p, S$ and $S_T$ are real analytic and positive. \footnote{Indeed Nash does not mention the positivity of $S_T$, although this is is certainly required by his argument when he reduces the existence of solutions of \eqref{e:comp_NS} to the existence of a solutions of a suitable parabolic system, cf.~\cite[(6) and (7)]{Nash1962}: the equation in $T$ is parabolic if and only if $\frac{\varkappa}{\rho T S_T}$ is positive.

I also have the impression that his argument does not really need the positivity of $S$ and $p$, although these are quite natural assumptions from the thermodynamical point of view.}

Similarly, the external force $F$ is given. 
Nash considers then the Cauchy problem for \eqref{e:comp_NS} in the whole threedimensional space, namely he assumes
that the density, the velocity and the pressure are known at a certain time, which without loss of generality we can assume to be the time $0$. This problem has received a lot of attention in the last 30 years and we refer to the books \cite{Feireisl,Lions} for an account of the latest developments in the mathematical treatment of \eqref{e:comp_NS}. 

\medskip

In order to give his existence result, Nash first passes to the Lagrangian formulation of \eqref{e:comp_NS} and he then eliminates the density $\rho$. Subsequently he shows the existence, for a finite time, of a (sufficiently) smooth solution of the resulting system of equations under the assumption that the initial data and the external force are (sufficiently) smooth. In particular, he writes the system as a second-order parabolic linear system
of partial differential equations with variable coefficients, where the latter depend upon the unknowns (it must be noted that such dependence involves first-order spatial derivatives of the unknowns and their time integrals). The existence result is therefore achieved through a fixed point argument, taking advantage of classical estimates for second-order linear parabolic systems. 

\section{Analyticity of the solutions of implicit\\ function problems with analytic data}

In 1966 Nash turned again one last time to the isometric embedding problem, addressing the real analytic case. More precisely, his aim was to prove that, if in Theorem~\ref{t:Ck_1} we assume that the metric $g$ is real analytic, then there is a real analytic isometric embedding of $(\Sigma, g)$ in a sufficiently large Euclidean space. The most important obstacle in extending the proof of \cite{Nash1956} to the real analytic case is the existence of a suitable smoothing operator which replaces the one in Section \ref{s:smoothing} in the real analytic context. 

In his twelve pages paper Nash gives indeed two solutions to the problem. Most of the paper is devoted to prove the existence of a suitable (real) analytic smoothing operator
on a general compact real analytic manifold. But he also remarks that the real analytic case of the isometric embedding problem for compact Riemannian manifolds $\Sigma$ can be reduced to the existence of real analytic isometric embeddings for real analytic Riemannian manifolds which are tori, at the price of enlarging the dimension of the Euclidean target: it simply suffices to take a real analytic immersion of $\Sigma$ into $\mathbb T^{2n+1}$ using Whitney's theorem and then to extend the real analytic Riemannian metric $g$ on $\Sigma$ to the whole torus (a problem which can be solved using Cartan's work \cite{Cartan1957}). 
On the other hand the existence of a suitable regularizing analytic operator on the torus is an elementary consequence of the Fourier series expansion. 

Nash leaves the existence of real analytic embeddings for noncompact real analytic Riemannian manifolds as open and it points out that ``\ldots The case of non-compact manifolds seem to call for a non-trivial generalization of the methods''. The noncompact case was indeed settled later by Gromov (cf.~\cite{GromovBook}).

\section{Arc structure of singularities}

In 1968 Nash wrote his last paper in pure mathematics. Although it was published 28 years later (see \cite{Nash1996}), its content was promoted by Hironaka and later by Lejeune-Jalabert (cf.~\cite{Jalabert}): thus the content of Nash's work became known very much before it was finally published. Nash's idea is to use the space of complex analytic arcs in a complex algebraic variety as a tool to study its singularities and in particular their resolutions (whose existence had been established only four years before Nash's paper in the celebrated work of Hironaka, \cite{Hironaka}). In his paper he formulated a question which became known in algebraic  geometry as Nash's problem. A complete solution of the problem has not yet been reached although a lot of progress has been made in recent years (we refer the reader to the very recent survey \cite{deFernex}). 

\medskip

Nash's problem (and his ideas) are nowadays formulated for varieties (in fact, schemes) on a general algebraically closed field of any characteristic. However \cite{Nash1996} is concerned with complex varieties and in this brief description we will stick to the latter case. Take therefore a complex variety $V$. The space $X$ of arcs in $V$ is then given by the jets of holomorphic maps $x: \Omega \to V$ where $\Omega$ is an arbitrary open
subset of $\mathbb C$ containing the origin.\footnote{In the modern literature it is customary to take an equivalent definition
of $X$ through formal power series; we refer to \cite{KollarNemethi} for the latter
and for several important subtleties related to variants of the Nash arc space.}
An interesting case is that where
$W= V_s$ is the set of singularities of $V$: $X(V_s)$ consists of those arcs which ``pass through'' a singularity. In \cite{Nash1996} Nash realized that this space has, roughly speaking, the structure of an ``infinite dimensional complex variety'' (for a precise formulation we refer to \cite[p.~32]{Nash1996} or to \cite[Th.~2.6]{deFernex}; see also the earlier work of Greenberg \cite{Greenberg})
which has finitely many irreducible components, cf.~\cite[Prop.~1]{Nash1996}. Nash calls such components {\em arc families}. 

The main idea of Nash is to establish a relation between the arc families of $X(V_s)$ and the irreducible components of the image of $V_s$ through a resolution of the singularities of $V$. More precisely, having fixed a resolution of the singularities $V^* \to V$ (namely a {\em smooth} algebraic variety $V^*$ together with a proper birational map $V^*\to V$), we can look at the components $W^*_1, \ldots\, W^*_L$ of the image $W^*$ of $V_s$ in $V^*$. Nash lifts almost every arc in $X(V_s)$ to a unique arc of $X (W^*)$ and through this procedure establishes the existence of an injective map from the arc families of $X (V_s)$ to the components of $W^*$, cf.~\cite[Prop.~2]{Nash1996}\footnote{In fact, Nash claims the proposition with {\em any} algebraic subset $W$ of $V$ in place of $V_s$ but, although the proposition does hold for $W=V_s$, it turns out to be false for a general algebraic subset $W$; cf.~\cite[Ex.~3.7]{deFernex} for a simple explicit counterexample.}.  As a corollary, given two different resolutions $V^*$ and $V^{**}$, and the corresponding components $W^*_1, \ldots, W^*_{L^*}$, $W^{**}_1, \ldots, W^{**}_{L^{**}}$ of the preimage of $V_s$ in $V^*$ and $V^{**}$, Nash establishes the existence of a birational correspondence $W_j^* \to W^{**}_k$ between those pairs which correspond to the same arc family (cf.~\cite[Cor., p.~38]{Nash1996}). 

As a consequence of his considerations, such components are {\em essential}, i.e., they must appear in any resolution of the singularities of $V$. He then raised the question whether all essential components must correspond to an arc family: this is what algebraic geometers call, nowadays, Nash's problem. In high dimension the answer is known to be negative since the work \cite{IK} and it has been shown very recently that in fact the answer is negative already for some threedimensional varieties, cf.~\cite{dF2013,JK}. It must be noticed that Nash was indeed rather careful with the higher dimensional case of his question: quoting \cite[p.~31]{Nash1996} ``\ldots We do not know how complete is the representation of essential components by arc families''. However in the twodimensional case, i.e. the case of algebraic surfaces, it is a classical fact that there is a unique minimal resolution, namely containing only essential components, and Nash conjectured that each essential component is indeed
related to an arc family.  The conjecture has been  proved only recently in \cite{FdBP}. 

\medskip

Nonetheless the studies on Nash's problem are very far from being exhausted. Indeed the answer has been proved to be affirmative in a variety of interesting cases (see the survey articles \cite{deFernex,PS15}) and several mathematicians are looking for the ``correct formulation'' of the question (see, for instance, \cite{JK}), possibly leading to a complete understanding of the relations between resolutions of the singularities and the arc space. 

\section{The Nash blow-up}

In algebraic geometry the term ``Nash blow-up'' refers to a procedure with which, roughly speaking, the singular points of an algebraic variety are replaced by all the limits of the tangent spaces to the regular points. If $X$ is an algebraic subvariety of $\mathbb C^n$ of pure dimension $r$, the Nash blow-up is then the (closure of the) graph of the Gauss map: more precisely, if we denote by $\mathbf{Gr}\, (r,n)$ the Grassmanian of $r$-dimensional complex linear subspaces of $\mathbb C^n$, then the Nash blow-up of $X$ is the closure of the set of pairs $(x, T_x X)\in \mathbb C^n \times \mathbf{Gr}\, (r, n)$, where $x$ varies among all regular points of $X$ and $T_x X$ denotes the tangent space to $X$ at $x$. Although such definition is given in terms of the embedding, it can be shown that in fact the Nash blow-up of $X$ depends only upon $X$.

A long standing open problem is whether after a finite number of Nash blow-ups every singular variety becomes smooth (indeed, in characteristic $p$ the answer is negative and one needs to state the problem in terms of ``normalized Nash blow-ups'', cf.~\cite{Nobile}). According to \cite{Spivakovsky}, such question was posed by Nash to Hironaka in a private communication in the early sixties and the term ``Nash blow-up'' was first used by Nobile a decade later in \cite{Nobile}, where he proved that the answer to Nash's question is affirmative for curves in characteristic $0$. Building upon the work of Hironaka \cite{Hironaka}, Spivakovsky proved in the late eighties that the answer is affirmative for surfaces in characteristic $0$ for the normalized Nash blow-up (cf.~\cite{Spivakovsky}). In general the question of Nash is still widely open and constitutes an active area of research.

Curiously, it seems that the problem was posed a few years before Nash by Semple in \cite{Semple} and for this reason some authors have used recently the term Semple-Nash blow-up, cf.~for instance \cite{TeisTrang}. 

\bibliographystyle{plain}

\end{document}